%% file: main.tex
\documentclass[12pt,reqno]{amsart}
\usepackage{setspace}
\usepackage{comment}
\usepackage{pdfpages}
\usepackage{fancyhdr}
\usepackage{amssymb}   
\usepackage{graphicx,color}
\usepackage{amsmath}
\usepackage{mathrsfs}
\usepackage{esint}
\usepackage{stmaryrd}
\usepackage{enumitem}
\usepackage[all]{xy}
\usepackage{textcomp}
\usepackage{amsbsy}
\usepackage{datetime}
\renewcommand{\dateseparator}{-}
\renewcommand{\today}{\the\year \dateseparator \twodigit\month
\dateseparator \twodigit\day}

\usepackage
[pdfauthor={Oleg Ivrii},
 pdftitle={The geometry of the Weil-Petersson metric in complex dynamics}]
{hyperref}

\DeclareMathOperator{\mix}{mix}
\DeclareMathOperator{\sma}{small}
\DeclareMathOperator{\med}{med}

\DeclareMathOperator{\Var}{Var}
\DeclareMathOperator{\Aut}{Aut}
\DeclareMathOperator{\Hbar}{\overline{\mathbb{H}}}

\DeclareMathOperator{\SL}{SL}
\DeclareMathOperator{\sla}{sl}
\DeclareMathOperator{\lar}{large}
\DeclareMathOperator{\hug}{huge}
\DeclareMathOperator{\WP}{WP}

\DeclareMathOperator{\supp}{supp}

\DeclareMathOperator{\const}{const}
\DeclareMathOperator{\re}{Re}
\DeclareMathOperator{\im}{Im}
\DeclareMathOperator{\ideal}{id}

\DeclareMathOperator{\Mod}{Mod}
\DeclareMathOperator{\Hdim}{H.dim }
\DeclareMathOperator{\half}{half}
\DeclareMathOperator{\cmod}{mod}
\DeclareMathOperator{\pinch}{pinch}
\DeclareMathOperator{\intext}{in}
\DeclareMathOperator{\Area}{Area}
\DeclareMathOperator{\Rat}{Rat}

\usepackage{chngcntr}

\theoremstyle{plain} %% bold header, slanted body

\newtheorem{lemma}{Lemma}[section]
\newtheorem{theorem}{Theorem}[section]
\newtheorem*{corollary}{Corollary}

\theoremstyle{definition} %% bold header, plain body

\newtheorem{conjecture}{Conjecture}

\newtheorem*{question}{Question}

\theoremstyle{remark} %%  slanted header, plain body
\newtheorem*{remark}{Remark}
\newtheorem*{remark2}{Higher degree}

\numberwithin{equation}{section}

\title[Weil-Petersson metric in complex dynamics]{The geometry of the Weil-Petersson \\ metric in complex dynamics}

\author[Oleg Ivrii]{Oleg Ivrii}
\address{Department of Mathematics and Statistics, University of Helsinki, 
         P.O. Box 68, FIN-00014, Helsinki, Finland}
\email{oleg.ivrii@helsinki.fi}

\begin{document}

\begin{abstract} 
In this work, we study  an analogue of the Weil-Petersson metric on the space of Blaschke products of degree 2 proposed by McMullen.
Via the Bers embedding, one may view the Weil-Petersson metric as a metric on the main cardioid of the Mandelbrot set.
We prove that the metric completion attaches the geometrically finite parameters from the Euclidean boundary of the main cardioid and conjecture that this is the entire completion.

For the upper bound, we estimate the intersection of a circle $S_r = \{z : |z| = r\}$, $r \approx 1$, with an invariant subset
$\mathcal G \subset \mathbb{D}$ called a half-flower garden, defined in this work. For the lower bound, we use gradients of multipliers of repelling periodic orbits on the unit circle. 
Finally, utilizing the convergence of Blaschke products to vector fields, we compute the rate at which the Weil-Petersson metric decays along radial degenerations.
\end{abstract}

%\subjclass[2010]{Primary 37F30; Secondary 30J10} 

\thanks{
This work is essentially a revised version of the author's PhD thesis at Harvard University. While at University of Helsinki, the author was supported by the 
Academy of Finland, project no. 271983.}

\maketitle
 \thispagestyle{empty}

  \onehalfspacing

\setcounter{tocdepth}{1}
\tableofcontents

\pagenumbering{arabic}

\input{introduction}

\input{background-in-analysis}

\input{blaschke-products}

\input{petals-flowers}

\input{quasiconformal-deformations}

\input{incompleteness-special}

\input{renewal-theory}

\input{simple-cycles}

\input{lower-bounds}

\input{small-horoball-argument}

\input{limiting-vector-fields}

\input{asymptotics}

\singlespacing

\input{main.bbl}

\end{document}

%% file: introduction.tex
\section{Introduction}

\subsection{Basic notation}
Let $m$ denote the Lebesgue measure on the unit circle $S^1$, normalized to have total mass 1.
Given two points $z_1, z_2 \in \mathbb{D}$, let $d_{\mathbb{D}}(z_1,z_2) = \inf \int_\gamma \rho$ denote the hyperbolic distance between $z_1$ and $z_2$, and $[z_1,z_2]$ 
be  the hyperbolic geodesic connecting $z_1$ and $z_2$. 
We use the convention that the hyperbolic metric on the unit disk is $\rho(z)|dz| = \frac{2|dz|}{1-|z|^2}$ while the Kobayashi metric is $\frac{|dz|}{1-|z|^2}$.
 For $z \in \mathbb{C} \setminus \{0\}$, let $\hat{z} := z/|z|$.
Let $\mathcal B_{p/q}(\eta) \subset \mathbb{D}$ be the  horoball of Euclidean diameter $\eta/q^2$ which rests on 
$e(p/q) := e^{2\pi i (p/q)}$ and $\mathcal H_{p/q}(\eta) = \partial \mathcal B_{p/q}(\eta)$ be its boundary horocycle.
To compare quantities, we use:
\begin{itemize}
\item $A \lesssim B$ means $A < \const \cdot \, B$,
\item  $A \sim B$ means $A/B \to 1$,
\item $A \asymp B$ means $C_1 \cdot  B < A < C_2 \cdot B$ for some constants $C_1, C_2 > 0$,
\item $A \approx_\epsilon B$ means $|A/B-1| \lesssim \epsilon$.
\end{itemize}

\subsection{The traditional Weil-Petersson metric}

To set the stage, we recall the definition and basic properties of the Weil-Petersson metric on  Teichm\"uller space.
Let $\mathcal T_{g,n}$ denote the  Teichm\"uller space of marked Riemann surfaces of genus $g$ with $n$ punctures. 
For a Riemann surface $X \in \mathcal T_{g,n}$, consider the spaces
\begin{itemize}
\item $Q(X)$ of holomorphic quadratic differentials with $\int_X |q| < \infty$,
\item $M(X)$ of measurable Beltrami coefficients satisfying $\|\mu\|_\infty < \infty.$
\end{itemize}
There is a natural pairing between quadratic differentials and Beltrami coefficients given by integration $\langle \mu, q \rangle = \int_X \mu q$.
One has natural identifications $$T_X^*\mathcal T_{g,n} \cong Q(X), \qquad T_X \mathcal T_{g,n} \cong M(X)/Q(X)^\perp.$$
We will discuss two natural metrics on Teichm\"uller space: the Teichm\"uller metric and the Weil-Petersson metric.
On the cotangent space, the Teichm\"uller and Weil-Petersson norms are given by 
$$\|q\|_{T} = \int_X |q|, \qquad \|q\|_{\WP}^2 = \int_X \rho^{-2} |q|^2.$$
The Teichm\"uller and Weil-Petersson lengths of tangent vectors are defined by duality, i.e.~$\|\mu\|_T := \sup_{\|q\|_T = 1} \bigl |\int_X \mu q \bigr |$ and
$\|\mu\|_{\WP} := \sup_{\|q\|_{\WP} = 1} \bigl | \int_X \mu q \bigr |$.
From the definitions, it is clear that the Teichm\"uller and Weil-Petersson metrics are invariant under the mapping class group $\Mod_{g,n}$. 
However, unlike the Teichm\"uller metric, the Weil-Petersson metric is not complete.

For the Teichm\"uller space of a punctured torus $\mathcal T_{1,1} \cong \mathbb{H}$, the mapping class group is $\Mod_{1,1} \cong \SL(2, \mathbb{Z})$. Let us denote the Weil-Petersson metric on $\mathcal T_{1,1}$ by $\omega_T(z)|dz|$.
To describe the metric completion of $\mathcal (\mathcal T_{1,1},\omega_T)$, we introduce a system of disjoint horoballs.  Let $B_{1/0}(\eta)$ denote the horoball
$\{z : y \ge 1/\eta\}$ that rests on $\infty = 1/0$ and $B_{p/q}(\eta)$ denote the horoball of Euclidean diameter $\eta/q^2$ that rests on $p/q$.
For a fixed $\eta \ge 0$, $\bigcup_{p/q \in \mathbb{Q} \cup \{\infty\}} B_{p/q}(\eta)$ is an $\SL(2,\mathbb{Z})$-invariant collection of horoballs.
 When $\eta = 1$, the horoballs have disjoint interiors but many mutual tangencies. We denote the boundary horocycles by $H_{p/q}(\eta) := \partial B_{p/q}(\eta)$ and $H_{1/0}(\eta) := \partial B_{1/0}(\eta)$.
 
Consider $\mathbb{H}$ with the usual topology.
Extend this topology to $\mathbb H^* = \mathbb{H} \cup \mathbb{Q} \cup \{\infty\}$ 
by further requiring $\{B_{p/q}(\eta)\}_{\eta \ge 0}$ to be open sets for $p/q \in \mathbb{Q} \cup \{\infty\}$. 
Let us also consider a family of incomplete $\SL(2,\mathbb{Z})$-invariant model metrics $ \rho_\alpha$ on the upper half-plane: for $\alpha > 0$, let $\rho_\alpha$ be the
unique  $\SL(2,\mathbb{Z})$-invariant  metric which coincides with the hyperbolic metric $|dz|/y$ on 
$\mathbb{H} \setminus \bigcup_{p/q \in \mathbb{Q} \cup \{\infty\}} B_{p/q}(1)$ and is equal to $|dz|/y^{1+\alpha}$ on $B_{1/0}(1)$.

\begin{lemma} For $\alpha > 0$, the metric completion of $(\mathbb{H}, \rho_\alpha)$ is homeomorphic to $\mathbb H^*$.
\end{lemma}
\begin{proof}[Sketch of proof] To see that the irrational points are infinitely far away in the $ \rho_\alpha$ metric, notice that the  horoballs $\{B_{p/q}(2)\}$ cover the upper half-plane, while by $\SL(2,\mathbb{Z})$-invariance,  the distance between  $H_{p/q}(2)$ and $H_{p/q}(3)$ is bounded below in the $ \rho_\alpha$ metric. Therefore, any path $\gamma$ that tends to an irrational number must pass through infinitely many protective shells $B_{p/q}(3) \setminus B_{p/q}(2)$. In fact, this argument shows that an incomplete path $\gamma$ is trapped within some horoball $B_{p/q}(3)$, from which it follows that it must eventually enter arbitrarily small horoballs. By the form of $ \rho_\alpha$ in  $B_{p/q}(1)$, it is easy to see that the completion attaches only one point to the cusp at $p/q$.
\end{proof}

\begin{theorem}[Wolpert]
The Weil-Petersson metric on $\mathcal T_{1,1}$ is comparable to $\rho_{1/2}$, i.e.~$1/C \le \omega_{T}/\rho_{1/2} \le C$ for some $C > 1$.
\end{theorem}

\begin{corollary}
The metric completion of $(\mathcal T_{1,1}, \omega_T)$ is homeomorphic to $\mathbb{H}^*$.
\end{corollary}
 
%  For background on Teichm\"uller theory and more information on the Weil-Petersson metric, we refer the reader to the books
 %  \cite{Hub, IT, Wol}.

\subsection{Main results}

In this paper, we replace the study of Fuchsian groups with complex dynamical systems on the unit disk $\mathbb{D} = \{ z: |z| < 1 \}$. Inspired by Sullivan's dictionary, we are interested in understanding the Weil-Petersson metric on the space
\begin{equation}\mathcal B_2 = 
\left\{
 \begin{array}{c}
   f: \mathbb{D} \to \mathbb{D} \text{ is a proper degree 2 map}\\
   \text{ with an attracting fixed point}
  \end{array}  \right\} \ \Bigl /  \text{ conjugacy by Aut}(\mathbb{D}).
\end{equation}
The multiplier at the attracting fixed point $a: f \to f'(p)$ gives a holomorphic isomorphism $\mathcal B_2 \cong \mathbb{D}$.
By putting the attracting fixed point at the origin, we can parametrize $\mathcal B_2$ by
\begin{equation}
a \in \mathbb{D}: \qquad z \to f_{a}(z) = z \cdot \frac{z+a}{1+\overline{a}z}.
\end{equation}

All degree $2$ Blaschke products are quasisymmetrically conjugate to each other on the unit circle, and except for the special map $z \to z^2$, they are quasiconformally conjugate on the entire disk. For this reason, it is somewhat simpler to work with $\mathcal{B}_2^\times
:= \mathcal B_2 \setminus \{z \to z^2\}$, the quasiconformal moduli space $\mathcal M(f)$ of a rational map described in \cite{MS}.
Given a Blaschke product $f \in \mathcal B_2^\times$, an $f$-invariant Beltrami coefficient on the unit disk $\mu \in M(\mathbb{D})^f$  defines a
 tangent vector in $T_f \mathcal B_2^\times$. 
Since an $f$-invariant Beltrami coefficient descends to a Beltrami coefficient on the quotient torus of the attracting fixed point,  $M(\mathbb{D})^{f} \cong M(T_f)$. According
to \cite{MS}, $\mu$ defines a trivial deformation in $\mathcal B_2^\times$ if and only if it defines a trivial deformation of $T_f \in \mathcal T_{1,1}$. In other words, one has a natural identification of tangent spaces $T_f \mathcal B_2^\times \cong T_{T_f} \mathcal T_{1,1}$ which shows that 
$\mathcal T_{1,1}$ is the universal cover of $\mathcal B_2^\times$. 

To make the parallels with Teichm\"uller theory more explicit, we  state  our results on the universal cover.
For this purpose, we pullback the Weil-Petersson metric on $\mathcal B_2$ by $a(\tau) = e^{2\pi i \tau}$ to obtain
a metric on $\mathcal T_{1,1} \cong \mathbb{H}$, which we also denote $\omega_B$.

\begin{conjecture}
\label{main-conjecture}
The metric $\omega_B$ on $\mathcal T_{1,1} \cong \mathbb{H}$ is comparable to $ \rho_{1/4}$ on $\{\tau : \im \tau < 1\}$. In particular, the metric completion of $(\mathcal T_{1,1}, \omega_B)$ is homeomorphic to $\mathbb{H}^*$.
\end{conjecture}

In this paper, we show that $1/4$ is the correct exponent in the conjecture above. More precisely, we show that:

\begin{theorem}
\label{small-horoball-argument}
The Weil-Petersson metric $\omega_B$ on $\mathcal T_{1,1} \cong \mathbb{H}$ satisfies:
\begin{enumerate}
\item[$(a)$] $\omega_B \le C \rho_{1/4}$.
\item[$(b)$] There exists $C_{\sma} >0$ such that on $\bigcup_{p/q \in \mathbb{Q}} B_{p/q}(C_{\sma})$, $\omega_B \ge (1/C) \rho_{1/4}$.
\end{enumerate}
\end{theorem}

\begin{corollary}
\label{initial-statement2}
The Weil-Petersson metric on $\mathcal B_2$ is incomplete. In fact, the Weil-Petersson length of each line segment $e(p/q) \cdot [1/2, 1)$ is finite.
\end{corollary}

\begin{corollary}
\label{initial-statement}
The space $\mathbb{H}^*$ naturally embeds into the completion of $(\mathcal{T}_{1,1}, \omega_B)$.
\end{corollary}
 
\begin{remark}
Since the Weil-Petersson metric is a real-analytic metric on $\mathcal B_2$, the cusp at infinity in
the  $\mathbb{H}^*$-model is 
somewhat special:  $$w_B \sim C e^{-2\pi \im \tau} |d\tau|, \qquad \text{as} \im \tau \to \infty.$$
\end{remark}

 Along radial rays $a \to e(p/q)$, we have a more precise estimate: 
 
\begin{theorem}
\label{fine-geometry}
Given a rational number $p/q \in \mathbb{Q}$, as $\tau = p/q + it \to p/q$ vertically, the ratio $\omega_B/ \rho_{1/4} \to C_q$, where $C_q$ is a positive constant independent of $p$.
\end{theorem}

\begin{conjecture}
We conjecture that $C_{q}$ is a universal constant, independent of $q$.
\end{conjecture}

In a forthcoming work \cite{ivrii-rescaling}, we will show that the Weil-Petersson metric is asymptotically periodic if we approach $a \to e(p/q)$ along a horocycle.
The proof combines ideas from the work of Epstein \cite{E} on rescaling limits with parabolic implosion.

\subsection{Properties of the Weil-Petersson metric}
\label{sec:properties}

In this section, we give a definition of the Weil-Petersson metric on $\mathcal B_2^\times \subset \mathcal B_2$ in the form most useful for our later work. In Section \ref{sec:notes-references}, we will give equivalent definitions which work on the entire space $\mathcal B_2$. For example, 
the Weil-Petersson metric may be described as the second derivative of the Hausdorff dimension
of  one-parameter families of Julia sets.

It is convenient to put the Beltrami coefficient on the exterior unit disk.
For a Beltrami coefficient $\mu \in M(\mathbb{D})$, we let $\mu^+$ denote the reflection of $\mu$ in the unit circle:
\begin{equation}
\label{eq:reflecting-beltrami}
\mu^+ =
\left\{
 \begin{array}{ll}
  0  \qquad & \text{for } z \in \mathbb{D}, \\
 \overline{(1/\overline z)^*\mu}  \qquad &  \text{for } z \in S^2 \setminus \mathbb{D}.
  \end{array}  \right.
\end{equation}

Suppose $X  \in \mathcal T_{g,n}$ is a Riemann surface and $\mu \in M(X)$ is a Beltrami coefficient. If $X \cong \mathbb{D}/\Gamma$, we can consider $\mu$ as a $\Gamma$-invariant Beltrami coefficient on the unit disk. Let $v$ be a solution of $\overline{\partial}v = \mu^+$. Since the set of all solutions is of the form $v + \sla(2, \mathbb{C})$, the third derivative
$v'''$ uniquely depends on $\mu^+$. As $v'''$ is an infinitesimal version of the Schwarzian derivative, it is naturally a quadratic differential. 
In \cite{McM-wp}, McMullen observed that
\begin{equation}
\label{eq:integral-average0}
\frac{
\|\mu\|_{\WP}^2}
{4 \cdot \Area(X, \rho^2)} \, = \, \mathcal I[\mu] \, = \, \lim_{r\to 1^{-}}  \frac{1}{2\pi} \int_{|z|=r} \biggl |\frac{v_{\mu^+}'''(z)}{\rho(z)^2} \biggr |^2 
d\theta.
\end{equation}
Similarly, given a Blaschke product $f \in \mathcal B_2^\times$, we can solve the equation $\overline{\partial} v = \mu^+$ for $\mu \in M(\mathbb{D})^f$. 
As above, a solution $v$ of the equation $\overline{\partial} v = \mu^+$ is well-defined up to adding a holomorphic vector field in $\sla(2, \mathbb{C})$ so that $v'''$ is uniquely defined.
Following \cite{McM-wp}, we {\em define} the Weil-Petersson metric $\|\mu\|_{\WP}^2 := \mathcal I[\mu]$ provided that the limit exists.
%(Thus, the Weil-Petersson metric on $\mathcal B_2^\times$ more closely resembles the genus-independent version of the Weil-Petersson metric on Teichm\"uller space.)
 In Section \ref{chap:renewal-theory}, we will use renewal theory to establish 
the existence of this limit for any $\mu \in M(\mathbb{D})^f$, invariant under a degree 2 Blaschke product other than $z \to z^2$.

\medskip

 \begin{figure}[h!]
  \centering
      \includegraphics[width=0.32\textwidth]{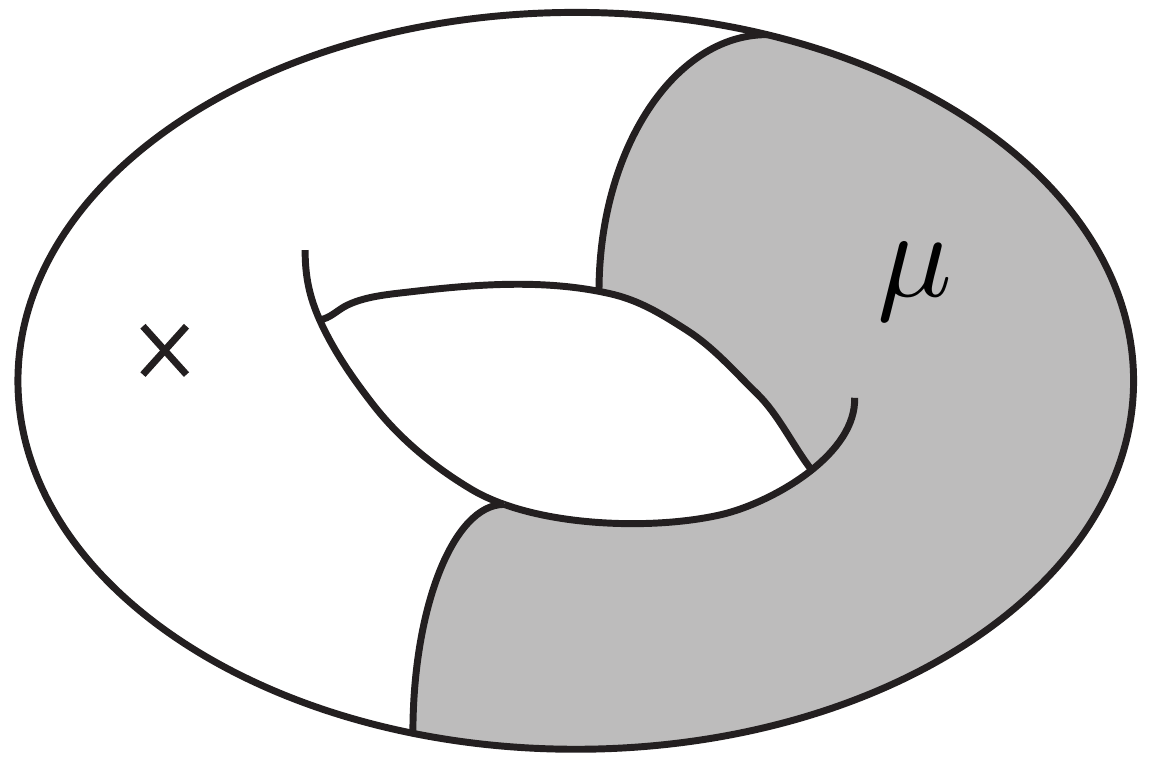}
  \caption{The support of the Beltrami coefficient takes up half of the quotient torus.}
   \label{fig:degenerating-to-1b}
\end{figure}

\subsection{A glimpse of incompleteness}% as $a \to 1$ radially}

We now sketch the proof of the upper bound in Theorem \ref{small-horoball-argument}.
To establish the incompleteness of the Weil-Petersson metric, we consider ``half-optimal'' Beltrami coefficients  $\mu_\lambda \cdot \chi_{\mathcal G(f_a)}$ which take up half 
of the quotient torus at the attracting fixed point, but are sparse near the unit circle.

\medskip
\medskip

 \begin{figure}[h!]
  \centering
      \includegraphics[width=0.3\textwidth]{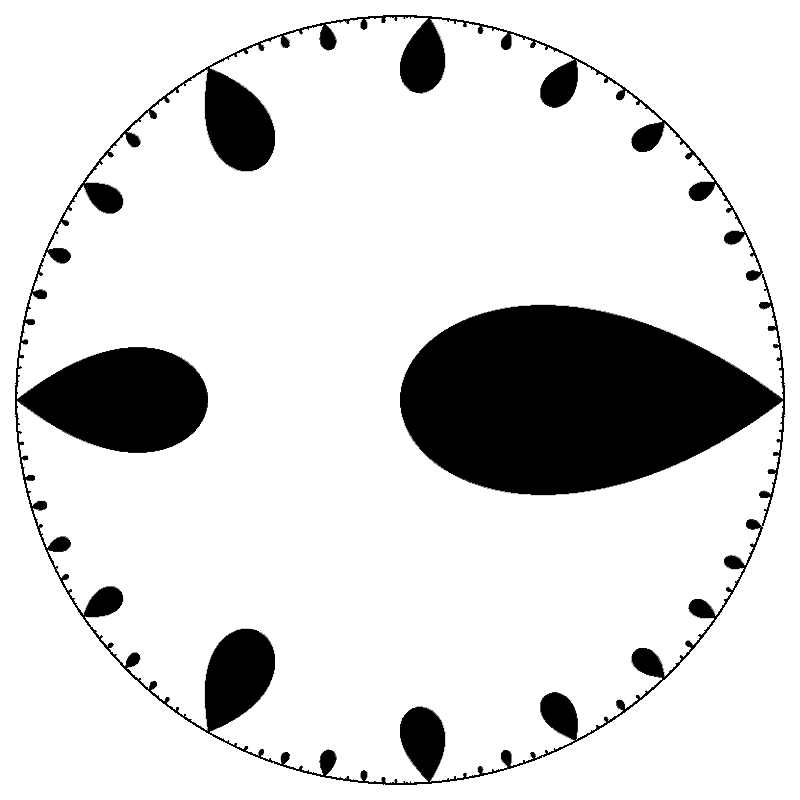}
      \qquad
      \qquad
            \includegraphics[width=0.3\textwidth]{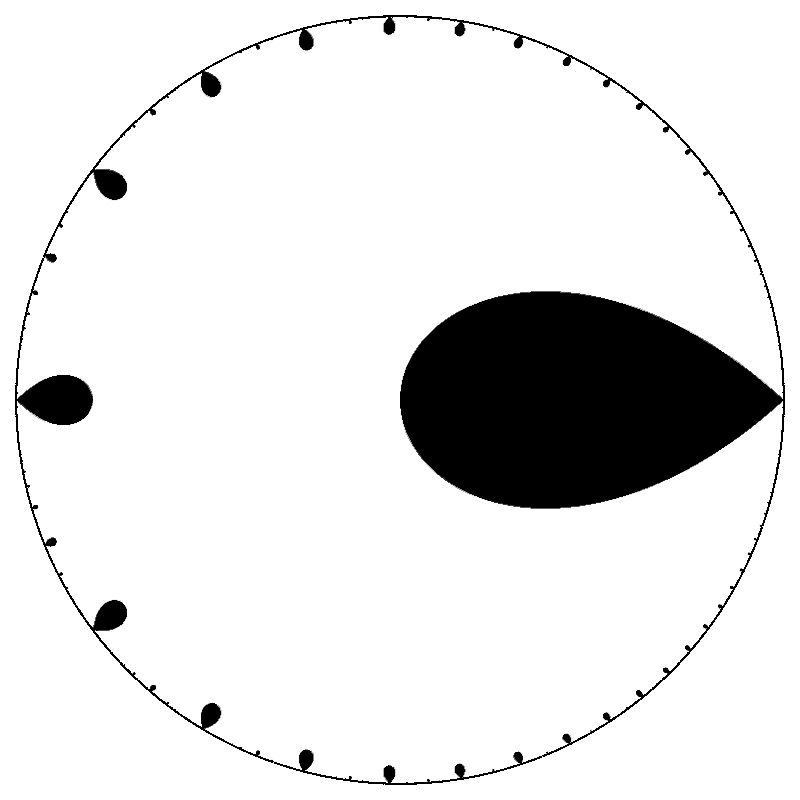}
  \caption{Gardens $\mathcal G(f_a)$ for the Blaschke products with $a = 0.5$ and $0.8$.}
   \label{fig:degenerating-to-1}
\end{figure}
 
 \begin{figure}[h!]
  \centering
      \includegraphics[width=0.8\textwidth]{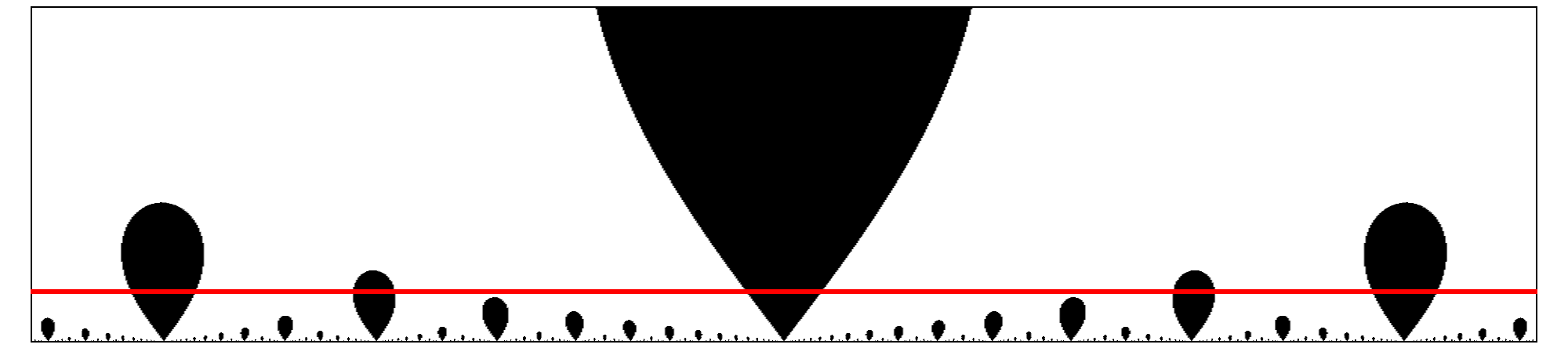}
  \caption{A blow-up of ${\mathcal G}(f_{0.5})$ near the boundary. A circle $\{z : |z| = r\}$} with $r$ close to 1 meets ${\mathcal G}(f_{0.5})$ in small density.
   \label{fig:degenerating-to-1c}
\end{figure}

The garden $\mathcal G(f_a) \subset \mathbb{D}$ is an invariant subset of the unit disk
whose quotient  $A = \mathcal G(f_a)/f_a \subset T_a$ is a certain annulus that takes up half of the Euclidean area of the quotient torus.
To give an upper bound for the Weil-Petersson metric, we estimate the length of the intersection of ${\mathcal G}(f_a)$ with $S_r := \{z: |z|=r\}$. In general, one has the estimate
\begin{equation}
\label{eq:quotient-wphyp}
\biggl ( \frac{\omega_B}{\rho_{\mathbb{D}^*}} \biggr ) ^2 \le C \cdot \limsup_{r \to 1} |{\mathcal G(f_a)} \cap S_r|.
\end{equation}
In order for this estimate to be efficient, we take $A$ to be a collar neighbourhood of the shortest $p/q$-geodesic in the quotient torus $T_{f_a}
\in \mathcal T_{1,1}$.
For the Blaschke product $f_a$ with parameter $a = e^{2\pi i \tau}$, $\tau \in  H_{p/q}(\eta)$, we prove
\begin{equation}
\label{eq:small-intersection}
\limsup_{r \to 1} |{\mathcal G(f_a)} \cap S_r| = O(\eta^{1/2}).
\end{equation}
Combining (\ref{eq:quotient-wphyp}) and (\ref{eq:small-intersection}), we see that $\omega_B \le C \rho_{1/4}$ on $\{\tau : \im \tau < 1\}$ as desired.

\begin{remark}
The trick of truncating the support of the Beltrami coefficient can be found in the proof of \cite[Corollary 1.3]{McM-cusps}. See also \cite{B}.
\end{remark}

\subsection{A glimpse of the convergence $\omega_B/ \rho_{1/4} \to C_{q}$}

We now  sketch the proof of Theorem \ref{fine-geometry}. 
To understand the behaviour of the Weil-Petersson metric as $a \to e(p/q)$ radially, we study the convergence of Blaschke products to vector fields. For example, as $a \to 1$ along the real axis, we will see that even though the maps $f_a(z) = z \cdot \frac{z+a}{1+\overline{a}z}$ tend pointwise to the identity, their long-term dynamics tends to the flow of the holomorphic vector field $\kappa_1 = z \cdot \frac{z-1}{z+1} \cdot \frac{\partial}{\partial z}$.
For the radial approach $a \to e(p/q)$, the maps $f_a(z) \to e(p/q)z$ converge pointwise to a rotation, and therefore the $q$-th iterates % $f_a^{\circ q}(z)$
 tend to the identity. 
We are thus led to extract a limiting vector field $\kappa_{q}$ by considering limits of the high iterates
of $f_a^{\circ q}$. It turns out that the vector field $\kappa_{q}$ is a $q$-fold cover of the vector field $\kappa_{1}$. In particular, it is independent of $p$.

 \begin{figure}[h!]
  \centering
      \includegraphics[width=0.4\textwidth]{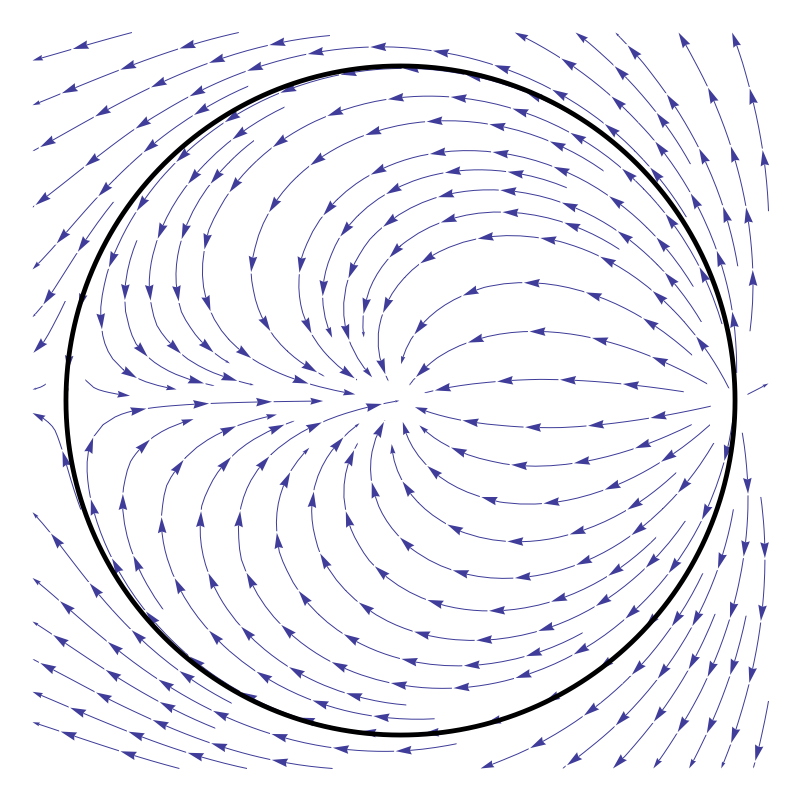}
      \qquad
      \qquad
            \includegraphics[width=0.4\textwidth]{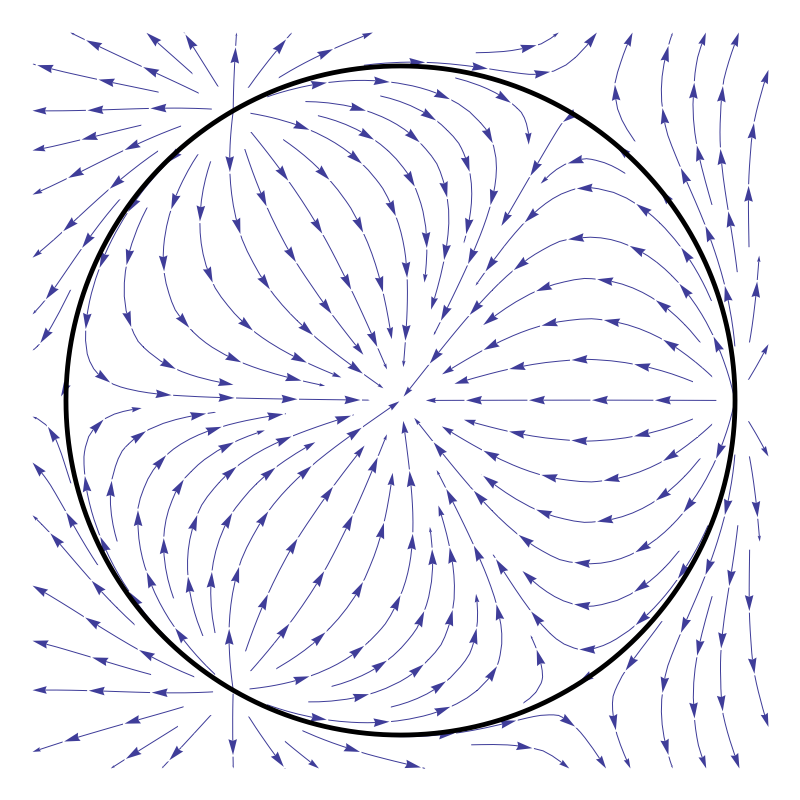}
  \caption{The vector fields $\kappa_1$ and $\kappa_{3}$.}
   \label{fig:aVF}
\end{figure}

From the convergence of Blaschke products to vector fields, it follows that the flowers that make up the gardens $\mathcal G(f_a)$ for $a \approx e(p/q)$ have nearly the same shape, up to affine scaling. 
Intuitively, for the integral average (\ref{eq:integral-average0}) to exist, when we replace $r = 1 - \delta$ by $r = 1 - \delta/2$ say, we expect to intersect twice as many flowers to replenish the integral, i.e.~we expect the number of flowers to be inversely proportional in $\delta$.
This leads us to explore an orbit counting problem for Blaschke products. The decay rate of the Weil-Petersson metric is governed by the dependence of the flower count on
the parameter variable $a$.
  
\subsection{Notes and references}

\label{sec:notes-references}

In this section, we describe the space of Blaschke products of higher degree and equivalent definitions of the Weil-Petersson metric.

\medskip

\noindent \textbf{Blaschke products of higher degree.} More generally, we can consider the space $\mathcal B_d$ of marked  Blaschke products of degree $d$ which have an attracting fixed point modulo conformal conjugacy.
By moving the attracting fixed point to the origin as before, one can parametrize $\mathcal B_d$ by
\begin{equation}
\{ a_1, a_2, \dots, a_{d-1} \} \in \mathbb{D}: \qquad z \to f_{\mathbf{a}}(z) = z \cdot \prod_{i=1}^{d-1} \frac{z+a_i}{1+\overline{a_i}z}.
\end{equation}
Let $a := a_1a_2 \cdots a_{d-1} = f_{\mathbf{a}}'(0)$ denote the multiplier of the attracting fixed point.
It is because the maps are {\em marked} that we can distinguish the conformal conjugacy classes of ${\mathbf{a}} = \{ a_1, a_2, \dots, a_{d-1}\}$ and 
$\zeta \cdot {\mathbf{a}} = \{ \zeta  a_1, \zeta  a_2, \dots, \zeta  a_{d-1} \}$. See \cite{McM-exp} for more on markings.

\medskip

\noindent \textbf{Mating.} It is a remarkable fact that given two Blaschke products $f_{\mathbf{a}}, f_{\mathbf{b}}$ of the same degree, one can find a rational map $f_{{\mathbf{a}}, {\mathbf{b}}}(z)$ -- the {\em mating} of $f_{\mathbf{a}}, f_{\mathbf{b}}$ -- whose Julia set is a quasicircle $\mathcal J_{{\mathbf{a}}, {\mathbf{b}}}$ which separates the Riemann sphere into two domains $\Omega_-, \Omega_+$ such that on one side $f_{{\mathbf{a}}, {\mathbf{b}}}(z)$ is conformally conjugate
to $f_{\mathbf{a}}$, and to $f_{\overline{{\mathbf{b}}}}$ on the other. The mating is unique up to conjugation by a M\"obius transformation.
One can prove the existence of a mating by quasiconformal surgery (see \cite{Mil-hc} for details). The mating $\mathcal B_d \times 
\overline{\mathcal B_d} \to {\Rat}_d$ varies holomorphically with parameters. 
A natural way to put a complex structure on $\mathcal B_d$ is via the {\em Bers embedding} $\mathcal B_d \to \mathscr P_d$
which takes a Blaschke product and mates it with $z^d$ to obtain a polynomial of degree $d$.
Here, the space $\mathscr P_d \cong {\mathbb{C}}^{d-1}$ is considered modulo affine conjugacy. The image of the Bers embedding is the generalized main cardioid in $\mathscr P_d$. 
\begin{question}
For $d \ge 3$, what is the completion of $\mathcal B_d$ with respect to the Weil-Petersson metric? Are the additional points precisely the geometrically finite parameters on the boundary of the generalized main cardioid? What is the topology on $\overline{\mathcal B_d}$?
\end{question}
 \begin{remark}
Wolpert showed that the metric completion of $(\mathcal T_{g,n}, \omega_T)$ is the augmented Teichm\"uller space $\overline{\mathcal T_{g,n}}$, the action of the mapping class group $\Mod_{g,n}$ extends isometrically to $(\overline{\mathcal T_{g,n}}, \omega_T)$ and the quotient $M_{g,n} = \overline{\mathcal T_{g,n}}/\Mod_{g,n}$ is the Deligne-Mumford compactification.  See \cite{Wol} for more information.
\end{remark}
\noindent \textbf{Equivalent definitions of the Weil-Petersson metric.} Suppose $f \in \mathcal B_d$ and $f_t, \ t \in (-\epsilon, \epsilon)$ is a smooth path
with $f_0 = f$, representing a tangent vector in $T_f \mathcal B_d$. Consider the vector field
 $v(z) := \frac{d}{dt} \bigl |_{t=0} \, H_{0,t}(z)$ where $H_{0,t}: \mathbb{D} \to \Omega_-(f_{0,t})$ is the conformal conjugacy between $f_0$ and $f_{0,t}$. If $f$ is a Blaschke product other than $z \to z^d$, one can define $\|\dot f_t\|_{\WP}^2$ by the integral average (\ref{eq:integral-average0}), while if $f(z) = z^d$, one can use a more complicated integral average described in \cite{McM-wp}.

\begin{remark}
 The definition of the Weil-Petersson metric via mating is slightly more general than the one via quasiconformal conjugacy given earlier because quasiconformal deformations do not exhaust the entire tangent space $T_f\mathcal B_d$ at the special parameters $f \in \mathcal B_d$ that have critical relations.
\end{remark}

  In \cite{McM-wp}, McMullen showed that
\begin{eqnarray}
\label{eq:mcm-theorem}
\|\dot f_t\|_{\WP}^2 = \frac{3}{4} \cdot \frac{\Var(\dot \phi, m)}{\int \log|\phi'|dm} 
& = & \frac{3}{4} \cdot \frac{d^2}{dt^2} \biggl |_{t=0} \Hdim \mathcal J_{0,t} \\ 
\label{eq:mcm-theorem2}
& = & - \frac{3}{16} \cdot \frac{d^2}{dt^2} \biggl |_{t=0} \Hdim(H_{t,t})_* m
\end{eqnarray}
where

\begin{itemize}
\item[]  $\mathcal J_{0,t}$ is the Julia set of $f_{0,t}$,
 
\item[]  $H_{t,t}: S^1 \to S^1$ is the conjugacy between $f_0$ and $f_t$ on the unit circle,

\item[]  $(H_{t,t})_* m$ is the push-forward of the Lebesgue measure,
  
\item[]  $\phi_t = \log|f_{0,t}'(H_{0,t}(z))|$,
  
\item[] $\int \log|\phi'|dm$ is the Lyapunov exponent,

\item[] $\Var(h, m) :=
\lim_{n \to \infty} \frac{1}{n} \int |S_n h(x)|^2 dm$ denotes the ``asymptotic variance'' in the context of dynamical systems.
\end{itemize}

\begin{remark}
Since $\mathcal J_{0,t}$ is a Jordan curve, $\Hdim \mathcal J_{0,t} \ge 1$, so $\frac{d}{dt} \bigl |_{t=0} \, \Hdim \mathcal J_{0,t} = 0$ and $\frac{d^2}{dt^2} \bigl |_{t=0} \, \Hdim \mathcal J_{0,t} \ge 0$. Similarly, since $(H_{t,t})_* m$ is a measure supported on the unit circle, $\Hdim(H_{t,t})_* m \le 1$, $\frac{d}{dt} \bigl |_{t=0} \, \Hdim(H_{t,t})_* m = 0$ and $\frac{d^2}{dt^2} \bigl |_{t=0} \, \Hdim(H_{t,t})_* m \le 0$.
\end{remark}

\subsection{Relations with quasiconformal geometry}

\label{sec:related-ideas}

 The characterizations (\ref{eq:mcm-theorem}) and (\ref{eq:mcm-theorem2}) of the Weil-Petersson metric are reflected in the duality between
quasiconformal expansion and quasisymmetric compression:

\begin{theorem}[Smirnov \cite{S}]
\label{thm-smirnov}
For a $k$-quasiconformal map $f: S^2 \to S^2$, $$\Hdim f(S^1) \le 1+k^2.$$
\end{theorem}

\begin{remark} If the dilatation $\mu(z) = \frac{\overline{\partial}f}{\partial f}$ is supported on the exterior unit disk, one has the stronger estimate
$\Hdim f(S^1) \le 1+\tilde k^2$ where $k = \frac{2 \tilde k}{1+\tilde k^2}$.
\end{remark}

\begin{theorem}[Prause, Smirnov \cite{PrS}]
\label{thm-smirnov-prause}
For a $k$-quasiconformal map $f:S^2 \to S^2$, symmetric with respect to the unit circle, one has $\Hdim f_*m \ge 1-k^2$.
\end{theorem}

An application of Theorem \ref{thm-smirnov} or Theorem \ref{thm-smirnov-prause} shows:

\begin{corollary}
The Weil-Petersson metric on $\mathcal B_2$ is bounded above by $\sqrt{3/32} \cdot \rho_\mathbb{D}$.
\end{corollary}

\begin{proof}
For a map $f_a \in \mathcal B_2$, the Bers embedding $\beta_{f_a}$
gives a holomorphic motion of the exterior unit disk $H: \mathcal B_2 \times (S^2 \setminus \mathbb{D}) \to \mathbb{C}$ given by $H(b, z) := H_{b,a}(z)$. Note that the
 motion $H$ is centered at $a$ since
$H(a, \cdot)$ is the identity.
By the $\lambda$-lemma (e.g.~see \cite[Theorem 12.3.2]{AIM}), one can extend $H$ to a holomorphic motion $\tilde H$ of the Riemann sphere satisfying
 $\|\mu_{\tilde H(b, \cdot)}\|_\infty \le \frac{b-a}{1-\overline{a}b}$. Observe that as $d_{\mathbb{D}}(b,a) \to 0$, $\frac{b-a}{1-\overline{a}b} \sim \frac{1}{2} \cdot d_{\mathbb{D}}(b,a)$.
Since each map $\tilde H(b, \cdot)$ is conformal on $S^2 \setminus \mathbb{D}$, by the remark following Theorem \ref{thm-smirnov}, we have $\|\dot f_t\|^2_{\WP} \le \frac{1}{4} \cdot \frac{3}{8} \cdot\|\dot f_t\|^2_{\rho_{\mathbb{D}}}$ as desired.
\end{proof}

\subsection*{Acknowledgements}
I would like to express my deepest gratitude to Curtis T. McMullen for his time, energy and invaluable insights.
 I  also want to thank Ilia Binder for many interesting conversations.

%% file: background-in-analysis.tex
\section{Background in Analysis}
In this section, we explain how to bound the integral (\ref{eq:integral-average0}) in terms of the density of the support of  $\mu$. We also discuss a version of Koebe's distortion theorem for maps that preserve the unit circle.
\subsection{Teichm\"uller theory in the disk}
For a Beltrami coefficient $\mu$, let $v(z) = v_\mu(z)$ be a solution of the equation $\overline{\partial} v = \mu$.
 The following formula is well-known (e.g.~see \cite[Theorem 4.37]{IT}):
\begin{equation}
\label{eq:v3}
 v'''(z) dz^2 =  \biggl ( -\frac{6}{\pi} \int_{\mathbb{C}} \frac{\mu(\zeta)}{(\zeta-z)^4} |d\zeta|^2 \biggr) dz^2
 \end{equation}
for $z \not\in \supp \mu$.
\begin{lemma}
\label{lem-mobius-invariance}
For a Beltrami coefficient $\mu$ and a M\"obius transformation $\gamma \in \Aut(S^2)$, we have
$v_{\gamma^*\mu}'''(z)  = v_\mu'''(\gamma z) \cdot \gamma'(z)^2$ whenever $\gamma z \not\in \supp \mu$.
In particular, if $\mu$ is supported on the exterior of the unit disk and $\gamma \in \Aut(\mathbb{D})$, then
 \begin{equation}
 \biggl | \frac{ v_\mu'''}{\rho^2} (\gamma(z)) \biggr | \ = \ \biggl | \frac{ v_{\gamma^*\mu}''' } {\rho^2 }(z) \biggr |, \qquad z \in \mathbb{D}.
 \end{equation}
\end{lemma}

\begin{proof}
The first statement follows from a change of variables and the identity 
\begin{equation}
\label{eq:mobius-invariance}
\frac{\gamma'(z_1) \gamma'(z_2)}{(\gamma(z_1)-\gamma(z_2))^2} = \frac{1}{(z_1-z_2)^2}, \qquad z_1 \ne z_2 \in \mathbb{C}, \ \gamma \in \Aut(S^2),
\end{equation}
while the second statement follows from the fact that $\gamma^*\rho = \rho$ for all $\gamma \in \Aut(\mathbb{D})$.
\end{proof}

To obtain upper bounds for the Weil-Petersson metric, we will use the following estimate:
\begin{theorem}
\label{wpbounds}
Suppose $\mu$ is a Beltrami coefficient which is supported on the exterior of the unit disk and has $\|\mu\|_\infty \le 1$. Then,
\begin{equation}
\label{eq:wpbounds}
 \limsup_{r\to 1^-} \frac{1}{2\pi} \int_{|z|=r} \biggl | \frac{v_\mu'''(z)}{\rho(z)^2} \biggr |^2 d\theta \  \le \   \frac{9}{4} \cdot \|\mu\|^2_\infty \cdot \ \limsup_{R \to 1^+} \ 
 \frac{1}{2\pi} \bigl |\supp \mu \cap S_R  \bigr |.
\end{equation}\end{theorem}

\begin{theorem}
\label{qbounds}
Suppose $\mu$ is a Beltrami coefficient which is supported on the exterior of the unit disk and has $\|\mu\|_\infty \le 1$. Let $\mu^- := \overline{(1/\overline{z})^*\mu}$ be 
its reflection in the unit circle. Then,
 \begin{enumerate}
\item[$(a)$] $|(v'''/\rho^2)(z)| \le 3/2 \cdot \|\mu\|_\infty$ for $z \in \mathbb{D}$.
\item[$(b)$] If $d_{\mathbb{D}}(z, \,\supp \mu ^-) \ge R$ then $|(v'''/\rho^2)(z)| \lesssim e^{-R}.$
\item[$(c)$]  $v'''/\rho^2$ is uniformly continuous in the hyperbolic metric.
\end{enumerate}
\end{theorem}
 \begin{proof}
By the M\"obius invariance of $|v_\mu'''/\rho^2|$, it suffices to prove these assertions at the origin. Clearly, $$|v'''(0)| \le \frac{6}{\pi} \int_{|\zeta|>1} \frac{1}{|\zeta|^4} \cdot |d\zeta|^2 \le 12 \int_1^\infty \frac{dr}{r^3} = 6.$$
Hence $|v'''/\rho^2(0)| \le \frac{3}{2}$. This proves $(a)$. For $(b)$, recall that $d_{\mathbb{D}}(0,z) = -\log(1-|z|)+O(1)$. Then,
 $$
 |v'''(0)| \le \frac{6}{\pi} \int_{1+Ce^{-R}>|\zeta|>1} \frac{1}{|\zeta|^4} \cdot |d\zeta|^2 \lesssim e^{-R}.
 $$
  For $(c)$, it suffices to observe that the kernel
  $
  \frac{1}{(\zeta - z)^4}
  $
  is uniformly continuous at $z=0$ for $\{ \zeta : |\zeta| > 1\}$.
\end{proof}

\begin{proof}[Proof of Theorem \ref{wpbounds}]  
Let $V_\mu(z) := \frac{6}{\pi} \int_{|\zeta|>1} \frac{|\mu(\zeta)|}{|\zeta-z|^4} \cdot |d\zeta|^2$. The calculation from part $(a)$ of Theorem \ref{qbounds} shows that $|V_\mu/\rho^2| \le 3/2 \cdot \|\mu\|_\infty$ has the same $L^\infty$ bound.
Set $\nu(\zeta) := \frac{1}{2\pi} \int  |\mu(e^{i\theta}\zeta)| d\theta$.
From Fubini's theorem, it is clear that $$\int_{|z|=r} |V_\mu/\rho^2| d\theta = \int_{|z|=r} |V_{\nu}/\rho^2|d\theta, \qquad 0 < r < 1.$$  
 Since $\limsup_{|\zeta| \to 1^+} |\nu(\zeta)| \le \|\mu\|_\infty \cdot \ \limsup_{R \to 1^+} \, \frac{1}{2\pi} \bigl |\supp \mu \cap S_R  \bigr |$,
$$
\limsup_{r \to 1^-} \frac{1}{2\pi} \int_{|z|=r} \biggl | \frac{V_\mu(z)}{\rho(z)^2} \biggr | d\theta 
\le \frac{3}{2} \cdot \|\mu\|_\infty \cdot \ \limsup_{R \to 1^+} \ \frac{1}{2\pi} \bigl |\supp \mu \cap S_R  \bigr |.
$$
Equation (\ref{eq:wpbounds}) follows by multiplying the $L^1$ and $L^\infty$ bounds.
\end{proof}

\subsection{A distortion theorem}

\label{chap:distortion-theorem}

The classical version of Koebe's distortion theorem says that if $h: B(0, 1) \to \mathbb{C}$ is univalent, then $|h'(z)-1| \lesssim |z|$ for $|z| < 1/2$.
We will mostly use a version of Koebe's distortion theorem for maps which preserve the real line or the unit circle:
\begin{theorem}
\label{koebe2}
Suppose $h: B(0, 1) \to \mathbb{C}$ is a univalent function which satisfies $h(0) = 0$, $h'(0) = 1$
and takes real values on $(-1,1)$. For $t < 1/2$,  $h$ is nearly an isometry in the hyperbolic metric on 
$B(0,t) \cap \mathbb{H}$, i.e.~$h^*(|dz|/y) \approx_t (|dz|/y)$.
\end{theorem}

In particular, $h$ distorts hyperbolic distance and area by a small amount:
 
\begin{corollary}
For $z_1, z_2 \in B(0,t) \cap \mathbb{H}$, $d_{\mathbb{H}}(z_1,z_2) = d_{\mathbb{H}}(h(z_1),h(z_2)) + O(t)$.
\end{corollary}

\begin{corollary}
\label{area-distortion-lemma}
If $\mathscr B$ is a round ball contained in $B(0,t) \cap \mathbb{H}$, then $$\Area \biggl (\mathscr B, \frac{|dz|^2}{y^2} \biggr )
 \approx_t \Area \biggl ( h(\mathscr B), \frac{|dz|^2}{y^2} \biggr ).$$
\end{corollary}

Above, ``$A \approx_t B$'' denotes that $|A/B -1| \lesssim t$. For a set $E \subset B(0,t)$, we call a set of the form $h(E)$ a $t$-{\em nearly-affine copy} of $E$.

Suppose $\mu$ is a Beltrami coefficient supported on the upper half-ball $B(0,1) \cap \mathbb{H}$. %Set $h^*\mu = \mu(h(z)) \cdot \frac{\overline{h'(z)}}{h'(z)}$.
It is easy to see that for $z \in B(0,t) \cap \mathbb{H}$, $\bigl |(h^*\mu)(z) - \mu(h(z)) \bigr | \lesssim t \cdot \|\mu\|_\infty$
where  $h^*\mu = \mu(h(z)) \cdot \frac{\overline{h'(z)}}{h'(z)}$. In terms of quadratic differentials, we have:
\begin{lemma}
\label{qd-almostinvariant}
On the lower half-ball $B(0,t) \cap \overline{\mathbb{H}}$,
\begin{equation}
\label{eq:qd-almostinvariant}
 \, \biggl |  \frac{v_\mu'''}{\rho^2}  (h(z)) - \frac{v_{h^*\mu}'''}{\rho^2}(z) \biggr | \ \lesssim \  \phi_1(t) \cdot \|\mu\|_\infty,
\end{equation}
for some function $\phi_1(t)$ satisfying $\phi_1(t) \to 0^+$ as $t \to 0^+$.
\end{lemma}
\begin{proof}
Given $R, \epsilon > 0$, we can choose $t > 0$ sufficiently small to guarantee that
$$|h'(\zeta) - 1| < \epsilon \qquad \text{and} \qquad
(z-\zeta) \approx_\epsilon (h(z)-h(\zeta))
$$
for $z \in B(0,t) \cap \overline{\mathbb{H}}$ and $\zeta \in \mathscr B = \{ w: d_{\mathbb{H}}(\overline{z}, w) < R\}$. Together with Theorem \ref{koebe2}, these facts imply (\ref{eq:qd-almostinvariant}) with $\mu$ replaced by $\mu \, \chi_{h(\mathscr B)}$.
  However, by part $(b)$ of Theorem \ref{qbounds}, the contributions of $\mu  (1-\chi_{h(\mathscr B)})$
and  $(h^*\mu) (1-\chi_{\mathscr B})$ to $(v'''_\mu/\rho^2)(h(z))$ and $(v'''_{h^*\mu}/\rho^2)(z)$ respectively are exponentially small in $R$. 
%The lemma follows.
\end{proof}

\subsection{Applications to Blaschke products.} %We will apply Koebe's distortion theorem to the inverse branches of Blaschke products. 
For a Blaschke product $f \in \mathcal B_d$,  let $\delta_c := \min_{c \in \mathbb{D}}(1 - |c|)$ where $c$ ranges over the critical points of $f$ that lie inside the unit disk. By the Schwarz lemma, the post-critical set of $f: S^2 \to S^2$ is contained in the union of $B(0, 1-\delta_c)$ and its reflection in the unit circle.

If $\zeta \in S^1$, the ball $B(\zeta, \delta_c)$ is disjoint from the post-critical set, and therefore all possible inverse branches $f^{-n}$ are well-defined univalent functions
on $B(\zeta, \delta_c)$.
For $0 < t < 1/2$, let $U_t := \{ z : 1 - t \cdot \delta_c \le |z| < 1\}$. For Blaschke products, we have the following analogue of Lemma \ref{qd-almostinvariant}:
\begin{lemma}
\label{qd-almostinvariant2}
If $\mu$ is an invariant Beltrami coefficient supported on the exterior unit disk, and if the orbit $z \to f(z) \to \dots \to f^{\circ n}(z)$ is contained in some $U_t$ with $t < 1/2$ sufficiently small, then
\begin{equation}
 \, \biggl |  \frac{v_\mu'''}{\rho^2}  (f^{\circ n}(z)) \cdot f^{\circ n}(z)^2 - \frac{v_{\mu}'''}{\rho^2}(z) \cdot z^2 \biggr | \  \lesssim \ \phi_2(t) \cdot \|\mu\|_\infty,
\end{equation}
for some function $\phi_2(t)$ satisfying $\phi_2(t) \to 0^+$ as $t \to 0^+$.
\end{lemma}

%% file: blaschke-products.tex
\section{Blaschke Products}
\label{sec:bp}

In this section, we give background information on Blaschke products. We discuss the quotient torus
at the attracting fixed point and special repelling periodic orbits called ``simple cycles'' on the unit circle.
In the next section, we will examine the interface between these two objects.

\subsection{Attracting tori}

   The dynamics of forward orbits of a Blaschke product
\begin{equation}
   f_a(z) = z \cdot \frac{z+a}{1+\overline{a}z}
 \end{equation}
    is very simple: all points in the unit disk are attracted to the origin. In this paper, we mostly assume that 
    the multiplier of the attracting fixed point $a = f'(0) \ne 0$. In this case, the linearizing coordinate $\varphi_{a}(z) := \lim_{n \to \infty} a^{-n} \cdot f_a^{\circ n}(z)$
 conjugates $f_{a}$ to multiplication by $a$, i.e.~
 \begin{equation}
\label{conjugacy}
\varphi_{a} ( f_{a} (z)) = a \cdot \varphi_{a} (z), \qquad z \in \mathbb{D}.
\end{equation}
It is well-known that (\ref{conjugacy}) determines $\varphi_{a}$ uniquely with the normalization $\varphi_{a}'(0) = 1$.

Let $\Omega$ denote the unit disk with the grand orbits of the attracting fixed and critical point removed.
From the existence of the linearizing coordinate, it is easy to see that
the quotient $\hat \varphi_{a}: \Omega \to T^\times_{a} := \Omega/(f_{a})$ is a torus with one puncture. We denote the underlying closed torus by $T_{a}$. 
We will also consider the intermediate covering map $\pi_a: \mathbb{C}^* \to T_a \cong \mathbb{C}^*/(\cdot\ \!a)$ defined implicitly by $\hat \varphi_{a} = \pi_a \circ  \varphi_{a}$.

\begin{remark2}
For a Blaschke product $f_{\mathbf{a}} \in \mathcal B_d$ with $a = f_{\mathbf{a}}'(0) \ne 0$, the quotient torus $T_{\mathbf{a}}^\times$ has at most $(d-1)$ punctures but there could be less if there are critical relations.
The reader may view the space $\mathcal B_d^\times \subset \mathcal B_d$ consisting of Blaschke products for which $T_{\mathbf{a}}^\times \in \mathcal T_{1,d-1}$ as a natural generalization of $\mathcal B_2^\times$.
\end{remark2}

\subsection{Multipliers of simple cycles}

On the unit circle, a Blaschke product has many repelling periodic orbits or cycles.
Since all Blaschke products of degree 2 are quasisymmetrically conjugate on the unit circle, we can label the periodic orbits of $f \in \mathcal B_2$ by the corresponding periodic orbits of $z \to z^2$.

A cycle is {\em simple} if $f$ preserves its cyclic ordering. In this case, we say that $\langle \xi_1, \xi_2, \dots, \xi_q \rangle$ has {\em rotation number $p/q$} if $f(\xi_i) = \xi_{i+p\  (\text{mod } {q})}$. (For simple cycles, we prefer to index the points $\{\xi_i\} \subset S^1$ in counter-clockwise order, rather than by their dynamical order.)

\medskip

{\em Examples of cycles of degree 2 Blaschke products:}
\begin{itemize}
\item $(1,2)/3$ has rotation number 1/2,
\item $(1, 2, 4)/7$ has rotation number 1/3,
\item $(1,2,3,4)/5$ is not simple.
\end{itemize}
  
  \medskip
  
  In degree 2, for every fraction $p/q \in \mathbb{Q}/\mathbb{Z}$, there is a unique simple cycle of rotation number $p/q$. 
  We denote its multiplier by $m_{p/q} := (f^{\circ q})'(\xi_1)$. 
Since Blaschke products preserve the unit circle, $m_{p/q}$  is a positive real number (greater than 1). It is sometimes more convenient to work with $L_{p/q} := \log (f^{\circ q})'(\xi_1)$ which is an analogue of the length of a closed geodesic of a hyperbolic Riemann surface.

%% file: petals-flowers.tex
\section{Petals and Flowers}

In this section, we give an overview of petals, flowers and gardens. 
As suggested by the terminology, gardens are made of flowers, and flowers are made of petals. We first give a general definition of a garden, but then we 
specify to ``half-flower gardens'' which will be used throughout this work.

In fact, for a Blaschke product $f_a \in \mathcal B_2^\times$, we will construct infinitely many half-flower gardens  $\mathcal G_{[\gamma]}(f_a) $ -- one for
every outgoing homotopy class of simple closed curves $[\gamma] \in \pi_1(T_a, *)$. 
However, in practice, we use the garden  $\mathcal G(f_a) := \mathcal G_{[\gamma]}(f_a)$ 
associated to the shortest geodesic $\gamma$ in the flat metric on the torus. 
For parameters $a \in \mathcal B_{p/q}(C_{\sma})$, the shortest curve $\gamma$ is uniquely defined and has
 rotation number $p/q$.  It is precisely for this choice of half-flower garden that the estimate (\ref{eq:small-intersection}) holds.
For example, to study radial degenerations with $a \to 1$, we consider gardens where flowers have only one petal (see Figure \ref{fig:degenerating-to-1}),
while for other parameters, it is more natural to use gardens where the flowers have more petals (see Figure \ref{fig:more-colourful-gardens} below).

 \begin{figure}[h!]
  \centering
      \includegraphics[width=0.3\textwidth]{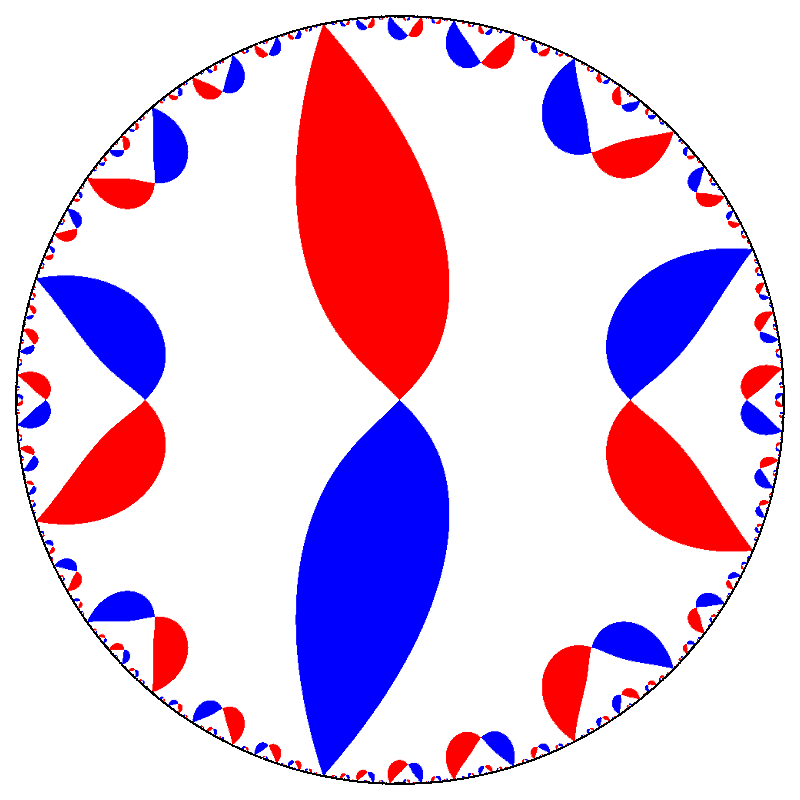}
      \qquad
      \qquad
            \includegraphics[width=0.3\textwidth]{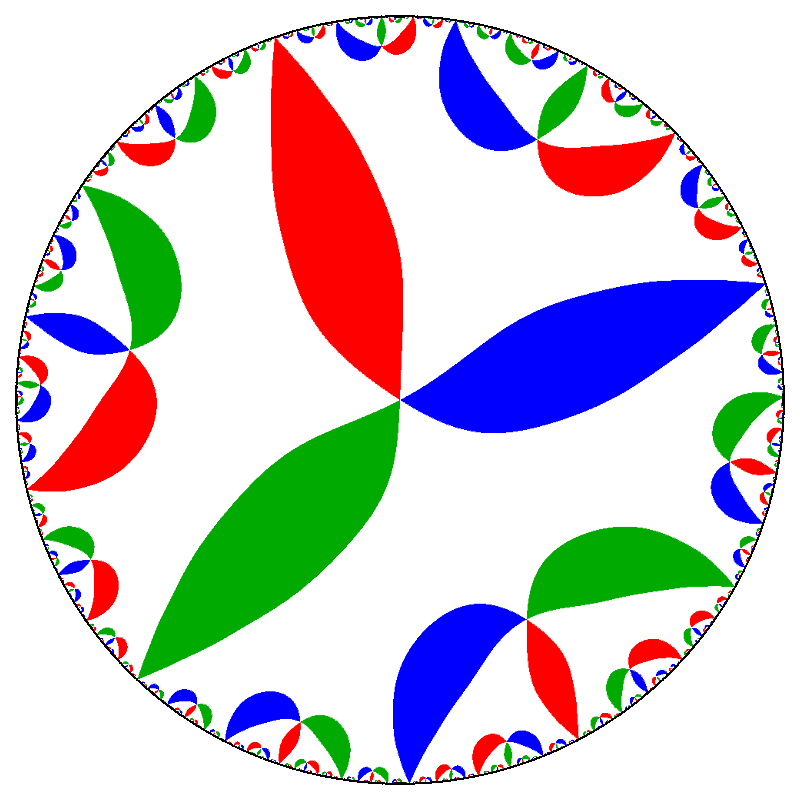}
  \caption{The gardens  $\mathcal G_{1/2}(f_{-0.6})$ and $\mathcal G_{1/3}(f_{0.66 \, \cdot \, e^{2\pi i /3}})$.}
   \label{fig:more-colourful-gardens}
\end{figure}

\subsection{Curves on the quotient torus}
Inside the first homotopy group 
$\pi_1(T_{a}, *) \cong \mathbb{Z} \oplus \mathbb{Z}$, there is a canonical generator $\alpha$ which is represented by 
counter-clockwise loops $\hat \varphi_a(\{z: |z| = \epsilon\})$ with $\epsilon > 0$ sufficiently small. By a {\em neutral} curve, we mean a curve whose homotopy class
in $\pi_1(T_{a}, *)$ is an integral power of $\alpha$. 
All non-neutral curves can be classified as either {\em incoming} or {\em outgoing}\,, depending on
their orientation: 
a curve $\gamma: \mathbb{R}/\mathbb{Z} \to T_{a}$ is {\em outgoing} if some (and hence every) lift $\gamma_i^* = \pi_a^{-1}\gamma_i$ in $\mathbb{C}^*$ satisfies
$$
\gamma_i^*(t+1) = (1/a)^q \cdot \gamma_i^*(t)  \qquad \text{for some }q \ge 1.
$$
In other words, $\gamma$ is outgoing if $\gamma_i^*(t) \to \infty$ as $t \to \infty$. A curve is {\em incoming} if the opposite holds, i.e.~if instead 
$\gamma_i^*(t) \to 0$ as $t \to \infty$.
 
A complementary (outgoing) generator $\beta$ is only canonically defined up to an integer multiple of $\alpha$. In terms of the basis $\{\alpha, \beta\}$, we say that an outgoing curve
homotopic to $(q-p)\alpha + p \beta$ has rotation number $p/q$. If we don't specify the choice of $\beta$, then $p/q$ is only well-defined modulo 1.

\subsection{Lifting outgoing curves}

Suppose $\gamma$ is a simple closed outgoing curve in $T_a^\times$ of rotation number $p/q$ mod 1. It has $q$ lifts to $\mathbb{C}^*$ under the
projection $\pi_a: \mathbb{C}^* \to T_a$, which we denote $\gamma^*_1, \gamma^*_2, \dots, \gamma^*_q$.
The curves $\gamma_i^*$ are ``spirals'' that join 0 to $\infty$. Each individual spiral is invariant
 under multiplication by $a^q$. We typically index the spirals so that multiplication by $a$ sends $\gamma^*_i$ to $\gamma^*_{i+p}$. 
Let $\tilde{\gamma_i} := \varphi_{a}^{-1}( \gamma_i^*)$ be (further) lifts in the unit disk emanating from the attracting fixed point.

\begin{lemma}
\label{ppc-lemma}
Suppose $\gamma$ is a simple closed outgoing curve in $T_a^\times$ of rotation number $p/q$. Then, 
$\tilde \gamma_i$ joins the attracting fixed point at the origin to a repelling periodic point $\xi_i \in S^1$ of rotation number number $p/q$.
\end{lemma}

\begin{proof}
Pick a point $z_1$ on $\tilde \gamma_i$, and approximate $\tilde \gamma_i$ by the backwards orbit of $f^{\circ q}$: $z_1 \leftarrow z_2 \leftarrow \dots \leftarrow z_{n} \leftarrow \dots$ By the Schwarz lemma, the backwards orbit is eventually contained in $U_{1/2} = \{ z : 1 - \delta_c/2 \le |z| < 1\},$
 i.e.~$z_n \in U_{1/2}$ for $n \ge N$.
Since the Blaschke product is asymptotically affine, the hyperbolic distance $d_{\mathbb{D}}(z_n, z_{n+1})$  between successive points is bounded as it
cannot substantially grow for $n \ge N$.
The boundedness of the backward jumps forces the sequence $\{z_n\}$  to converge 
 to a repelling periodic point $\xi_i$ on the unit circle. The same argument shows that the hyperbolic length of the arc of $\tilde \gamma_i$
  from $z_{n}$ to $z_{n+1}$ is bounded, and therefore $\tilde \gamma_i$ itself must converge to $\xi_i$. Since $f(\tilde \gamma_i) = \tilde \gamma_{i+p}$, we have $f(\xi_i) = \xi_{i+p}$. Furthermore, since the lifts $\tilde \gamma_i \subset \mathbb{D}$
 are disjoint, the points $\{\xi_i\}$ are arranged in counter-clockwise order which means that the repelling periodic orbit $\langle \xi_1, \xi_2, \dots, \xi_q \rangle$ has rotation number $p/q$.
\end{proof}

\subsection{Definitions of petals and flowers}
\label{chap:general-gardens}

An annulus $A \subset T_{a}^\times$ homotopy equivalent in $T_{a}^\times$ to an outgoing geodesic of rotation number $p/q$ has $q$ lifts in the unit disk emanating from the origin. We call these lifts {\em petals} and denote them $\mathcal P^A_i$, with $i=1,2, \dots, q$. 
Each petal connects the attracting fixed point to a repelling periodic point. 
Naturally, the {\em flower} is defined as 
the union of the petals: $\mathcal F = \bigcup_{i=1}^q \mathcal P^A_i$. 
We refer to the attracting fixed point as the {\em $A$-point} of the flower and to the repelling periodic points as the {\em $R$-points}\,.
By construction, flowers are forward-invariant regions. The {\em garden} is the totally-invariant region obtained by taking the union of all the repeated pre-images of the flower:
$${\mathcal G} = \bigcup_{n = 0}^\infty f_{a}^{-n}(\mathcal F).$$
We refer to the
iterated pre-images of petals and flowers as {\em pre-petals} and {\em pre-flowers} respectively.
In degree 2, a flower has two pre-images: itself and an {\em immediate pre-flower} which we denote $\mathcal F_*$ for convenience.
Each pre-flower has two proper pre-images.
We define the $A$ and $R$ points of pre-flowers as the pre-images of the $A$ and $R$ points of the flower.
 We typically label a pre-petal by its $R$-point and a pre-flower by its $A$-point.

\subsection{Half-flower gardens}
We now construct the special gardens that will be used in this work.
For this purpose, observe that an outgoing homotopy class $[\gamma] \in \pi_1(T_a, *)$ determines a foliation of the quotient torus $T_a$ by parallel lines, which are closed  geodesics in the flat metric on $T_a$. 
Explicitly, we can first 
 foliate the punctured  plane $\mathbb{C}^*$ by the logarithmic spirals % that are invariant under multiplication by $a^q$:
$$
\gamma^*_{\theta} := \{e^{t\log a^q} \cdot e^{i\theta} : t \in [-\infty, \infty) \}, \qquad 0 \le \theta < 2\pi,
$$
and then quotient out by  $(\cdot\ \!a)$.
The branch of $\log a^q$ is chosen so that $\pi_a(\gamma_{\theta}^*) \in [\gamma]$.
Note that since each individual spiral is only invariant under $(\cdot\ \!a^q)$, a single line on the quotient torus $T_a$ corresponds to $q$ equally-spaced spirals in $\mathbb{C}^*$. Therefore, $T_a$ is foliated by the parallel lines 
$\gamma_\theta := \pi_a(\gamma_{\theta}^*) $ with $0 \le \theta < 2\pi/q.$

For a Blaschke product $f_a \in \mathcal B_2^\times$, the quotient torus $T_a^\times$ has one puncture. Let $A^1 = T_a \setminus \gamma_{\theta_c}$ be the
complement of the ``singular line'' that passes through this puncture. 
For $0 < \alpha \le 1$, let $A^\alpha \subset A^1$ be the middle round annulus with
$
\Area(A^\alpha) / \Area(A^1) = \alpha.
$
By the construction of Section \ref{chap:general-gardens}, the annulus $A^1$ defines a system of petals
$
\mathcal P^1_i
$, $i = 1, 2,\dots, q$, 
which we calls {\em whole petals}\,.
Similarly, an {\em $\alpha$-petal} $\mathcal P_i^\alpha$ is defined as a petal constructed using
 the annulus $A^\alpha \subset T_a^\times$.  By default, we take $\alpha = 1/2$ and 
 write $\mathcal P_i = \mathcal P_i^{1/2}$.
We define the half-flower $\mathcal F$ as the union of all the half-petals. 

Alternatively, one can describe whole petals and half-petals in terms of linearizing rays. 
 A {\em linearizing ray}\,, or a {\em linearizing spiral}\, if $a \notin (0,1)$, is defined as the pre-image $\tilde \gamma_\theta := \varphi_a^{-1}(\gamma_\theta^*)$, $0 \le \theta\le 2\pi$ emanating from the attracting fixed point.
If a whole petal $\mathcal P^1$ consists of linearizing rays with arguments in $(\theta_1,\theta_2) = (\frac{x-y}{2}, \frac{x+y}{2})$, 
then the associated {\em $\alpha$-petal} $\mathcal P^\alpha$ is the union of the linearizing rays with arguments in $(\frac{x-\alpha y}{2}, \frac{x+\alpha y}{2})$.

\medskip

\noindent \textbf{Convention.} In the rest of the paper, we use this system of flowers. When working with $a \approx e(p/q)$, we let $\mathcal F = \mathcal F_{p/q}$ denote the flower constructed from a foliation of the quotient torus by $p/q$-curves, arising from the choice of $\log a^q \approx \log 1 = 0$. 

\begin{remark2}
One can similarly define petals and flowers similarly for Blaschke products of degree $d \ge 3$: 
Call a line $\gamma_\theta \subset T_{\mathbf{a}}$ {\em regular} if it is contained in $T_{\mathbf{a}}^\times$ and {\em singular} if it passes through a puncture. The singular lines partition
$T_{\mathbf{a}}$ into annuli, the lifts of which we call {\em whole petals}\,. The number of $(p/q)$-cycles of whole petals is at most $d-1$, but there could be less if several critical points lie
on a single line. %Half-petals and half-flowers can be defined similarly to the degree $d=2$ case.
%One difficulty that arises for Blaschke products of degree $d \ge 3$ is that petals can ``bifurcate''. 
\end{remark2}

%% file: quasiconformal-deformations.tex
\section{Quasiconformal Deformations}

In this section, we describe the Teichm\"uller metric on $\mathcal B_2^\times$ and define the half-optimal Beltrami coefficients which are supported on the half-flower gardens from the previous section. We also discuss pinching deformations.

For a Beltrami coefficient $\mu$ with $\|\mu\|_\infty < 1$, let $w_\mu$ be the quasiconformal map fixing 0, 1, $\infty$ whose dilatation is $\mu$.
Given a rational map $f(z) \in \Rat_d$, an invariant Beltrami coefficient $\mu \in M(S^2)^{f}$ 
defines a (possibly trivial) tangent vector in $T_f \Rat_d$ represented by the path  $f_{t} = w_{t \mu} \circ f \circ (w_{t \mu})^{-1}$,\, $t \in (-\epsilon, \epsilon)$.

If $\mu \in M(\mathbb{D})$, one can also consider the symmetrized version $w^\mu$ which is the quasiconformal map that has dilatation
 $\mu$ on the unit disk and is symmetric with respect to inversion in the unit circle.
 For a Blaschke product $f \in \mathcal B_d$ and a Beltrami coefficient $\mu \in M(\mathbb{D})^f$, the symmetric deformation
 $$
 f_t = w^{t \mu} \circ f \circ (w^{t \mu})^{-1}, \qquad t \in (-\epsilon, \epsilon), 
 $$
defines a path in $\mathcal B_d$. Note that while we use symmetric deformations to move around the space  $\mathcal B_d$, we use asymmetric deformations $w_{t\mu^+} \circ f \circ (w_{t\mu^+})^{-1}$ to compute the Weil-Petersson metric as the definition   
 of $\|\mu\|_{\WP}$ involves $v(z) = \frac{d}{dt} \bigl |_{t=0} \, w_{t\mu^+}(z)$.
  
The formula for the variation of the multiplier of a fixed point of a rational map will play a fundamental role in this work:

\begin{lemma}[e.g.~Theorem 8.3 of \cite{IT}] 
\label{attr-torus-multiplier}
Suppose $f_0(z)$ is a rational map with a fixed point at $p_0$ which is either attracting or repelling, and $\mu \in M(S^2)^{f_0}$.
Then, $f_t = w_{t\mu} \circ f_0 \circ (w_{t\mu})^{-1}$ has a fixed point at $p_t = w_{t\mu}(p_0)$ and
\begin{equation}
\label{eq:varmult}
\frac{d}{dt} \biggl |_{t=0} \log f_t'(p_t) = \pm \frac{1}{\pi}  \cdot \int_{T_{p_0}} \frac{\mu(z)}{z^2} \cdot |dz|^2
\end{equation}
where $T_{p_0}$ is the quotient torus at $p_0$. The sign is ``$\,+$'' in the repelling case and ``\,$-$'' in the attracting case.
\end{lemma}

\subsection{Teichm\"uller metric}

As noted in the introduction, $\mathcal T_{1,1}$ is the universal cover of $\mathcal B_2^\times$ 
since one has an identification of the tangent spaces $T_{f_a}\mathcal B_2^\times \cong T_{T_a} \mathcal T_{1,1}$. The Teichm\"uller metric
on $\mathcal B_2^\times$ makes this correspondence a local isometry.
More precisely, for a Beltrami coefficient $\mu \in M(\mathbb{D})^{f_a}$ representing a tangent vector in $T_{f_a}\mathcal B_2^\times$, 
$$
\|\mu\|_{T(\mathcal B_2^\times)} := \|(\hat{\varphi}_a)_*\mu\|_{T(\mathcal T_{1,1})}.
$$
A well-known result of Royden says that the Teichm\"uller metric on $\mathcal T_{1,1}$ is equal to the Kobayashi metric; therefore,
 the same is true for the Teichm\"uller metric on $\mathcal B_2^\times \cong \mathbb{D}^*$.
 Explicitly, the Teichm\"uller metric on $\mathcal B_2^\times$ is $\frac{|da|}{|a|\log|a|^2}$.

Lemma \ref{attr-torus-multiplier} distinguishes a one-dimensional subspace of Beltrami coefficients in $M(\mathbb{D})^{f_a}$, namely ones of the form
$\mu_\lambda = {\varphi_a^*}(\lambda \cdot (w/\overline{w}) \cdot (d\overline{w}/dw))$ with $\lambda \in \mathbb{C}$. We refer to these coefficients as {\em optimal} Beltrami coefficients. Here, ``optimal'' is short for ``multiplier-optimal'' which refers to the fact that $\mu_\lambda$ maximizes the absolute value of $(d/dt)|_{t=0} \log a_t$ out of all Beltrami coefficients with $L^\infty$-norm $|\lambda|$.

For a tangent vector $\mathbf{v} \in T_{T_a^\times}\mathcal T_{1,1}$, the {\em Teichm\"uller coefficient} $\mu_{\mathbf{v}}$ associated to $\mathbf{v}$ is the unique Beltrami coefficient of minimal $L^\infty$ norm which represents $\mathbf{v}$.
It is well-known that Teichm\"uller coefficients have the form $\lambda \, \overline{q}/|q|$
with $q \in Q(T_a^\times)$, where $Q(T_a^\times)$ is the space of integrable holomorphic quadratic differentials on the punctured torus $T_a^\times$. 
 In particular, $\|\mu_{\mathbf{v}}\|_T =  \sup_{\|q\|_T = 1} \bigl |\int_{T_a^\times} \mu q \bigr | = \|\mu_{\mathbf{v}}\|_\infty$.

Since the quotient torus $T_a^\times$ associated to a degree 2 Blaschke product $f_a \in \mathcal B_2^\times$ has one puncture, $Q(T_a^\times)$ is one-dimensional.
If we represent $T_a^\times \cong \mathbb{C}^*/(\cdot\,a)$, then $Q(T_a^\times)$ is spanned by $(\pi_a)_*(dw^2/w^2)$. 
Thus, in degree 2, the notions of Teichm\"uller coefficients and optimal coefficients agree.
 
\begin{remark2}
For a Blaschke product $f_{\mathbf{a}} \in \mathcal B_d^\times$ of degree $d \ge 3$, the quotient torus has $d-1 \ge 2$ punctures, and so $Q(T_{\mathbf{a}}) \subsetneq Q(T_{\mathbf{a}}^\times)$. Therefore, optimal Beltrami coefficients represent only a complex 1-dimensional set of directions in $T_{T_{\mathbf{a}}^\times}\mathcal T_{1,d-1}$.
In particular, to understand the Weil-Petersson metric on spaces of Blaschke products of higher degree, one would need to study 
other deformations.
\end{remark2}

Given an optimal Beltrami coefficient $\mu_\lambda$ and a half-flower garden $\mathcal G(f_a)$, we define the {\em half-optimal Beltrami coefficient} 
as $\mu_\lambda \cdot \chi_{\mathcal G}$. 

\begin{lemma}
\label{half-speed}
The half-optimal Beltrami coefficient $\mu \cdot \chi_{\mathcal G}$ is half as effective as the optimal Beltrami coefficient $\mu$,  i.e.~the map
 $f_{t}(\mu \cdot \chi_{\mathcal G}) := w^{t \mu \cdot \chi_{\mathcal G}} \circ f_{0} \circ (w^{t \mu \cdot \chi_{\mathcal G}} )^{-1}$
  is conformally conjugate to $f_{\tilde t}(\mu) := w^{\tilde t \mu} \circ f_{0} \circ (w^{\tilde t \mu})^{-1}$
where $\tilde t$ is chosen so that $d_{\mathbb{D}}(0,t) = 2 d_{\mathbb{D}}(0,\tilde t)$.
\end{lemma}

\subsection{Pinching deformations}
\label{sec:pinching-coefficients}

A closed torus $X = X_\tau = \mathbb{C} / \langle 1, \tau \rangle$, $\tau \in \mathbb{H}$, carries a natural flat metric which is unique up to scale. To study lengths of curves on $X$, we normalize the total area to be 1.
Given a slope $p/q \in \mathbb{Q} \cup \{ \infty \}$, let $\gamma_{p/q} \subset X$ denote the Euclidean geodesic obtained by projecting
$(\tau - p/q) \cdot \mathbb{R}$ down to $X$. We define the {\em pinching deformation} (with respect to $\gamma_{p/q}$) as the geodesic in $\mathcal T_{1} \cong \mathbb{H}$ which joins $\tau$ to $p/q$. We further define the {\em pinching coefficient} $\mu_{\pinch} \in M(X)$ as the Teichm\"uller coefficient which represents the unit tangent vector in the direction of this geodesic.
Intrinsically, the pinching deformation is ``the most efficient deformation'' that shrinks the Euclidean length of $\gamma_{p/q}$. More precisely, $X_t$ is the marked Riemann surface with $d_T(X, X_t) = \frac{1}{2} \log \frac{t+1}{t-1}$
for which $L_{X_t}(\gamma)$ is minimal, where $d_T$ is the Teichm\"uller distance in $\mathcal T_{1}$.

One can also define pinching deformations for annuli: given an annulus $A = A_0$, the pinching deformation $(A_t)_{t \ge 0}$ is the deformation for which the modulus of $A_t$ grows as quickly as possible. For the annulus $A_{r,R} := \{z : r < |z| < R \}$, the pinching deformation is given by the Beltrami coefficients 
\begin{equation}
t \cdot \mu_{\pinch} = t \cdot (w/\overline{w}) \cdot (d\overline{w}/dw), \quad t \in [0,1).
\end{equation}
With these definitions, the operation of ``pinching a torus $X$ with respect to a Euclidean geodesic $\gamma$'' is the same as ``pinching the annulus $A = X \setminus \gamma$.'' Indeed, the modulus of $X_\tau \setminus \gamma_{p/q}$ is just
\begin{equation}
\cmod(X_\tau \setminus \gamma_{p/q}) \, = \, \frac{\Area X_\tau}{L_{X_\tau}(\gamma_{p/q})^2} \, = \, \left\{ 
\begin{array}{lr} \frac{|\im \tau|}{|q\tau - p|^2}, & \text{if }p/q \ne \infty, \\ |\im \tau|, & \text{if }p/q = \infty. \end{array} \right.
\end{equation}
The above formula appears in \cite[Section 5]{McM-cyc}, although McMullen normalizes the area of $X_\tau$ to be $|\im \tau|$. The modulus of course is independent of the normalization.

\newpage

%% file: incompleteness-special.tex
\section{Incompleteness: Special Case}

\label{chap:basic-incompleteness}

In this section, we show that the Weil-Petersson metric on $\mathcal B_2$ is incomplete as we 
take $a \to 1$ along the real axis. 
As noted in the introduction, to show the estimate $\omega_B/\rho_{\mathbb{D}^*} \lesssim (1-|a|)^{1/4}$ on $(1/2,1]$, it suffices to prove:

 \begin{theorem}
 \label{simple-incompleteness}
For a Blaschke product $f_a \in \mathcal B_2$ with $a \in [1/2, 1)$, we have
  \begin{equation}
  \label{eq:lam-bounds}
\limsup_{r \to 1} \ |  {\mathcal G}(f_a) \cap S_r  | = O \bigl (\sqrt{1-|a|} \bigr).
\end{equation}
\end{theorem}

We will deduce Theorem  \ref{simple-incompleteness} from:

 \begin{theorem}
\label{qg-ps}
For a Blaschke product $f_a \in \mathcal B_2$ with $a \in [1/2, 1)$, 
\begin{enumerate}
\item[$(a)$]  Every pre-petal lies within a bounded hyperbolic distance of a geodesic segment. 
\item[$(b)$] The hyperbolic distance between any two pre-petals exceeds $d_{\mathbb{D}}(0,a) - O(1)$.
\end{enumerate}
 \end{theorem}

One curious feature of hyperbolic geometry is that {\em a horocycle connecting two points is exponentially longer than the geodesic.} Indeed, if $-x+iy, x+iy \in \mathbb{H}$, then the hyperbolic length of the horocycle joining them is $2(x/y)$ while
the geodesic length is only 
$
\int_{\theta}^{\pi-\theta} \frac{dt}{\sin t} = 2 \log(\cot(\theta/2))
$ where $\cot\theta = x/y$. As $\cot \theta \approx 1/\theta$ for $\theta$ small, this is  approximately $2 \log(2 \cdot x/y)$.
With this in mind, we argue as follows: 

\begin{proof}[Proof of Theorem \ref{simple-incompleteness}]
By part $(a)$ of Theorem \ref{qg-ps}, the hyperbolic length of the intersection of $S_r$ with any single pre-petal is $O(1)$. By part $(b)$ of Theorem \ref{qg-ps}, whenever the circle $S_r$ intersects a pre-petal, an arc of hyperbolic 
length $O \bigl (\sqrt{1-|a|} \bigr)$ is disjoint from the other pre-petals.
Therefore, only the $O \bigl (\sqrt{1-|a|} \bigr )$-th part of $S_r$ can be covered by pre-petals. 
\end{proof}

\subsection{Quasi-geodesic property}
\label{sec:qg-property}

% We first verify the quasi-geodesic property for petals:

\begin{lemma} 
\label{qg-for-petals}
For $a \in [1/2,1)$, the petal $\mathcal P(f_a)$ lies within a bounded hyperbolic neighbourhood
of a geodesic ray. 
\end{lemma}

\begin{proof}
By symmetry, the linearizing ray $\tilde \gamma_0 = \varphi_a^{-1}((0,\infty))$ is the line segment $(0,1)$ which happens to be a geodesic ray. We therefore need to show that the petal $\mathcal P(f_a) = \varphi_a^{-1}(\{\re z > 0\})$ lies within a bounded hyperbolic neighbourhood
of $\tilde \gamma_0$. Suppose $z \in  \mathcal P(f_a)$ lies outside a small ball $B(0, \delta)$. Let $F$ be the fundamental domain bounded by $\{\zeta : |\zeta|=\delta\}$ and
its image under $f_a$.
Under iteration, $z$ eventually lands in $F$, e.g.~$z_0 = f_a^{\circ N}(z) \in F$, with $\lim_{n \to \infty} \arg f^{\circ n}(z) \in (-\pi/2,\pi/2)$.
On the other hand, the limiting argument of the critical point $\lim_{n \to \infty} \arg f^{\circ n}(c) = \pi$ since the forward orbit of the critical point is contained in the segment $(-1,0)$.
Therefore, we can pick a point $x_0 \in \tilde \gamma_0$ for which $d_{\Omega}(z_0,x_0) = d_{T_a^\times}(\pi_a(z_0), \pi_a(x_0))= O(1)$.
Let $x = f^{-N}(x_0)$ be the $N$-th pre-image of $x_0$ along $\tilde \gamma_0$. Clearly, 
\begin{equation}
 d_{\mathbb{D}}(z, x) \le d_{\Omega}(z,x) = d_{T_a^\times}(\pi_a(z_0), \pi_a(x_0)) = O(1).
\end{equation}
This completes the proof.
\end{proof}

\subsection{The structure lemma}

To establish the quasi-geodesic property for pre-petals, we show the ``structure lemma'' which says that the pre-petals are nearly-affine copies of the immediate pre-petal, while
$f: \mathcal P_{-1} \to \mathcal P$ is approximately the involution about the critical point, i.e.~$f|_{\mathcal P_{-1}} \approx m_{0 \to c} \circ (-z) \circ m_{c \to 0},$ where $m_{0 \to c} = \frac{z+c}{1+\overline{c}z}$ and $m_{c \to 0} = \frac{z-c}{1-\overline{c}z}$.
For a Blaschke product $f$, its {\em critically-centered version} is given by
$$
\tilde f = m_{c \to 0} \circ f \circ m_{0 \to c}.
$$
Naturally, the petals and pre-petals of $\tilde f$ are defined as the images of petals and pre-petals of $f$ under $m_{c \to 0}$.

\begin{lemma}[Structure lemma] 
\label{similarity-lemma}
For $a \in [1/2,1)$ on the real axis,
 \begin{enumerate} 
 \item[{\em (i)}] The critically-centered petal $\tilde{\mathcal P} \subset B \bigl (1, \const \cdot\, \sqrt{1-|a|} \bigr )$.
\item[{\em (ii)}] The immediate pre-petal $\mathcal P_{-1} \subset B \bigl (-1, \const \cdot\, (1-|a|) \bigr )$.
\end{enumerate}
\end{lemma}

\begin{proof}
Part (i) follows from Lemma \ref{qg-for-petals} as $m_{c \to 0}\bigl((0,1) \bigr ) = (-c, 1)$. To pin down the size and location of the immediate pre-petal, we use the fact that
for a degree 2 Blaschke product, $c$ is the hyperbolic midpoint of $[0,-a]$.
This implies that in the critically-centered picture, the $A$-point of the petal is $m_{c \to 0}(0) = -c$ while the $A$-point of the immediate pre-petal is $m_{c \to 0}(-a) = c$.
Therefore, by Koebe's distortion theorem, $\tilde{\mathcal P}_{-1} \subset B \bigl (-1, \const \cdot\, \sqrt{1-|a|} \bigr )$.
Part (ii) follows by applying $m_{0 \to c}$ to the last statement.
\end{proof}

 \begin{figure}[h!]
 \label{fig:cc-versions}
  \centering
      \includegraphics[width=0.3\textwidth]{images/a08.png}
      \qquad
      \qquad
            \includegraphics[width=0.3\textwidth]{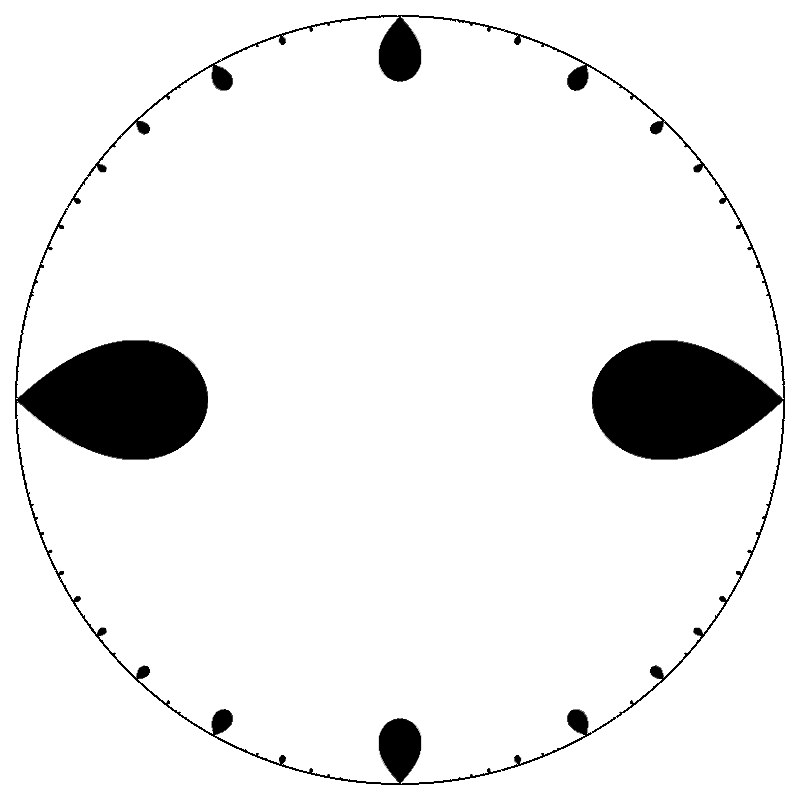}
  \caption{Half-petal families for the Blaschke products $f_{0.8}$ and $\tilde f_{0.8}$.}
\end{figure}

\subsection{Petal separation}
\label{sec:petal-separation}

We can now prove that the petals are far apart:

\begin{proof}[Proof of part $(b)$ of Theorem \ref{qg-ps}]
Since the petal $\mathcal P$ is contained in a bounded hyperbolic neighbourhood of $(0,1)$ and the immediate pre-petal $\mathcal P_{-1}$ is contained in a bounded hyperbolic neighbourhood
of $(-1, -a)$, it follows that 
$$
d_{\mathbb{D}}(\mathcal P, \mathcal P_{-1}) = d_{\mathbb{D}}(0,-a) - O(1).
$$
By the Schwarz lemma, given two pre-petals
$\mathcal P_{\zeta_1}$ and $\mathcal P_{\zeta_2}$ with $f^{\circ n_1}(\zeta_1) = f^{\circ n_2}(\zeta_2) =1$ and $n_1 \ne n_2$ (say $n_1>n_2$),
$$d_{\mathbb{D}}(\mathcal P_{\zeta_1}, \mathcal P_{\zeta_2}) 
\ge d_{\mathbb{D}} \Bigl (f^{\circ(n_1-1)}(\mathcal P_{\zeta_1}),f^{\circ (n_1-1)}(\mathcal P_{\zeta_2}) \Bigr )
 \ge d_{\mathbb{D}} (\mathcal P_{-1}, \mathcal P_1).
$$
To complete the proof, it suffices to show that pre-petals $\mathcal P_{\zeta_1}$ and $\mathcal P_{\zeta_2}$ are far
apart in the case that they have a common parent, e.g.~when $f(\zeta_1)=f(\zeta_2)=\zeta$. 
We prove this using a topological argument. Observe that $-1$ and $1$ separate the unit circle in two arcs, each of which is mapped to $S^1 \setminus \{1\}$ by $f_a$.
Therefore, any path in the unit disk connecting $\mathcal  P_{\zeta_1}$ and $\mathcal  P_{\zeta_2}$ must intersect the line segment $(-1,1) \subset \overline{\mathcal P_1^1} \cup \overline{\mathcal P_{-1}^1}$.

However, we already know
that
the distance between  $\mathcal P_{\zeta_i}$ to either 
 $\mathcal P_{1}$ and  $\mathcal P_{-1}$ is greater than $d_{\mathbb{D}}(0,a)-O(1)$ which tells us that the hyperbolic $(\frac{1}{2} \cdot d_{\mathbb{D}}(0,a)-O(1))$-neighbourhood of $(-1,1)$ is disjoint from $\mathcal P_{\zeta_1}$
and $\mathcal P_{\zeta_2}$. This completes the proof.
\end{proof}

%% file: renewal-theory.tex
\section{Renewal Theory}
\label{chap:renewal-theory}

In this section, we show that for a Blaschke product other than $z \to z^d$, the integral average (\ref{eq:integral-average0})
defining the Weil-Petersson metric converges. The proof is based on renewal theory, which is the study of the distribution of repeated pre-images of a point. In the context of hyperbolic dynamical systems, this has been developed by Lalley \cite{La}. We  apply his results to Blaschke products, 
thinking of them as maps from the unit circle to itself. Using an identity for the Green's function, we extend renewal theory to points inside the unit disk. 
Renewal theory will also be instrumental in giving bounds for the Weil-Petersson metric.

For a point $x$ on the unit circle, let $n(x, R)$ denote the number of repeated pre-images $y$ (i.e.~$f^{\circ n}(y) = x$ for some $n \ge 0$) for which
$\log |(f^{\circ n})'(y)| \le R$. Also consider the probability measure $\mu_{x,R}$ on the unit circle which gives equal mass to each of the $n(x, R)$ pre-images. 
We show:

\begin{theorem}
 \label{renewal-theory}
For a Blaschke product $f \in \mathcal B_d$ other than $z \to z^d$,
\begin{equation}
\label{eq:renewal}
n(x, R) \sim \frac{e^R}{\int \log |f'|dm} \qquad \text{as} \quad R \to \infty.
\end{equation}
Furthermore, as $R \to \infty$, the measures $\mu_{x,R}$ tend weakly to the Lebesgue measure.
\end{theorem}

For a point $z \in \mathbb{D}$, let $\mathcal N(z, R)$ be the number of repeated pre-images of $z$ that lie in the ball centered at the origin of hyperbolic radius $R$. 

 \begin{theorem}
 \label{renewal-theory2}
Under the assumptions of Theorem \ref{renewal-theory}, we have
\begin{equation}
\mathcal N(z, R) \sim \frac{1}{2} \cdot \log\frac{1}{|z|} \cdot \frac{e^R}{\int \log |f'|dm} \qquad \text{as} \quad R \to \infty.
\end{equation}
As before, when $R \to \infty$, the $\mathcal N(z, R)$ pre-images become equidistributed on the unit circle with respect to the Lebesgue measure.
\end{theorem}

\subsection{Green's function}
\label{sec:greens-function}

Let $G(z) = \log\frac{1}{|z|}$ be the Green's function of the disk with a pole at the origin. It is uniquely characterized by three properties:
 \begin{enumerate}
 \item[(i)] $G(z)$ is harmonic on the punctured disk, 
 \item[(ii)] $G(z)$ tends to 0 as $|z| \to 1$, 
 \item[(iii)] $G(z) -  \log \frac{1}{|z|}$ is harmonic near $0$.
 \end{enumerate}

\begin{lemma}
\label{green-identity}
For a Blaschke product $f \in \mathcal B_d$, we have
\begin{equation}
\label{eq:green-identity}
\sum_{f(w_i)=z} G(w_i) = G(z), \qquad z \in \mathbb{D}.
\end{equation}
\end{lemma}
To prove Lemma \ref{green-identity}, it suffices to check that $\sum_{f(w_i)=z} G(w_i)$ also satisfies the three properties above. We leave the verification to the reader.
From equation (\ref{eq:green-identity}), it follows that the Lebesgue measure on the unit circle is invariant under $f$. Indeed, for a point $x \in S^1$, one can 
apply the lemma to $z = rx$ and take $r \to 1$ to obtain $\sum_{f(y)=x} |f(y)|^{-1} =1$.
(Alternatively, one can apply $\frac{\partial}{\partial z}$ to both sides of (\ref{eq:green-identity}) to obtain the somewhat stronger
statement $\sum_{f(w)=z} \frac{f(w)}{wf'(w)} =1$.)

In fact, the Lebesgue measure is ergodic. The argument is quite simple (see \cite{SS} or \cite{Ha}); for the convenience of the reader, we reproduce it here: 
given an invariant set $E \subset S^1$, form the harmonic extension
$u_E(z)$ of $\chi_E$. Since $\chi_{f^{-1}{E}} = \chi_E \circ f$, $u_E$ is a harmonic function in the disk which is invariant under $f$. But 0 is an attracting fixed point, so $u_E$ must actually be constant, which forces $E$ to have measure 0 or 1 as desired.
From the ergodicity of Lebesgue measure, it follows that conjugacies of distinct Blaschke products are not absolutely continuous.

\subsection{Weak mixing}

For the exceptional Blaschke product $z \to z^d$, the pre-images of a point $x \in S^1$ come in packets and so $n(x, R)$ is a step function.  Explicitly, 
$$n(x,R) = 1+d+d^2+\dots + d^{\lfloor \log R/\log d \rfloor}.$$ While $n(x, R)$ has exponential growth,
due to the lack of mixing, some values of $R$ are special. All other Blaschke products satisfy the required mixing property and  Theorem \ref{renewal-theory} follows from \cite[Theorem 1 and formula (2.5)]{La}.

 \begin{proof}[Sketch of proof of Theorem \ref{renewal-theory}.]
In the language of thermodynamic formalism, we must check that the potential
$\phi_f(x) = -\log|f'(x)|$ is non-lattice, i.e.~that there does not exist a bounded function $\gamma$ such that $\phi= \psi + \gamma - \gamma \circ f$ with $\psi$ valued in a discrete subgroup of $\mathbb{R}$. 
% (To be honest, in \cite{La}, this equation holds not on $S^1$ but on the shift space $\Sigma = \{0,1,\dots,d-1\}^\mathbb{N}$ that codes the dynamics of $f$.)
% 
To the contrary, if such a $\psi$ exists, then the multiplier spectrum $$\{ \log (f^{\circ n})'(\xi) : f^{\circ n}(\xi) = \xi \}$$ is contained in a discrete subgroup of $\mathbb{R}$.
Following the proof of  \cite[Proposition 5.2]{PP}, we see that there exists a function $w \in C^\alpha(\Sigma)$ %, $0 < \alpha < 1$ 
satisfying 
\begin{equation}
\label{eq:weak-mixing}
w(f(x)) = e^{ia\phi_f(x)} w(x), \qquad \text{for some } a \in \mathbb{R} \setminus \{0\}.
\end{equation}
Here, $\Sigma = \{0,1,\dots,d-1\}^\mathbb{N}$ is the shift space which codes the dynamics of $f$ on the unit circle.
However, if we work directly on the unit circle and repeat the proof of  \cite[Proposition 4.2]{PP}, we obtain a function $w \in C^\alpha(S^1)$ satisfying (\ref{eq:weak-mixing}).
 Since $w(x)$ is non-vanishing and has constant modulus, we can scale it by a constant if necessary so that $|w(x)| = 1$.
 By comparing the topological degrees of both sides of  (\ref{eq:weak-mixing}), we see that the topological degree of $w$ is
  0. In particular, $w$ admits a continuous branch of logarithm.

 If  $w(x) = e^{i v(x)}$ then $v \circ f = a \cdot \phi_f + v + 2\pi k$ for some constant $k \in \mathbb{Z}$. 
 Therefore, $\phi_f \sim 2\pi k / a$ is cohomologous to a constant.
This tells us that the Lebesgue measure $m$ must also be the measure of maximal entropy. However, the measure of the maximum entropy 
is a topological invariant, thus if we have a conjugacy $h$ between $z^d$ and $f(z)$, then the measure of the maximal entropy is $h_* m$.
However, we know that the conjugacies of distinct Blaschke products are {\em not} absolutely continuous, therefore, we must have $f(z) = z^d$.
 \end{proof}

   \subsection{Computation of entropy}
   
    Since the dimension of the unit circle is equal to 1, the entropy $h(f, m)$ of the Lebesgue measure coincides with the Lyapunov exponent $\frac{1}{2\pi} \int \log |f'(e^{i\theta})| d\theta$. We may compute the latter quantity using Jensen's formula:
\begin{lemma}
\label{jensen}
If $a = f_{\mathbf{a}}'(0) \ne 0$, the entropy of the Lebesgue measure for the Blaschke product $f_{\mathbf{a}}(z)$ with critical points $\{c_i\}$ and zeros $\{z_i\}$ is given by
\begin{equation}
\frac{1}{2\pi} \int \log |f_{\mathbf{a}}'(e^{i\theta})| d\theta = \sum_{\text{cp}} G(c_i) - G(a) = \sum_{\text{cp}} G(c_i) - \sum_{\text{zeros}} G(z_i).
\end{equation}
\end{lemma}

In particular, for degree 2 Blaschke products, as $a$ tends to the unit circle, the entropy $h(f_a, m) \sim 1-|c| \sim \sqrt{2(1-|a|)}$.
\subsection{Laminated area}

For a measurable set $E$ in the unit disk,
let $\hat{E}$ denote its {\em saturation} under taking pre-images, i.e.~$
\hat{E} = \{ \zeta : f^{\circ n}(\zeta) \in E \text{ for some }n \ge 0\}.
$
For a saturated set $\hat E$, we define its {\em laminated area} as
$ \mathcal A(\hat E) = \lim_{r \to 1^-} \frac{1}{2\pi} |E \cap S_r |$
and say that ``$E$ subtends the $ \mathcal A(\hat E)$-th part of the lamination.'' 
By Koebe's distortion theorem (see Section \ref{chap:distortion-theorem}), 
we have the following useful estimate:

\begin{lemma}
\label{laminated-area-in-linearity-zone}
Suppose $E$ is a subset of $U_t := \{ z : 1 - t \cdot \delta_c \le |z| < 1\}$ with $t < 1/2$. If $E$ is is disjoint from all of its pre-images, then
\begin{equation}
\mathcal A(\hat E) \approx_{t}  \frac{1}{2\pi \, h(f_a, m)} \int_E \frac{1}{1-|z|} \cdot |dz|^2.
\end{equation}
\end{lemma}

(The notation ``$A \approx_\epsilon B$'' means that $|A/B -1| \lesssim \epsilon$.)

\begin{proof}
By breaking up the set $E$ into little pieces, we may assume that $E \subset B(x, t)$ for some $x \in S^1$.
We claim that $\int_E \frac{1}{1-|z|} \cdot |dz|^2 \approx_{t} \int_{f^{-n}(E)} \frac{1}{1-|z|} \cdot |dz|^2$, uniformly in $n \ge 0$. By Lemma \ref{area-distortion-lemma},
for each $n$-fold pre-image $E_y$
of $E$,  with $f^{\circ n}(y) = x$, we have
$$
\int_{E_y} \frac{1}{1-|z|} \cdot |dz|^2  \approx_{t} |(f^{\circ n})'(y)|^{-1} \cdot \int_{E} \frac{1}{1-|z|} \cdot |dz|^2.
$$
 The claim follows in view of the the identity $\sum_{f^{\circ n}(y) = x}|(f^{\circ n})'(y)|^{-1} = 1$ (recall that the Lebesgue measure is invariant). Therefore, we may assume that $E \subset U_{t'}$ with $t' > 0$ arbitrarily small, i.e.~we can pretend that $f^{-1}$ is essentially affine.

By approximation, it suffices to consider the case when $E = \mathscr R$ is a ``rectangle'' of the form 
$$
\Bigl \{z : \ 1 - |z| \in \bigl (\delta, (1+\epsilon_1)\delta \bigr ), \ \arg z \in \bigl (\theta_0, \theta_0 + \epsilon_2 \delta \bigr ) \Bigr \}
$$
with $\epsilon_1, \epsilon_2$ small.
For $k$ large, the circle $S_{1-\delta/k} = \{z : |z| = 1 - \delta/k\}$ intersects $\approx \epsilon_1 k/h$ pre-images of $\mathscr R$. 
As the hyperbolic length of $S_{1-\delta/k}$ is $\sim 2\pi k/\delta$ and each pre-image has ``horizontal'' hyperbolic length $\approx \epsilon_2$, the laminated area
$
\mathcal A(\hat{\mathscr R}) \approx
 \frac{\epsilon_1 \epsilon_2}{2\pi h} \cdot \delta
$
 as desired.
\end{proof}

Recall from \cite{McM-wp} that a continuous function $h: \mathbb{D} \to \mathbb{C}$ is {\em almost-invariant} if for any $\epsilon > 0$, there exists $r(\epsilon) < 1$, so that
for any orbit $z \to f(z) \to \dots \to f^{\circ n}(z)$ contained in $ \{ z : r \le |z| < 1\}$, we have $|h(z) - h(f^{\circ n}(z))| < \epsilon$.

\begin{theorem}
\label{almost-invariant-lamination}
Suppose $f$ is a Blaschke product other than $z \to z^d$, and $h$ is an almost-invariant function. Then the limit \ 
$
\lim_{r \to 1^-} \frac{1}{2\pi} \int_{|z|=r}  h(z) d\theta
$
exists.
\end{theorem}

\begin{proof}[Proof]
Let $E$ be a backwards fundamental domain near the unit circle, e.g.~take $E = f^{-1}(B(0,s)) \setminus B(0,s)$ with $s \approx 1$. Split $E$ into many pieces on which $h$ is approximately constant. Applying Lemma \ref{laminated-area-in-linearity-zone} to each piece and summing over the pieces, we see that as $r \to 1$, $\frac{1}{2\pi} \int_{|z|=r}  h(z) d\theta$ oscillates
by an arbitrarily small amount. Therefore, the limit exists.
\end{proof}

Applying the above theorem with $h = |v'''/\rho^2|^2$, which is almost-invariant by Lemma \ref{qd-almostinvariant2}, gives:
\begin{corollary}
Given a Blaschke product $f \in \mathcal B_d$ other than $z \to z^d$, the limit in the definition of the Weil-Petersson metric (\ref{eq:integral-average0})
exists  for every vector field $v$ that is associated to a tangent vector $T_f \mathcal B_d$.
\end{corollary}

%% file: simple-cycles.tex
\section{Multipliers of Simple Cycles}

\label{chap:simple-cycles}

In this section, we study the behaviour of repelling periodic orbits of degree 2 Blaschke products with small multipliers.
Recall from Section \ref{sec:bp} that $L_{p/q}$ denotes the logarithm of the multiplier of the unique cycle that has rotation number $p/q$.
Given $\mu \in M(\mathbb{D})^{f_0}$ representing a vector in $T_{\mathcal B_2^\times}f_0$, let  $\dot L_{p/q}[\mu] := (d/dt)|_{t=0} \, L_{p/q}(f_t)$ where
we perturb $f_0$ using the symmetric deformation $f_t = w^{t \mu} \circ f \circ (w^{t \mu})^{-1}$, $t \in (-\epsilon, \epsilon)$.

Let $\mathcal B_{p/q}(\eta)$ be the  horoball in the unit disk of Euclidean diameter $\eta/q^2$ which rests on 
$e(p/q) \in S^1$ and $\mathcal H_{p/q}(\eta) = \partial \mathcal B_{p/q}(\eta)$ be its boundary horocycle. We show:

\begin{theorem}
\label{small-cycles}
There exists a constant $C_{\sma} > 0$ such that for a Blaschke product $f_a \in \mathcal B_2$ with $a \in \mathcal H_{p/q}(\eta)$ and $\eta < C_{\sma}$,
we have:
\begin{itemize}
\item[{\em (i)}] As $\eta \to 0^+$, $m_{p/q} - 1 \sim \eta/2.$
\item[{\em (ii)}] If $\gamma_{p/q} \subset T_{a}$ is the shortest curve in the quotient torus at the attracting fixed point  (which necessarily has
 rotation number $p/q$)  and
$\mu_{pinch} \in M(\mathbb{D})^{f_a}$ is the associated pinching coefficient with $\|\mu_{\pinch}\|_\infty =1$, then
$$
|\dot L_{p/q}[\mu]/L_{p/q}| \asymp 1.
$$
\end{itemize}
\end{theorem}
In other words, the gradient of $L_{p/q}$ is within a bounded factor of the maximal possible.
We now make some useful definitions.
Let $T_{p/q}$ denote the quotient torus associated to the repelling periodic orbit of rotation number $p/q$ and $T^{\intext}_{p/q} \subset T_{p/q}$ be the half of the torus which is associated to points inside the unit disk. Let $P^1_{p/q} \subset T^{\intext}_{p/q}$ be the footprint of $\mathcal F^1$ in $T^{\intext}_{p/q}$, i.e.~the part of $T^{\intext}_{p/q}$ filled by $\mathcal F^1$.
The footprint $P_{p/q}$ of $\mathcal F = \mathcal F^{1/2}$ is defined similarly.
To prove Theorem \ref{small-cycles}, we need:

\newpage

\begin{lemma}
\label{small-horoballs}
There exists $C_{\sma} > 0$ sufficiently small so that for $a \in \mathcal B_{p/q}(C_{\sma})$, %we have:
\begin{itemize}
\item[{\em (i)}] The footprint $P_{p/q}^1$ of the whole petal  contains a definite angle of opening at least $0.99 \, \pi$.
\item[{\em (ii)}] The footprint $P_{p/q}$ of the half-petal is contained in a central angle of $0.51 \, \pi$.
\end{itemize}
\end{lemma}
In turn,  Lemma \ref{small-horoballs} is proved by comparing the ``petal correspondence'' with the holomorphic index formula. The argument is
essentially due to McMullen, see \cite[Theorem 6.1]{McM-cyc}; however, we will spell out the details since we need slightly more information.

\subsection{Conformal modulus of an annulus}

We use the convention that the annulus 
$A_{r,R} := \{z : r < |z| < R \}$
 has modulus $\frac{\log(R/r)}{2\pi}$, which is the extremal length of the curve family
$\Gamma_\uparrow(A_{r,R})$ consisting of curves that join
the two boundary components of  $A_{r,R}$. 
We denote the dual curve family by $\Gamma_\circlearrowleft(A_{r,R})$, consisting of curves that separate the two boundary components. Then, $\lambda_{\Gamma_\uparrow(A)} \cdot \lambda_{\Gamma_\circlearrowleft(A)} = 1$.
 For background on extremal length and moduli of curve families, we refer the reader to \cite{GM}.

If $B \subset A$ is an essential sub-annulus of $A$, we say that $B$ is {\em round} in $A$ if the pair $(A, B)$ is conformally equivalent to a pair of concentric round annuli $(A_{r,R}, A_{r',R'})$ with $A_{r',R'} \subset A_{r,R}$. Alternatively, $B$ is round in $A$ if
the pinching deformations for $A$ and $B$ are compatible, i.e.~if $\mu_{\pinch}(B) = \mu_{\pinch}(A)|_B$.
\begin{lemma}%[see \cite{McM-cusps}] 
Suppose $S^* = \{ e^{i\theta} \cdot e^{\mathbb{R} \log \alpha} : \theta_1 < \theta < \theta_2\} \subset  \mathbb{C}^*$ where $|\alpha| > 1$ and a branch of \,$\log \alpha$ has been chosen.
Then the annulus 
\begin{equation}
\label{eq:mod-spiral}
S^* \ / \  \{z \sim \alpha z\}
\qquad \text{has modulus} \qquad
(\theta_2 - \theta_1) \, \re \Bigl(\frac{1}{\log \alpha}\Bigr ).
\end{equation}
\end{lemma}

Suppose $T^* \subset \mathbb{C}^*$ is a region bounded by two Jordan curves $\gamma_1, \gamma_2$ which are invariant under multiplication by $\alpha$,
with $|\alpha| > 1$.
By analogy with (\ref{eq:mod-spiral}), we define the {\em generalized angle} $\beta$ between $\gamma_1$ and $\gamma_2$ by the formula 
$\cmod (T^* \, / \,  \{z \sim \alpha z\}) = \beta  \, \re \bigl(\frac{1}{\log \alpha}\bigr ).
$
\subsection{Holomorphic index formula}

We now recall the statement of the holomorphic index formula. If $g(z)$ is a holomorphic map, the {\em index} of a fixed point $\zeta$ is defined as
\begin{equation}
\label{eq:holo-index-theorem}
I_\zeta := \frac{1}{2\pi i} \int_\gamma \frac{dz}{z - g(z)}
\end{equation}
where $\gamma$ is any sufficiently small counter-clockwise loop around $\zeta$.
If the multiplier $\lambda = g'(\zeta)$ is not 1, this expression reduces to $\frac{1}{1-\lambda}$.
By the residue theorem, one has:
\begin{theorem}[Holomorphic Index Formula]
Suppose $R(z)$ is a rational function and $\{\zeta_i\}$ are its fixed points. Then, $\sum I_{\zeta_i} = 1.$
\end{theorem}
For a Blaschke product $f \in \mathcal B_d$, the holomorphic index formula says that
\begin{equation}
\sum \frac{1}{r_i-1} =  \frac{1-|a|^2}{|1-a|^2}
\end{equation}
where the sum ranges over the repelling fixed points on the unit circle, and $a = f'(0)$ is the multiplier of the attracting fixed point.

\subsection{Petal correspondence}
Since a whole petal joins the attracting fixed point to a repelling periodic point, it provides a conformal equivalence between the annuli $A^1 \subset T_a^\times$ and $P^1_{p/q} \subset T_{p/q}$.
As there are $q$ whole petals at the attracting fixed point, %we have:
\begin{equation}
\label{petal-correspondence}
\frac{\beta}{\log m_{p/q}} = \re \frac{1}{q} \cdot \frac{2\pi}{\log(1/a^q)}
\end{equation}
where $\beta$ is the generalized angle representing the modulus of $\cmod P_{p/q}^1$.
Observe that the holomorphic index formula gives a lower bound on $m_{p/q}$:
\begin{equation}
\label{holo-index}
\frac{1}{m_{p/q}-1} \le \frac{1}{q} \cdot \frac{1-|a^q|^2}{|1-a^q|^2}.
\end{equation}

\begin{proof}[Proof of Lemma \ref{small-horoballs}]
Suppose $a \in \mathcal H_{p/q}(\eta)$. If $\eta > 0$ is small, then $a^q \in \mathcal H_1(\frac{\eta + \theta}{q})$ with $|\theta|$ small.
On this horocycle, $\re \frac{1}{\log(1/a^q)} \approx \frac{q}{ \eta + \theta}$
 while the Poisson kernel $\frac{1-|a^q|^2}{|1-a^q|^2} \approx \frac{2q}{\eta + \theta}$.
 Note that if $\eta > 0$ is small, equation  $(\ref{holo-index})$ forces $m_{p/q}$ to be close to 1, which in turn ensures
 that the ratio $\frac{\log m_{p/q}}{m_{p/q}-1}$ is close to 1. Comparing  $(\ref{petal-correspondence})$ and $(\ref{holo-index})$ like in \cite{McM-cyc}, we deduce
 that $\beta$ is close to $\pi$.
By the standard modulus estimates (see Lemmas  \ref{standard-modulus-estimates1}  and  \ref{standard-modulus-estimates2} below), it follows that the footprint $P^1_{p/q}$ must contain
an angle of opening close to $\pi$. They also show that the footprint of the half-petal $P_{p/q}$ is contained in a central angle of opening slightly greater than $\pi/2$.
\end{proof}

With preparations complete, we can now prove Theorem \ref{small-cycles}:

\begin{proof}[Proof of Theorem \ref{small-cycles}]
For (i), we plug $\beta \approx \pi$ into $(\ref{petal-correspondence})$ to obtain
$$
1/\log m_{p/q} \approx 2/\eta  \quad \text{or} \quad m_{p/q} \approx 1 + \eta/2.
$$
Part (ii) requires a bit more work. Since the footprint of the whole petal $P_{p/q}^1$ contains an angle of $>0.51\pi$, it is easy to construct an
invariant Beltrami coefficient which  effectively deforms the quotient torus of the repelling periodic orbit.
As $\mathcal B_2$ is one-dimensional, we see that for an optimal Beltrami coefficient $\mu$, we must have either
\begin{equation}
\label{eq:lower-bounds-mult}
|\dot L_{p/q}[\mu]/L_{p/q}| \asymp 1 \quad \text{or} \quad |\dot L_{p/q}[i \mu]/L_{p/q}| \asymp 1.
\end{equation}
We need to show that the first alternative holds when $\mu = \mu_{\pinch} \in M(\mathbb{D})^f$ is the optimal pinching coefficient built from the attracting torus.
As the dynamics of  $f^{\circ q}$ is approximately linear near a repelling periodic point,  $\mu = \mu_{\pinch}$ descends to a Beltrami coefficient $\nu \in M(T_{p/q})$, with
$\supp \nu \subset T_{p/q}^{\intext}$. 
Since $\mu|_{A^1}$ is the optimal pinching coefficient for $A^1$, $\nu|_{P_{p/q}^1}$ is the optimal pinching coefficient for the annulus ${P_{p/q}^1}$. 
By Lemma \ref{small-horoballs}, when $\eta > 0$ is small, the footprint $P_{p/q}^1$ takes up most of $T_{p/q}^{\intext}$, and 
since $T_{p/q}^{\intext}$ is a round annulus in $T^{p/q}$, $\nu$ is approximately equal to the optimal pinching coefficient for $T_{p/q}$ on $T_{p/q}^{\intext}$. 

When we consider deformations
$f^{t\mu}$ in the Blaschke slice, we use the Beltrami coefficient $\mu + \mu^+$, which corresponds to $\nu + \nu^+ \in M(T_{p/q})$. We see that $\nu + \nu^+ \in M(T_{p/q})$ is
approximately equal to the optimal pinching coefficient for $T_{p/q}$ (at least away from the trace of the unit circle in $T_{p/q}$). 
In other words,  pinching $T_a$ with respect to a $p/q$ curve has nearly the same effect as pinching $T_{p/q}$ with respect to a $0/1$ curve.
This gives
$|\dot L_{p/q}[\mu]/L_{p/q}| \asymp 1$.
\end{proof}

\subsection{Standard modulus estimates.}
For the  convenience of the reader, we state the standard estimates for moduli of annuli that we have used in the proofs of Lemma \ref{small-horoballs} and Theorem \ref{small-cycles}.

\begin{lemma}
\label{standard-modulus-estimates1}
Suppose $A = A_{r,R}$ and $B \subset A$ is an essential sub-annulus. For any $\epsilon > 0$, there exists $\delta > 0$ and $m_0 > 0$ such that if 
$\cmod A > m_0$ and 
$$\cmod B \ge (1-\delta) \cmod A,
$$ then $B$
contains the ``middle'' annulus of modulus $(1-\epsilon)\cmod A$.
\end{lemma}
\begin{proof} 
We first prove an analogous statement with rectangles in place of annuli.
Suppose $R = [0,m] \times [0,1]$ is a rectangle of modulus $m \ge 4/\epsilon$,  
and $S = (ABCD)$ is a conformal sub-rectangle, with $(AB) \subset [0,m] \times \{1\}$ and $(CD) \subset [0,m] \times \{0\}$.
We will show that if $S$ does not contain the middle sub-rectangle of modulus $(1-\epsilon)m$, then $\cmod S \le (1-\epsilon/4)m$. 

By symmetry, we may assume that $S$ is missing a curve joining $z_1 = i y_0$ and $z_2 = (\epsilon/2) m + iy_1$. Note that $m = \lambda_{\Gamma_\leftrightarrow(R)}$ is the extremal length of the horizontal curve family. Giving an upper bound on the extremal length of $\Gamma_{\leftrightarrow}(S)$ is equivalent to finding a lower bound on the extremal length of the vertical curve family $\Gamma_{\updownarrow}(S)$. For this purpose, consider the metric 
\begin{equation}
\label{eq:special-metric}
\rho =
\left\{
 \begin{array}{ll}
 \chi_S,  \qquad &  \re z \ge (\epsilon/4)m, \\
 0,  \qquad &  \re z < (\epsilon/4)m.
  \end{array}  \right.
\end{equation}
Observe the $\rho$-length of any curve in $\Gamma_\updownarrow(S)$ is at least 1, yet $\Area(\rho) \le (1 - \epsilon/4)m$. Therefore,
$
\lambda_{\Gamma_\updownarrow(S)} > \frac{\lambda_{\Gamma_\updownarrow(R)}}{1 - \epsilon/4}
$ as desired.  

We can deduce the original statement with annuli from the special case when $(AB) = (CD) + i$
by representing the pair $B \subset A$ as $A = R/\{z \sim z+i\}$ and $B = S/\{z \sim z+i\}$. Indeed, $\cmod A = m$ while $\cmod B \ge \cmod S$ can only increase 
since
a path in $\Gamma_\circlearrowleft(B)$ contains a path in $\Gamma_\updownarrow(S)$.
\end{proof}
Essentially the same argument shows that:
\begin{lemma}
\label{standard-modulus-estimates2}
Suppose $A = A_{r,R}$ has modulus $\cmod A > m_0$ and $B_1, B_2, B_3 \subset A$ are three essential disjoint annuli, with $B_2$ sandwiched between $B_1$ and $B_3$. For any $\epsilon > 0$, there exists $\delta > 0$ and $m_0 > 0$ such that if $\cmod A > m_0$ and 
 $$
 \cmod B_2 \ge (1/2  - \delta)\cmod A, \qquad \cmod B_i \ge (1/4  - \delta)\cmod A, \quad i=1,3,
 $$ then $B_2$ is  contained within the
 ``middle'' annulus of modulus $(1/2 + \epsilon)\cmod A$. 
\end{lemma}
We leave the details to the reader.

%% file: lower-bounds.tex
\section{Lower bounds for the Weil-Petersson metric}
\label{chap:lower-bounds}
In this section, we explain how one can obtain lower bounds for the Weil-Petersson metric using
 the multipliers of repelling periodic orbits on the unit circle.
  We first consider the Fuchsian case and then handle the Blaschke case by approximation. Somewhat frustratingly, the approximation argument comes with a price: in the Blaschke case, to give a lower bound for the Weil-Petersson metric, we must insist that {\em the quotient torus of the repelling periodic orbit changes at a definite rate in the Teichm\"uller metric.} 
It is precisely this ``minor'' detail which prevents us from showing that the completion of the Weil-Petersson metric on $\mathcal B_2$ attaches precisely the points $e(p/q) \in S^1$
and forces us to restrict our attention to small horoballs.
The difficulty is caused by the error term in Lemma \ref{qd-almostinvariant2}. For details, see the proof of Theorem \ref{lower-bounds} below.

For instance, it is well-known that in Teichm\"uller space, the Weil-Petersson length of a curve $X: [0,1] \to \mathcal T_{g}$ with 
$L_{X(0)}(\gamma) = L_1$
and  $L_{X(1)}(\gamma) = L_2 > L_1$ is bounded below by a definite constant $C(g, L_1, L_2)$. As hinted above, we are unable to prove the analogous statement for the Weil-Petersson metric on
$\mathcal B_2$ where we replace the ``length of a hyperbolic geodesic'' by ``the logarithm of the multiplier of a periodic orbit.''
We note that in order to resolve Conjecture \ref{main-conjecture} from the introduction using the method described here, one would need to show:

\begin{conjecture}
For any Blaschke product $f \in \mathcal B_2$, there exists a repelling periodic orbit $f^{\circ q}(\xi)=\xi$ with $(f^{\circ q})'(\xi) < M_2$ and $\mu \in M( \mathbb{D})^f$ of norm 1 for which
$|\dot L_{0,t}(\xi) / L(\xi)| \asymp 1$,  
where we perturb $f = f_0$  asymmetrically with  $f_{0,t} =  w_{t\mu} \circ f \circ (w_{t\mu})^{-1}$.
In terms of symmetric deformations $f_{t,t} =  w^{t\mu} \circ f \circ (w^{t\mu})^{-1}$, it suffices to check that either 
$|\dot L_{t,t}(\xi) / L(\xi)| \asymp 1$ or $|\dot L_{it,it}(\xi) / L(\xi)| \asymp 1$.
\end{conjecture}

\subsection{Lower bounds in Teichm\"uller space}

Consider a linear map $f(z) = \lambda z$ with $\lambda > 1$. 
Given a Beltrami coefficient $\mu \in M(\mathbb{H})^{f}$ supported on the upper half-plane, form the maps $f_t = w_{t \mu} \circ f_0 \circ (w_{t \mu})^{-1}$. Since we use the asymmetric deformations $w_{t \mu}$, 
 the multipliers $\lambda_t =  f_t'(w_{t\mu}(0))$ are not necessarily real. We view $v =   (d/dt)|_{t=0} \, w_{t \mu}$ as a holomorphic vector field on the lower half-plane.

Let $\pi: \mathbb{C} \to \mathbb{C}/(\cdot \, \lambda)$ be the quotient map. The Beltrami coefficient $\mu$ descends to the quotient
torus, which we also denote $\mu$ when there is no risk of confusion.
Our goal is to give a lower bound for $|v'''/\rho^2|$ in terms of   $\|\mu\|_{T(\mathcal T_1)} = |\dot L_0 / (2L_0)|$ where $L_t = \log \lambda_t$ and
 $\dot L_t = (d/dt)|_{t=0} \log \lambda_t$. 
Suppose first that  $\mu$ is a {\em radial} Beltrami coefficient of the form
\begin{equation}
\label{eq:rbc}
\mu(z) =k(\theta) \cdot \frac{z}{\overline{z}} \cdot \frac{d\overline{z}}{dz}.
\end{equation}

\begin{lemma}
For the radial Beltrami coefficient $\mu$ given by  (\ref{eq:rbc}),
\begin{equation}
v(z) \, = \, \frac{d}{dt}\biggl |_{t=0} w_{t\mu}(z) \, = \, - \frac{1}{2\pi} \cdot z \log z \cdot \int_0^\pi k(\theta) d\theta, \qquad z \in \overline{\mathbb H},
\end{equation}
and therefore,
\begin{equation}
v'''(z) = \frac{1}{2\pi} \cdot \frac{1}{z^2} \cdot \int_0^\pi k(\theta) d\theta, \qquad z \in \overline{\mathbb H}.
\end{equation}
\end{lemma}
\begin{proof}
We compute:
\begin{eqnarray*}
v(z) & = &
\frac{1}{2\pi} \int_{\mathbb{H}} \frac{z(z-1)}{\zeta(\zeta-1)(\zeta-z)} \cdot k(\theta) \cdot (\zeta/\overline{\zeta}) |d\zeta|^2, \\
& = &
 \frac{z}{2\pi} \int_0^\pi k(\theta) \int_0^\infty \frac{(z-1)e^{i\theta}}{(re^{i\theta}-1)(re^{i\theta}-z)} drd\theta, \\
& = & 
\frac{z}{2\pi} \int_0^\pi k(\theta) \int_0^\infty  \biggl( \frac{1}{r-e^{-i\theta}} - \frac{1}{r-ze^{-i\theta}} \biggr) drd\theta, \\
& = & 
\frac{z}{2\pi} \int_0^\pi k(\theta) \cdot (-\log z) d\theta.
\end{eqnarray*}
(Since we are working in $\mathbb{C} \setminus (-\infty,0]$, the branch of the logarithm is well-defined.)
\end{proof}
In view of Lemma \ref{attr-torus-multiplier}, this shows
\begin{equation}
\label{eq:pointwise-lower-bound}
\biggl | \frac{v'''(z)}{\rho_{\Hbar}(z)^2} \biggr|  \asymp  \biggl | \frac{(d/dt)|_{t=0} \log \lambda_t}{\log \lambda_0} \biggr | \asymp \|\mu\|_T, \qquad
\frac{5\pi}{4} < \arg z < \frac{7\pi}{4},
\end{equation}
for radial $\mu$.
For an arbitrary Beltrami coefficient $\mu \in M(\mathbb{H})^f$, the {\em pointwise} lower bound  (\ref{eq:pointwise-lower-bound}) need 
not hold in general.  However, we can deduce an averaged version of (\ref{eq:pointwise-lower-bound}) from the radial case, which suffices for our purposes.
Indeed, by replacing $\mu(z)$ with $\mu(rz)$ and averaging over $r \in (r_1, r_2)$, $r_2/r_1 = \lambda_0$ yields
\begin{equation}
\fint_{r_1}^{r_2} \biggl | \frac{v'''(re^{i\theta})}{\rho_{\Hbar}(re^{i\theta})^2} \biggr| \cdot \frac{dr}{r} \gtrsim 
  \biggl | \frac{(d/dt)|_{t=0} \log \lambda_t}{\log \lambda_0} \biggr | \asymp  \| \mu\|_T, \qquad \frac{5\pi}{4} < \theta < \frac{7\pi}{4}.
\end{equation}
Integrating over $\theta$ and applying the Cauchy-Schwarz inequality, we obtain:
\begin{lemma}
\label{lower-teich-lemma}
Suppose $\mu \in M(\mathbb{H})$ is invariant under $z \to \lambda_0 z$ and $v =   (d/dt)|_{t=0} \, w_{t \mu}$ as above.
For an ``annular rectangle''  $\mathscr R =  {S_{\theta_1,\theta_2} \cap F_{r_1,r_2}}$, 
$$S_{\theta_1,\theta_2} = \{z : \arg z \in (\theta_1, \theta_2)\} \quad \text{and} \quad F_{r_1,r_2} = \{z : r_1  < |z| < r_2 \},$$
with $(\theta_1,\theta_2) \subseteq (5\pi/4,7\pi/4)$ and $r_2/r_1 = \lambda_0$, we have
\begin{equation}
\label{eq:lower-teich}
\fint_{\mathscr R} \biggl | \frac{v'''(z)}{\rho_{\Hbar}(z)^2} \biggr|^2 \cdot \rho_{\Hbar}^2 |dz|^2 \gtrsim   \biggl | \frac{(d/dt)|_{t=0} \log \lambda_t}{\log \lambda_0} \biggr |^2 \asymp  \|\mu\|^2_T.
\end{equation}
\end{lemma}

We can use Lemma \ref{lower-teich-lemma} to study the Weil-Petersson metric on Teichm\"uller space.
Suppose $X \in \mathcal T_{g}$ is a Riemann surface and $\gamma \subset X$ is a simple geodesic whose length is bounded above and below, e.g.~$L_1 < L_X(\gamma) < L_2$. Let $p: \mathbb{H} \to X = \mathbb{H}/\Gamma$ be the universal covering map chosen so that the imaginary axis covers $\gamma$.
By the collar lemma (e.g.~see \cite[Theorem 3.8.3]{Hub}), there exists an annular rectangle $\mathscr R$ with $(r_1,r_2) = (1, e^{L_X(\gamma)})$ and $(\theta_1,\theta_2) =  (-\pi/2 - \epsilon_{L_2}, -\pi/2 + \epsilon_{L_2})$ which has definite hyperbolic area, and for which $(p \circ \overline{z}) |_\mathscr R$ is injective. It follows that for a Beltrami coefficient $\mu \in M(\mathbb{H})^\Gamma$, we have
 $\|p_*\mu\|_{\WP(\mathcal T_g)} \gtrsim \|\pi_*\mu\|_{T(\mathcal T_1)}$.

For applications to dynamical systems, it is easier to work with round balls instead of annular rectangles. An averaging argument similar to the one above
shows:

\begin{lemma}
\label{lower-teich-lemma2}
Suppose the multiplier $\lambda_0 = f'(0) < M_2$ is bounded from above. Given $0 < R < 1$, one can find a ball $\mathscr B = \{ w  : d_{\overline{\mathbb{H}}}(-iy_0, w) < R \}$,
$1 \le y_0 \le \lambda_0$, for which
$$
\fint_{\mathscr B} \biggl | \frac{v'''(z)}{\rho_{\Hbar}(z)^2} \biggr|^2 \cdot |dz|^2 \gtrsim   \biggl | \frac{(d/dt)|_{t=0} \log \lambda_t}{\log \lambda_0} \biggr |^2.
$$\end{lemma}

\subsection{Lower bounds in complex dynamics}
For a Blaschke product $f \in \mathcal B_2$ and $\mu \in M(\mathbb{D})^f$,  we consider the quadratic differential $v''' = v'''_{\mu^+}$ and the two-parameter family 
$
f_{s,t} := w_{\mu_{s,t}} \circ f \circ (w_{\mu_{s,t}})^{-1}
$
where $\mu_{s,t} := s\mu + (t\mu)^+$ and $\nu^+ := \overline{(1/\overline z)^* \nu}$.

\begin{theorem}[Blowing up]
\label{lower-bounds}
Suppose $f(z) \in \mathcal B_2$ is Blaschke product and $f^{\circ q}(\xi) = \xi$ is a repelling periodic point on the unit circle with $(f^{\circ q})'(\xi) < M_2$.
If $\mu(z) \in M( \mathbb{D})^f$
satisfies $\|\mu\|_\infty \le 1$ and $|\dot L_{0,t}(\xi) / L(\xi)| \asymp 1$, then  
 there exist a ball %of a definite hyperbolic radius
\begin{equation}
\label{eq:small-ball-estimate1}
\mathscr B = B\Bigl (\xi \cdot (1-c_1 \cdot \delta_c), c_2 \cdot \delta_c \Bigr) \quad \text{for which} \quad \fint_{\mathscr B} \biggl |\frac{v'''(z)}{\rho(z)^2} \biggr |^2 \cdot |dz|^2 \asymp  1.
 \end{equation}
\end{theorem}

\begin{proof}
By Lemma \ref{lower-teich-lemma2}, we can find a small ball $\mathscr B_0$ of  definite hyperbolic size
 near $\xi$ for which
\begin{equation}
\label{eq:small-ball-estimate}
\fint_{\mathscr B_0} \biggl |\frac{v'''(z)}{\rho(z)^2} \biggr |^2 \cdot |dz|^2 \asymp  |\dot L_{0,t}(\xi) / L(\xi)|^2.
\end{equation}
Using the forward iteration of $f$ (and Koebe's distortion theorem), we can blow up this ball so that its Euclidean size is comparable to $\delta_c$.
Note that 
due to the error term in Lemma \ref{qd-almostinvariant2}, in order for the estimate (\ref{eq:small-ball-estimate}) to remain
meaningful, we must insist that $|\dot L_{0,t}(\xi) / L(\xi)|$ is bounded from below.
\end{proof}

\begin{theorem}[Blowing down]
\label{0t-bounds}
In the setting of Theorem \ref{lower-bounds},
if the multiplier is bounded from both below and above, $M_1 < (f^{\circ q})'(\xi) < M_2$, then
\begin{equation}
 \label{eq:lb-bounds}
 \limsup_{r\to 1^-} \frac{1}{2\pi} \int_{|z|=r} \biggl | \frac{v'''(z)}{\rho(z)^2} \biggr |^2 d\theta \asymp 1.
 \end{equation}  
\end{theorem}
\begin{proof}[Sketch of proof]
In view of  Lemma \ref{qd-almostinvariant2}, the estimate (\ref{eq:small-ball-estimate1}) holds for the inverse images of $\mathscr B$. 
Since the multiplier is bounded from below,  the constants $c_1$ and $c_2$  in Theorem \ref{lower-bounds}  can be chosen small enough so that the repeated inverse images of $\mathscr B$ are disjoint from $\mathscr B$ (and thus from each other). 
\begin{comment}
Indeed, by Lemma \ref{weak-linearity} below, if we choose $\mathscr B \subset B \bigl (\xi, \delta_c/(KM_2) \bigr )$, then 
we can guarantee that $\mathscr B$ is disjoint
from $f^{\circ k}(\mathscr B)$ for $k=1,2,\dots, q-1$. However, if we make the hyperbolic radius of $\mathscr B$ small (depending on $M_1$), we can
guarantee that in $q$ steps, $\max |f^{\circ q}(\mathscr B)| < \min_{z \in \mathscr B} |z|$, therefore, by the Schwarz lemma, $\mathscr B$ is disjoint from all its forward iterates.
\end{comment}
By Lemmas \ref{jensen} and \ref{laminated-area-in-linearity-zone}, the laminated area $\mathcal A(\hat{\mathscr B})$ is bounded from below, which proves (\ref{eq:lb-bounds}).
 \end{proof}

In Section \ref{chap:coarse-geometry}, we will use the ``blowing up'' and ``blowing down'' techniques to give lower bounds for the Weil-Petersson metric 
when the multiplier of the repelling periodic orbit is small.

\begin{remark}
To give lower bounds for the Weil-Petersson metric, we used the  gradient of the multiplier of a periodic orbit in the $\mu$ direction. In view of the the identities
\begin{eqnarray*}
(d/dt)|_{t=0} \, \log (f_{t,t}^{\circ q})'(\xi_{t,t})& = & 2 \, \re \, (d/dt)|_{t=0} \, \log (f_{0,t}^{\circ q})'(\xi_{0,t}), \\
(d/dt)|_{t=0} \, \log (f_{it,it}^{\circ q})'(\xi_{it,it})& = & 2 \, \im \, (d/dt)|_{t=0} \, \log (f_{0,t}^{\circ q})'(\xi_{0,t}),
\end{eqnarray*}
we can also use the gradient of the multiplier in the Blaschke slice, i.e.~in the $\mu + \mu^+$ or $i \mu + (i\mu)^+$ directions.
\end{remark}

%% file: small-horoball-argument.tex
\section{Incompleteness: General Case}
\label{chap:coarse-geometry}

In this section, we prove Theorem \ref{small-horoball-argument} which says that the Weil-Petersson metric is comparable to the model metric $\rho_{1/4}$
in the small horoballs. Note that outside the small horoballs, the upper bound is automatic: 
see the corollary to Theorem \ref{thm-smirnov} or use
part $(a)$ of Theorem \ref{qbounds}.

  Unraveling definitions, we need to show that if
$f_a \in \mathcal B_2$, $a \in \mathcal H_{p/q}(\eta)$, $\eta < C_{\sma}$ and 
$\mu = \mu_\lambda =  \varphi_a^*(\lambda \cdot z/\overline{z} \cdot d\overline{z}/dz) \in M(\mathbb{D})^{f_a}$ is an optimal Beltrami coefficient with $|\lambda|=1$,
then $\|\mu \cdot \chi_{\mathcal G}\|^2_{\WP} \asymp \eta^{1/2}$.

For $a \in \mathcal B_{p/q}(C_{\sma})$, the flowers are still well-separated; 
however, we no longer have uniform control on the quasi-geodesic property. Indeed, when $a^q \in \mathbb{C} \setminus
 [0,\infty)$, multiplication by $a^q$ traces out a logarithmic spiral $\{a^{qt}, t>0\}$, and if we take $a^q \to 1$ along a horocycle, this logarithmic 
 spiral tends to $\overline{\mathbb{D}}$ in the Hausdorff topology.  Nevertheless, we can still show that $\limsup_{r \to 1^-} |\mathcal G(f_a) \cap S_r|$ is small. 
 The following lemma is the key to both the upper and lower bounds for $\|\mu \cdot \chi_{\mathcal G}\|^2_{\WP}$\,:
\begin{lemma}
\label{strong-linearization-at-rfp}
Suppose that $\langle \xi_1, \xi_2, \dots, \xi_q \rangle$ is a repelling periodic orbit of a Blaschke product $f \in \mathcal B_2$ whose multiplier is 
$m < M_{\sma} := 1 + \frac{1}{16}$. 
There exists a constant $K > 0$ sufficiently large such that the branch of $(f^{\circ q})^{-1}$ which takes $\xi_i$ to itself, maps
$B  (\xi_i, R  )$  strictly inside of itself, where $R := \frac{\delta_c}{K \sqrt{m-1}}$.
\end{lemma}
% In particular,
\begin{corollary}
For each $i = 1, 2, \dots, q$, the formula
\begin{equation}
\varphi_{\xi_i}(z) := \lim_{n \to \infty} m^{n} \Bigr ( (f^{\circ nq})^{-1}(z) - \xi_i \Bigr ) 
\end{equation}
defines a univalent holomorphic function on $B(\xi_i, R)$ satisfying
$$\varphi_{\xi_i} \circ (f^{\circ q})^{-1} = m^{-1} \cdot \varphi_{\xi_i}, \quad \varphi_{\xi_i}(\xi_i) = 0, \quad (\varphi_{\xi_i})'(\xi_i) = 1.$$
\end{corollary}

By Koebe's distortion theorem, Lemma \ref{strong-linearization-at-rfp} implies that  the dynamics of $f^{\circ q}$ is nearly linear in the balls $B(\xi_i, R)$, i.e.~if $z, f^{\circ q}(z), f^{\circ 2q}(z), \dots, f^{\circ nq}(z) \in B(\xi_i, t \cdot R)$ with $t \le 1/2,$ then:
\begin{equation}
\biggl | \frac{|(f^{\circ nq})'(z)|}{m^n} - 1 \biggr | \lesssim t
\qquad\text{and}\qquad
\bigl |\arg(f^{\circ nq}(z) - \xi_i) - \arg(z - \xi_i) \bigr | \lesssim t.
\end{equation}

\begin{remark}
Note that Lemma \ref{strong-linearization-at-rfp} is only significant for repelling periodic orbits with small multipliers.
For $m > M_{\sma}$, one can apply Koebe's distortion theorem to the inverse branch $(f^{\circ q})^{-1}$ on $B(\xi_i, \delta_c)$ to
see that
$(f^{\circ q})^{-1}$ maps the ball $B(\xi_i, \delta_c/K)$ inside of itself.
\end{remark}
Combining Lemma \ref{strong-linearization-at-rfp}  with part (ii) of Lemma \ref{small-horoballs} gives: 
\begin{theorem}[Flower bounds]
\label{flower-bounds}
There exists a constant $\pi/2 < \theta_1 < \pi$ such that for any $f_a \in \mathcal B_2$ with $a \in \mathcal B_{p/q}(C_{\sma})$,
\begin{equation}
\label{eq:flower-bounds}
\mathcal F \ \subset \ \bigcup_{i=1}^q S \bigl (\xi_i, \theta_1, R\bigr ) \, \cup \, B \bigl (0, 1 - 0.5 \cdot R \bigr ) \ =: \ \bigcup S_i \, \cup \, B.
\end{equation}
\end{theorem}
(The notation $S(\zeta, \theta, R) := \bigl \{z : \arg(z/\zeta - 1) \in (\pi - \frac{\theta}{2}, \pi + \frac{\theta}{2}) \bigr \} \cap B(\zeta, R)$ denotes the central sector at $\zeta \in S^1$ of opening $\theta$.)
\begin{remark}
We do not need to know {\em any} information about the behavior of the flower within the ball $B(0, 1-0.5 \cdot R)$.
\end{remark}
With  the help of Theorem \ref{flower-bounds}, we extend the flower separation and structure lemmas to the wider class of parameters. Since the statements
are interrelated, we state them as a single theorem:
\begin{theorem} 
\label{petal-separation-2}
For $a \in \mathcal H_{p/q}(\eta)$ with $\eta < C_{\sma}$,
 \begin{enumerate} 
\item[$(a)$] The hyperbolic distance $d_{\mathbb{D}}(\mathcal F, c) \ge \frac{1}{2} \log(1/\eta) - O(1)$.
\item[$(b)$] The hyperbolic distance $d_{\mathbb{D}}(\mathcal F, {\mathcal F}_*) \ge \log(1/\eta) - O(1)$.
\item[$(c)$] The hyperbolic distance between any two pre-flowers exceeds $\log \eta - O(1)$.
 \item[$(a')$] The critically-centered flower $\tilde{\mathcal F} \subset B(-\hat{c}, \const \cdot\, \eta^{1/2})$.
\item[$(b')$] The immediate pre-flower $\mathcal F_* \subset B(\hat{c}, \const\cdot\, \delta_c \cdot \eta^{1/2})$.
\end{enumerate}
\end{theorem}
 Using Theorems \ref{flower-bounds} and \ref{petal-separation-2}, it is easy to deduce Theorem \ref{small-horoball-argument}. We give the details in Section \ref{sec:conclusion-sha}.

\subsection{Linearization at repelling periodic orbits}

To show Lemma \ref{strong-linearization-at-rfp},
we recall a formula for the derivative of a Blaschke product on the unit circle:
\begin{lemma} [Equation (3.1) of \cite{McM-cyc}]
\label{derivative-blaschke}
Given a Blaschke product $f_{\mathbf{a}} \in \mathcal B_d$, for $\zeta \in S^1$, we have
\begin{equation}
|f_{\mathbf{a}}'(\zeta)| = 1 + \sum_{i=1}^{d-1} \frac{1-|a_i|^2}{|\zeta+a_i|^2}.
\end{equation}
\end{lemma}
In particular, the absolute value of the derivative of a Blaschke product is always greater than 1 on the unit circle.
Specifying Lemma \ref{derivative-blaschke}
to degree 2 and rearranging, we obtain:
\begin{lemma}
\label{critical-distance-lemma-x}
Suppose $f \in \mathcal B_2$ is degree 2 Blaschke product and $\zeta \in S^1$. Then,
\begin{equation}
\label{eq:critical-distance}
|\zeta + a | = \sqrt{\frac{1-|a|^2}{|f'(\zeta)|-1}}.
\end{equation}
\end{lemma}
Lemma \ref{critical-distance-lemma-x} says that if $|f'(\zeta)|$ is close to 1, then $\zeta$ is far away from the point $-a$.
For example, the condition $|f'(\zeta)| < 1 + \frac{1}{16}$ guarantees that $|\zeta+a| \ge 4 \, \delta_c$ and $|\zeta - \hat{c}| \ge 3 \, \delta_c$.
(Since  the critical point $c$ is the hyperbolic midpoint of $[0, -a]$, we have $(\frac{1}{2} + \frac{1}{\sqrt{2}}) \delta_c \ge \sqrt{1-|a|^2} \ge \delta_c$.)

\begin{lemma}
\label{dbl-derivatives}
There exists a constant $K > 0$ such that for 
any degree 2 Blaschke product $f_a \in \mathcal B_2$ and $\zeta \in S^1 \setminus B(\hat{c}, 3\delta_c)$,
\begin{equation}
\label{eq:injectivity-ge}
|f'(z) - f'(\zeta)| \ \le \ \frac{|f'(\zeta)| - 1}{2}, \qquad z \in B \biggl (\zeta,\, \frac{|\zeta+a|}{K} \biggr ).
\end{equation}
In particular, $f$  is injective on $B \bigl (\zeta, \frac{|\zeta+a|}{K} \bigr )$ with $f \bigl (B (\zeta, \frac{|\zeta+a|}{K} ) \bigr) \supset 
B \bigl (f(\zeta), \frac{|\zeta+a|}{K} \bigr)$, and the branch of $f^{-1}$ defined on $B \bigl (f(\zeta), \frac{|\zeta+a|}{K} \bigr)$ which takes $f(\zeta) \to \zeta$ is a contraction.
\end{lemma}

\begin{proof}
Differentiating twice gives $f''(z) = \frac{2(1-|a|^2)}{(1+\overline{a}z)^3}$
which implies that
\begin{equation}
|f''(z)| \asymp \frac{1-|a|^2}{|z+a|^3}, \qquad z \in \mathbb{D} \setminus B(\hat{c}, 2\delta_c).
\end{equation}
In view of (\ref {eq:critical-distance}), this gives
$$
|f''(z)| \le \frac{C}{|\zeta + a|} \cdot \bigl (|f'(\zeta)| - 1 \bigr), \qquad z \in B \biggl (\zeta,\, \frac{|\zeta +a|}{3} \biggr ).
$$
Therefore, (\ref{eq:injectivity-ge}) holds with $K = \min(1/3, 1/(2C))$.
\end{proof}

\begin{proof}[Proof of Lemma \ref{strong-linearization-at-rfp}] Let $m_i = |f'(\xi_i)|$ so that $m = |(f^{\circ q})'(\xi_1)| = m_1m_2 \cdots m_q$.
Since each $1 \le m_i \le m \le 1 + \frac{1}{16}$, by Lemma \ref{dbl-derivatives}, $f^{-1}$ is a contraction on each ball $B(\xi_i, R)$.
Therefore, the composition $(f^{\circ q})^{-1}$ is a contraction as well.
\end{proof}

\subsection{Separation and structure revisited}

The following lemma provides a convenient way for estimating hyperbolic distances between points in the unit disk:%, which we call the {\em principle of the highest point}\,:
\begin{lemma}
\label{mob-move}
Suppose $z_1,z_2 \in \mathbb{D}$ and $z_0$ is the point on the hyperbolic geodesic $[z_1,z_2]$ closest to origin. 
If $z_0$ does not coincide with either endpoint, then
\begin{equation}
\label{eq:c-dist-estimate}
d_\mathbb{D}(z_1,z_2) = d_\mathbb{D}(|z_1|, |z_0|) + d_\mathbb{D}(|z_0|,|z_2|) + O(1).
\end{equation}
\end{lemma}

\begin{corollary}
\label{mob-move2}
If $\zeta_1, \zeta_2 \in S^1$ and the balls $B(\zeta_1, 2r_1)$ are $B(\zeta_2, 2r_2)$ are disjoint, then
\begin{equation}
d_\mathbb{D} \Bigl( B(\zeta_1, r_1), \ B(\zeta_2, r_2) \Bigr ) = d_\mathbb{D} \Bigl ( \zeta_1(1-r_1), \ \zeta_2(1-r_2) \Bigr) + O(1).
\end{equation}
\end{corollary}
To deduce the corollary from lemma, it suffices to observe that for two points
 $z_1 \in \partial B(\zeta_1, r_1) \cap \mathbb{D}$ and $z_2 \in \partial B(\zeta_2, r_2) \cap \mathbb{D}$, the highest point $z_0$ on $[z_1,z_2]$ satisfies $1 - |z_0| \asymp |\zeta_1-\zeta_2|$.

Recall that the condition $|f'(\zeta)| < 1 + \frac{1}{16}$ guarantees that $|\zeta - \hat{c}| \ge 3 \, \delta_c$.
Applying the corollary with $B(\hat{c}, 2 \cdot \delta_c)$ and  $B \Bigl (\zeta, 2 \cdot \frac{|\zeta+a|}{K} \Bigr )$ and $K \ge 2$ shows:
\begin{lemma}
\label{critical-distance-lemma}
Suppose that $f_a \in \mathcal B_2$ is a degree 2 Blaschke product and $\zeta \in S^1$ is such that $|f'(\zeta)|<1+\frac{1}{16}$. For $K \ge 2$, we have
\begin{equation*}
d_{\mathbb{D}} \Bigl ( c , \, B(\zeta,  R_\zeta) \Bigr ) = \frac{1}{2} \log \frac{1}{|f'(\zeta)|-1} + \log K + O(1), \quad R_\zeta := \frac{|\zeta+a|}{K}.
\end{equation*}
\end{lemma}

We now deduce Theorem \ref{petal-separation-2} from Theorem \ref{flower-bounds}:

\begin{proof}[Proof of Theorem \ref{petal-separation-2}]
Recall from 
Theorem \ref{small-cycles} that $\eta \asymp (m-1)$.
By Theorem \ref{flower-bounds}, the flower $\mathcal F$ is contained in $\mathbb{D} \setminus B(\hat{c}, R/2)$.
This implies that 
$$
d_{\mathbb{D}}(\mathcal F, c) \ge d_{\mathbb{D}} \bigl (\hat{c} (1 - R/2), \, c \bigr ) = \log \Bigl (\frac{1}{K \sqrt{m-1}} \Bigr ) + O(1) =  \log \frac{1}{\eta^{1/2}} + O(1) 
$$
 and
$\tilde{\mathcal F} \subset m_{c \to 0} \bigl (\mathbb{D} \setminus B(\hat{c}, R/2) \bigr ) \subset B(-\hat c, \, \const \cdot\, \eta^{1/2}).$
This proves $(a)$ and $(a')$.

Applying Koebe's distortion theorem to the appropriate branch of $\tilde f^{-1}$ on $B(-\hat{c}, 1)$, we see that $\tilde {\mathcal F}_*$ is a nearly-affine copy of $\tilde{\mathcal F}$.
Furthermore, since $c$ is the midpoint of the hyperbolic geodesic $[0,-a]$, 
we must have $\tilde {\mathcal F}_* \approx -  \tilde{\mathcal F}$. Therefore,
$$
d_\mathbb{D}(\mathcal F, \mathcal F_*) =d_\mathbb{D}(\tilde{\mathcal F}, \tilde{\mathcal F}_*) \ge \log \frac{1}{\eta} + O(1)
$$
and $\mathcal F_* \subset m_{0 \to c} \bigl (B(\hat c, \, \const \cdot\, \eta^{1/2}) \bigr ) \subset B(\hat{c}, \const\cdot\, \delta_c \cdot \eta^{1/2})$.
This proves $(b)$ and  $(b')$.

Finally, $(c)$ follows from the Schwarz lemma and the trick used in the proof of part $(b)$ of Theorem \ref{qg-ps}. 
\end{proof}

\subsection{Proof of the main theorem}
\label{sec:conclusion-sha}
We are now ready to show that 
$$\|\mu \cdot \chi_{\mathcal G}\|^2_{\WP} \lesssim \eta^{1/2} \qquad \text{for } a \in \mathcal H_{p/q}(\eta) \text{ with }\eta < C_{\sma}.
$$
We first prove the upper bound. 
Reflecting (\ref{eq:flower-bounds}) about the critical point, we see that the immediate pre-flower $\mathcal F_*$ is contained in the union of the reflections $\bigcup S_i^* \, \cup \, B^*$.
We claim that
\begin{equation}
\label{eq:lam-conclusion}
 \int_{{\mathcal F}_*} \frac{|dz|^2}{1-|z|} \lesssim \delta_c \sqrt{m-1}.
\end{equation}
Assuming the claim, Lemmas \ref{jensen} and \ref{laminated-area-in-linearity-zone} tell us that the laminated area
 $$\mathcal A({\mathcal G}(f_a)) \,  \lesssim \, \frac{\delta_c \sqrt{m-1}}{h(f_a,m)} \, \asymp \, \sqrt{m-1} \, \asymp\, \eta^{1/2},$$
  which by Theorem \ref{wpbounds} implies $\|\mu \cdot \chi_{\mathcal G}\|^2_{\WP} \lesssim \eta^{1/2}$ as desired.
To prove (\ref{eq:lam-conclusion}), we need to carefully reflect the flower about the critical point. 

The reflection $B^*$ of the ball $B (0, 1 - 0.5 \cdot R )$ is contained in a horoball of diameter $\asymp \delta_c \sqrt{m-1}$, resting on
$\hat{c}$. Therefore,
$ \int_{B^*} \frac{|dz|^2}{1-|z|} \lesssim \delta_c \sqrt{m-1}$. 
Similar reasoning shows that the reflection $S_i^*$ of $S_i$ is contained in a sector $S(\xi_i^*, \theta_2, R_i^*)$, with $\theta_1 < \theta_2 < \pi$ and 
\begin{equation}
R_i^* \asymp \delta_c \cdot \sqrt{m_i-1} \cdot \frac{\sqrt{m_i - 1}}{\sqrt{m-1}} = \delta_c \cdot \frac{m_i - 1}{\sqrt{m-1}}.
\end{equation}
The total contribution of these sectors to the integral (\ref{eq:lam-conclusion})  is roughly
\begin{equation}
\label{eq:sector-sum}
\int_{\bigcup S_i^*} \frac{|dz|^2}{1-|z|} \asymp \delta_c \sum\frac{m_i - 1}{\sqrt{m-1}} \asymp \delta_c\sqrt{m-1}.
\end{equation}
 This proves the upper bound. 
 
 For the lower bound, observe that by the blowing up technique of Theorem  \ref{lower-bounds} together with  Theorem \ref{small-cycles}, there exist balls
\begin{equation}
\mathscr B_i = B\biggl (\xi_i \cdot \Bigl (1-c_1 \cdot \frac{\delta_c}{\sqrt{m-1}} \Bigr ), \ c_2 \cdot \frac{\delta_c}{\sqrt{m-1}} \biggr)
\end{equation}
lying in the sectors $S_i$ for which
$
 \label{lower-bounds-for-small-multipliers}
 \fint_{\mathscr B_i}  |v'''/\rho^2(z)  |^2 \cdot |dz|^2 \asymp 1.
$
The reflection $\mathscr B_i^*$ of $\mathscr B_i$ is a ball of definite hyperbolic size whose Euclidean center is located roughly at height $$\asymp \delta_c \cdot \sqrt{m_i-1} \cdot \frac{\sqrt{m_i - 1}}{\sqrt{m-1}} \ = \ \delta_c \cdot \frac{m_i - 1}{\sqrt{m-1}}.$$ 

Since the (repeated) pre-images of the $\mathscr B_i^*$ are disjoint, and each repeated pre-image is a near-affine copy of $\mathscr B_i^*$, 
by Lemmas \ref{jensen} and \ref{laminated-area-in-linearity-zone}, 
$$
\mathcal A \biggl ( \bigcup_i \hat{\mathscr B_i^*} \biggr ) \, \asymp \, 
\sum\frac{m_i - 1}{\sqrt{m-1}} \, \asymp \, \sqrt{m-1} \, \asymp \, \eta^{1/2}.$$ Thus, the lower bounds match the upper bounds up to a multiplicative constant.
This concludes the proof of Theorem \ref{small-horoball-argument}.

%% file: limiting-vector-fields.tex
\section{Limiting Vector Fields}

In this section, we  study the convergence of Blaschke products to vector fields. 
%In the next chapter, we will use this to show precise rate of decay of the Weil-Petersson metric as $a \to e(p/q)$ radially in $\mathcal B_2$.
For a Blaschke product $f_{\mathbf{a}}(z) = z \prod_{i=1}^{d-1} \frac{z+a_i}{1+\overline{a_i}z}$, set $z_i := -a_i$.
By a {\em radial degeneration}\,, we mean a sequence of Blaschke products $f_{\mathbf{a}} \in \mathcal B_d$ such that:

\begin{enumerate}
\item The multiplier of the attracting fixed point tends (asymptotically) {\em radially} to $e(p/q)$, i.e.~$\arg(e(p/q)-a) \to \arg(e(p/q))$.

\item Each $z_i$ converges to some point $e(\theta_i) \in S^1$.

\item The limiting ratios of speeds at which the zeros escape are well-defined, i.e.~ 
$$1-|z_i| \sim \rho_i \cdot (1-|a|)$$
with $\rho_i > 0$  and $\sum \rho_i = 1$.
\end{enumerate}
To a radial degeneration, one can associate a natural probability measure $\mu$ on the unit circle which takes the escape rates into account:
$\mu$ gives mass $\rho_i/q$ to $e(\theta_i+j/q)$. Here, we use the convention that if some of the points coincide, we sum the masses.
 \begin{theorem} 
 \label{algebraic-convergence}
One can compute:
\begin{equation}
\label{eq:limvf}
\kappa(z) = \lim_{a \to 1} \frac{ f_{\mathbf{a}}^{\circ q}(z) - z}{1-|a|^q} \to - z \int \frac{\zeta-z}{\zeta+z} d\mu_\zeta.
\end{equation}
Furthermore,
\begin{equation}
\label{eq:limvferror}
f_{\mathbf{a}}^{\circ q}(z)-z-(1-|a|^q) \kappa(z) = O \Bigl ( (1-|a|^q)^2 \Bigr )
\end{equation}
uniformly in the closed unit disk away from $\supp \mu$.
\end{theorem}

\noindent \textbf{Examples:} 
\begin{enumerate}
\item[$(1)$] As $a \to 1$ radially in $\mathcal B_2$, $f_a \to \kappa_1 := z \cdot \frac{z-1}{z+1} \cdot \frac{\partial}{\partial z}$.
\item[$(2)$] As $a \to e(p/q)$ radially in $\mathcal B_2$, $f_a^{\circ q} \to \kappa_{q} := q \cdot ((-1)^{q+1} \cdot z^q)^*\kappa_1$.
\end{enumerate}

 Let $\{g^\eta\}_{0<\eta<1}$ be the semigroup generated by $\kappa$ written in multiplicative notation, i.e.~$g^{\eta_1} \circ g^{\eta_2} = g^{\eta_1 \eta_2}$, normalized so that
 $(g^\eta)'(0) = \eta$.
Using (\ref{eq:limvferror}), we promote the algebraic convergence in (\ref{eq:limvf}) to the dynamical convergence of 
the high iterates of $f_{\mathbf{a}}$ to the flow generated by $\kappa(z)$:
\begin{theorem}
\label{dynamical-convergence}
For $0 < \eta < 1$, if we choose $T_{a, \eta}$ so that
$$
(f_{\mathbf{a}}^{\circ q \, T_{a,\eta}})'(0) \to \eta,
$$
then
$f_{\mathbf{a}}^{\circ q \, T_{a,\eta}} \to g^\eta$ uniformly in the closed unit disk away from $\supp \mu$. 
\end{theorem}
For applications, it is convenient to use the convergence of linearizing coordinates:
 \begin{corollary}
As $a \to e(p/q)$ radially, the linearizing coordinates $\varphi_{\mathbf{a}}: \mathbb{D} \to \mathbb{C}$ converge to the linearizing coordinate $\varphi_\kappa := \lim_{\eta \to 0^+} (1/\eta) \cdot g^\eta(z)$ of the semigroup generated by the limiting vector field $\kappa$.
\end{corollary}
\begin{remark}
More generally, one can consider {\em linear degenerations} where $a \to e(p/q)$ asymptotically along a linear ray, i.e.~with $a \approx e(p/q)(1 - \delta + \delta \cdot Ti)$
and $\delta$ small. In this case,
the limiting vector field takes the more general form
\begin{equation}
\label{eq:limvf2}
\kappa(z) = \lim_{a \to 1} \frac{ f_{\mathbf{a}}^{\circ q}(z) - z}{1-|a|^q} \to - z \int \frac{\zeta-z}{\zeta+z} d\mu_\zeta + Ti \cdot z.
\end{equation}
We call $\mu$ the {\em driving measure} and $T$ the {\em rotational factor}.
\end{remark}

 \begin{figure}[h!]
  \centering
      \includegraphics[width=0.4\textwidth]{images/aVF.png}
      \qquad
      \qquad
            \includegraphics[width=0.4\textwidth]{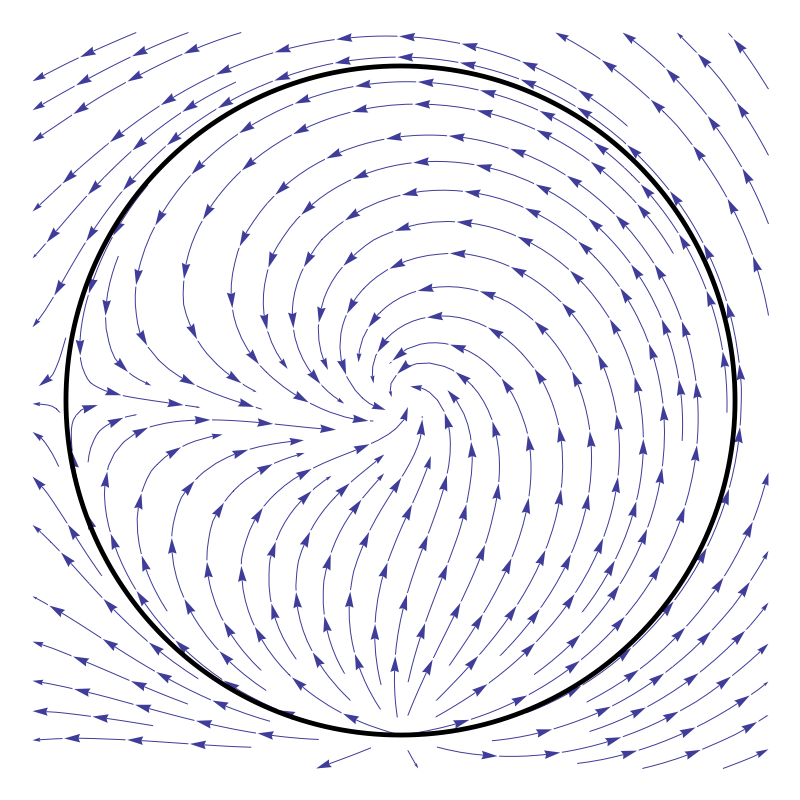}
  \caption{The vector fields $z \cdot \frac{z-1}{z+1} \cdot \frac{\partial}{\partial z}$ and $\Bigl (z \cdot \frac{z-1}{z+1} + iz \Bigr)  \frac{\partial}{\partial z}$.}
   \label{fig:aVF4}
\end{figure}

\subsection{Blaschke vector fields}

Before proving Theorems \ref{algebraic-convergence} and \ref{dynamical-convergence}, let us examine the vector fields that may be obtained by this process. Recall that for a holomorphic vector field $\kappa$, the poles of $\kappa$ are the {\em saddle points}, while the zeros are {\em sources} if $\text{Re } \kappa'(z) > 0$ and {\em sinks} if $\text{Re } \kappa'(z) < 0$
(if $\text{Re } \kappa'(z) = 0$, then $z$ is a {\em center} but in our setting, this possibility does not occur).
 
Observe that for $\zeta \in S^1$, the map $z \to - \frac{\zeta+z}{\zeta-z}$ takes the unit disk onto the left half-plane.
Therefore, as a function of $z$ on the unit circle, $-\int \frac{\zeta+z}{\zeta-z} d\mu_\zeta$ takes purely imaginary values and (its imaginary part) is monotone increasing in $\arg z$ (except at the poles of $\kappa$). It follows that 
$\kappa = - z \int \frac{\zeta+z}{\zeta-z} d\mu_\zeta$ is tangent to the unit circle, has simple poles and in between any two adjacent poles has a unique zero. 
Since $\kappa'(0) = -1$, the point $0$ is a sink.
Conversely, it can be shown that any vector field with the above properties comes from a radial degeneration of some sequence of Blaschke products.
Since we will not need this fact, we omit the proof.

\begin{lemma} Let $M_a(z) = \frac{z+a}{1+\overline{a}z}$. Suppose $a \approx A \in S^1$ with $a = A(1 - \delta + \delta \cdot Ti)$ and $\delta > 0$.
Then,
\begin{equation}
\label{eq:mobius-poisson}
 \frac{M_a(z)/A-1}{1-|a|} = \Bigl ( - \frac{A-z}{A+z} + Ti \Bigr ) + O\Bigl ((1-|a|)^2 \Bigr)
\end{equation}
where the estimate is uniform over $a$ in any non-tangential sector at $A$.
\end{lemma}

\begin{proof}
This is an exercise in differentiation. One simply needs to compute
$$
\frac{\partial}{\partial \delta} \biggl |_{\delta=0} \ \frac{1}{A} \cdot \frac{z + A(1-\delta+\delta \cdot Ti)}{1+(1/A)(1-\delta-\delta \cdot Ti)z} = \frac{z-A}{z+A} + Ti
$$
and use the fact that $1-|a| \approx \delta$.
\end{proof}
We first prove Theorem \ref{algebraic-convergence} in the case when $a \to 1$. For a Blaschke product  $f_{\mathbf{a}}(z) = z \prod_{i=1}^{d-1}  \frac{z+a_i}{1+\overline{a_i}z}$,
let $A_i = \hat{a_i}$, $A = \hat{a}$ and $T = T(f_{\mathbf{a}}) = -i \cdot \frac{A - 1}{1-|a|}$. The idea is to compare $f_{\mathbf{a}}(z)$
to the vector field $\kappa(f_{\mathbf{a}})$ given by (\ref{eq:limvf2}) with driving measure $\mu(f_{\mathbf{a}}) = \sum \frac{1-|a_i|}{1-|a|} \cdot \delta_{-A_i}$ and rotational factor $T(f_{\mathbf{a}})$\,:
 \begin{lemma}
 \label{vfc-lemma} 
The estimate
\begin{equation}
f_{\mathbf{a}}(z)-z-(1-|a|) \kappa(z) = O \Bigl ( (1-|a|)^2 \Bigr )
\end{equation}
holds uniformly for $z$ in the closed unit disk away from $\supp \mu$.
\end{lemma}
\begin{proof}
Using that
$\prod (1+\delta_i) = 1 + \sum \delta_i + O(\max |\delta_i|^2  )$ gives
\begin{eqnarray*}
f_{\mathbf{a}}(z) - z & = &  z \biggl ( \prod \frac{z+a_i}{1+\overline{a_i}z} - \prod A_i \biggr ) + z(A - 1) \\
& \approx & Az \sum \Bigl ( \frac{1}{A_i}  \frac{z+a_i}{1+\overline{a_i}z} - 1 \Bigr ) + z(A - 1).
\end{eqnarray*}
Therefore,
$$
\frac{ f_{\mathbf{a}}(z) - z}{1-|a|} \approx
- Az \sum \rho_i \cdot \Bigl ( \frac{A_i-z}{A_i+z} \Bigr ) + Ti \cdot z = -  Az \int \frac{-\zeta+z}{-\zeta-z} d\mu_\zeta + T i \cdot z
$$
as desired.
\end{proof}
Theorem \ref{algebraic-convergence} now follows in the case when $a \to 1$ since for radial degenerations, the rotational factor $T(f_{\mathbf{a}}) \to 0$.

\subsection{Radial degenerations with $a \to e(p/q)$.} 
As noted above, for a radial degeneration with $a \to e(p/q)$, we consider the limiting vector field of
$f_{\mathbf{a}}^{\circ q}$. In view of  Lemma \ref{vfc-lemma}, to show that $f_{\mathbf{a}}^{\circ q}$ converges to a vector field $\kappa$ whose driving measure gives mass $\rho_i/q$ to each point $e(\theta_i + j/q)$,
 it suffices to analyze the zero set of $f_{\mathbf{a}}^{\circ q}$.

 Let us first consider the case of a generic radial degeneration, i.e.~when the points $e(\theta_i + j/q)$ with $1 \le i \le d-1$ and $0 \le j \le q-1$ are all different. The zero set of $f_{\mathbf{a}}^{\circ q}$ consists of the zeros of
$f_{\mathbf{a}}$ and their $1,2,\dots,(q-1)$-fold pre-images.  We omit the trivial zero at the origin and split the remaining zeros of $f_{\mathbf{a}}^{\circ q}$ into two groups: the {\em dominant zeros} and {\em subordinate zeros}.
The dominant zeros are the zeros
 $z_i = z_{i,0} $ of $f_{\mathbf{a}}(z)$ and their shadows  $z_{i,j}$ near $ e(-j p/q)z_i$. We will refer to all other zeros as the subordinate zeros.
From formula (\ref{eq:green-identity}), it follows that the heights of the subordinate zeros are insignificant compared to the heights of the dominant zeros.
 Thus, only the dominant zeros contribute to the limiting vector field.

Let us now consider the general case. For a point $z \in \mathbb{D}$,
call $w$ a {\em dominant pre-image} of $z$ under $f_{\mathbf{a}}$ if it is located near $e(-p/q)\hat{z}$, i.e.~if $|\hat{w} - e(-p/q)\hat{z}| \le \epsilon$.
Otherwise, we say that $w$ is a {\em subordinate pre-image}. We define a {\em dominant zero} of $f_{\mathbf{a}}^{\circ q}$ to be a point $z \in \mathbb{D}$ which is the $j$-fold dominant pre-image of some $z_i$, with $0 \le j \le q-1$.
 To show that the driving measure $\mu$ has the desired expression, it suffices to show that the subordinate zeros have negligible height. 
 We prove this in two lemmas:

\begin{lemma}
Suppose $f_{\mathbf{a}}(z) = z \prod \frac{z+a_i}{1+\overline{a_i}z} \in \mathcal B_d$ with  $|a| = |f_{\mathbf{a}}'(0)| \approx 1$. 
For a point $z \in B \bigl (0, 1 - K \sqrt{1-|a|} \bigr )$ with $K \ge 1$, the hyperbolic distance $d_{\mathbb{D}}(f_{\mathbf{a}}(z), a z) < C/K^2$.
\end{lemma}

\begin{proof}
The map $z \to \frac{z+a_i}{1+\overline{a_i}z}$ takes the ball $B \bigl (0, 1-K\sqrt{1-|a|} \bigr )$ inside the ball
$$B \biggl ( a_i, \ (C_1/K) \cdot \sqrt{1-|a|} \cdot \frac{1-|a_i|}{1-|a|} \biggr ).$$
Multiplying over $i=1,2, \dots, d-1$, we see that $\prod \frac{z+a_i}{1+\overline{a_i}z} \in B\Bigl (a, (C_2/K) \sqrt{1-|a|} \Bigr )$ which shows that
 $|f_{\mathbf{a}}(z) - az| \le (C_2/K) \sqrt{1-|a|}$ as desired. 
 \end{proof}

\begin{lemma}
Suppose $w$ satisfies $f(w) = z$ yet  $|\hat w - e(-p/q)\hat{z} | \ge \epsilon$. Then, 
\begin{equation}
\frac{G(w)}{G(z)} = O_\epsilon(1-|a|).
\end{equation}
\end{lemma}

\begin{proof}
Consider the hyperbolic geodesic $[0, w]$. Set $w_0 := (1 - K \sqrt{1-|a|}) \cdot w$ and write $[0,w] = [0,w_0] \cup [w_0, w]$.
Since $f$ restricted to the first segment $[0, w_0]$ is nearly a rotation by $e(p/q)$, we see that
 during the first part of the journey from $f(0)=0$ to $f(w)=z$ along $f([0,w])$ we have moved in the wrong direction, i.e.~ 
\begin{eqnarray*}
d_{\mathbb{D}}(f(w_0), f(w)) & = &  d_{\mathbb{D}}(f(w_0),0) +d_{\mathbb{D}}(0,f(w)) - O_\epsilon(1), \\
& = &  d_{\mathbb{D}}(w_0,0) +d_{\mathbb{D}}(0,z) - O_\epsilon(1),
\end{eqnarray*}
as Lemma \ref{mob-move} shows.
By the Schwarz lemma,  $d_{\mathbb{D}}(w_0, w) \ge d_{\mathbb{D}}(f(w_0), f(w))$. Therefore, we must have
$d_{\mathbb{D}}(0,w) = d_{\mathbb{D}}(0,w_0) +  d_{\mathbb{D}}(w_0,w)  \ge 2d_{\mathbb{D}}(0,w_0) +  d_{\mathbb{D}}(0,z) - O_\epsilon(1)$
to make up for this detour.
\end{proof}

\subsection{Asymptotic semigroups}
By an {\em asymptotic semigroup} on a domain $\Omega$, written in additive notation, we mean a family of holomorphic maps $\{f_t\}_{t \ge 0}: \Omega_t \to \mathbb{C}$, with
$\Omega_t \to \Omega$ in the Carath\'eodory topology,
satisfying 
\begin{equation}
\label{eq:smg1}
f_t(z) = z + t \cdot \kappa(z) + O_K(t^2), \qquad 0 < t < t_K,
\end{equation}
for some holomorphic vector field $\kappa$. 
(To convert to multiplicative notation, write $f_t = g^{\eta(t)} $ with $\eta(t) = e^{-t}$.)
The notation $O_K$ denotes that the implicit constant is uniform on compact subsets of $\Omega$.
The condition (\ref{eq:smg1}) implies that
\begin{equation}
\label{eq:smg2}
\Bigl |f_{t}(z) - f_{t_1}(f_{t_2}(z)) \Bigr | \le O_K(t^2), \qquad t = t_1 + t_2.
 \end{equation} 
We now show that the short term iteration of $f_t$ approximates the flow of $\kappa$\,:
\begin{theorem}
\label{stronger-dc}
Given a ball $B(z_0, R)$ compactly contained in $\Omega$, one can find a $t_0 > 0$, 
so that for $z \in B(z_0, R)$ and $t < t_0$, the limit
\begin{equation}
\label{eq:asymsg}
g_t(z) := \lim_{\|\mathscr P\| \to 0} f_{\mathscr P}(z) := \lim_{\max t_i \to 0}  f_{t_n}(f_{t_{n-1}}( \cdots (f_{t_1}(z)) \cdots ))
\end{equation}
over all possible partitions $t_1 + t_2 + \dots + t_n = t$ exists, and defines a holomorphic function on $B(z_0, R)$.
\end{theorem}
Above, the notation $f_{\mathscr P}(z)$ denotes the expression 
$f_{t_n}(f_{t_{n-1}}( \cdots (f_{t_1}(z)) \cdots ))$ where
$\mathscr P$ is a partition of the interval $[0, t]$ by the points $\tau_k =\sum_{j \le k} t_j$.
The existence of the limit in (\ref{eq:asymsg}) implies that $\{g_t\}$ satisfies $g_s \circ g_t = g_{s+t}$ as long as $g_{s+t}$ is well-defined. Clearly, the vector field $\kappa$ is the generator of 
 the semigroup $\{g_t\}$.

\begin{proof}
Choose two balls $B(z_0, R_2) \supset B(z_0, R_1) \supset B(z_0,R)$ compactly contained in $\Omega$. We then choose $t_0^* > 0$ so that
$
|f_t(z) - z| \le C_{R_1}t$ for $ z \in B(z_0, R_1)$ and  $t \le t_0^*$.
By taking $t_0 := \min \bigl ( \frac{R_1 - R}{C_{R_1}}, t_0^* \bigr )$, we guarantee that all computations
 $f_{t_k} \circ \cdots \circ f_{t_1}(z)$
  with $z \in B(z_0,R)$ and $t_1 + t_2 + \dots + t_k < t_0$ stay within  $B(z_0,R_1)$.
Therefore, for $0 < t< t_0$, a finite partition $\mathscr P$ of the interval $[0,t]$ defines 
a holomorphic function $f_{\mathscr P}(z)$ on $B(z_0,R)$.

 To prove the convergence of (\ref{eq:asymsg}), it suffices to show that if $\mathscr Q$ is a refinement of $\mathscr P$, then
$|f_{\mathscr P}(z) - f_{\mathscr Q}(z)| \le Ct \|\mathscr P\|$. Actually, it suffices to show that for an arbitrary partition $\mathscr P$ of $[0,t]$,
one has $|f_{\mathscr P}(z) - f_t(z)| \le C t^2$. For this purpose, we introduce some bookkeeping: in view of (\ref{eq:smg2}), we say that the cost of splitting an interval of length $T$ into two intervals is $C \cdot T^2$.
Using the greedy algorithm, its not hard to show that the minimal cost of any partition of $[0,t]$ is at most $O(t^2)$.

To combine the ``costs,'' we use the fact that on $B(z_0, R_1)$, the hyperbolic metric $\rho_{B(z_0,R_2)}$ is comparable to the Euclidean metric.
Therefore, by the Schwarz lemma,
\begin{eqnarray*}
d_{B(z_0,R_2)} \Bigl (z,  f_{t_n} \circ \cdots \circ f_{t_1}(z) \Bigr ) & \le & \sum_{k=1}^n d_{B(z_0,R_2)} \Bigl (f_{t_k} \circ \cdots \circ f_{t_1}(z), \ f_{t_{k-1}} \circ \cdots \circ f_{t_1}(z) \Bigr) \\
& \lesssim & \sum_{k=1}^n \, \Bigl | f_{t_k} \circ \cdots \circ f_{t_1}(z) -  f_{t_{k-1}} \circ \cdots \circ f_{t_1}(z) \Bigr |
\end{eqnarray*}
which gives the claim. 
\end{proof}
Clearly, Theorem \ref{dynamical-convergence} is a special case of Theorem \ref{stronger-dc}, where $\Omega = \mathbb{C} \setminus P(\kappa)$ is the complement of the set of poles of $\kappa$. By the Schwarz lemma, inside the unit disk, $g^t(z)$ can be defined for {\em all} time, whereas on the unit circle, one
 can only define $g^t(z)$ until one hits a  pole of $\kappa$.

%% file: asymptotics.tex
\section{Asymptotics of the Weil-Petersson metric}

In this section, we prove Theorem \ref{fine-geometry} which says that as $a \to e(p/q)$ radially in $\mathcal B_2$, the ratio 
$\omega_B/\rho_{1/4}$ tends to a constant, depending on the denominator $q$. In the language of half-optimal Beltrami coefficients, we need to show that for a fixed $\lambda$ with $|\lambda| = 1$, $\|\mu_\lambda \cdot \chi_{\mathcal G}\|_{\WP}
\sim C'_{q}(1-|a|)^{1/4}$.
As noted in the introduction, the key observation is the convergence of Blaschke products to vector fields. The convergence of the linearizing coordinates
(the corollary to Theorem \ref{dynamical-convergence}) gives:

\begin{theorem}
\label{flower-convergence}
As $a \to e(p/q)$ radially, 
\begin{itemize}
\item[{\em (i)}] The flowers $\mathcal F_{p/q}(f_a) \to \mathcal F_{p/q}(\kappa_{q})$ in the Hausdorff topology.
\item[{\em (ii)}] The optimal Beltrami coefficients $\mu_\lambda(f_a) = \varphi_a^*(\lambda \cdot z/\overline{z} \cdot d\overline{z}/dz)$ converge uniformly to 
$\varphi_{\kappa_{q}}^*(\lambda \cdot z/\overline{z} \cdot d\overline{z}/dz)$ on compact subsets of $\mathbb{D} \setminus \{0\}$.
\end{itemize}
\end{theorem}
 
Together with Lemma \ref{strong-linearization-at-rfp}, which controls the shapes of flowers near the unit circle, Theorem \ref{flower-convergence}
 implies the quasi-geodesic property:

\begin{lemma}[Quasi-geodesic property]
\label{bg-lemma2}
As $a \to e(p/q)$ radially, each petal $\mathcal P_{\xi_i(f_a)}(f_a)$ lies within a bounded distance of the geodesic ray $[0, \xi_i(f_a)]$. Alternatively, the flower ${\mathcal F}(f_a)$ lies within a bounded neighbourhood of the hyperbolic convex hull of the origin and the ends $\xi_i(f_a)$.
\end{lemma}
 
For convenience of the reader, we give an alternative proof of the strong linearization property of Lemma \ref{strong-linearization-at-rfp}
using the existence of the limiting vector field $\kappa = \kappa_q$.
Observe that as $a \to e(p/q)$ radially, the $p/q$-cycle $\langle \xi_1(f_a), \xi_2(f_a), \dots, \xi_q(f_a)\rangle$ converges to the set of sources $\langle \xi_1, \xi_2, \dots, \xi_q \rangle$ of $\kappa$.
Note that $\kappa'(\xi_i) > 0$ is a positive real number since $\kappa$ is tangent to the unit circle.

Choose a ball $B(\xi_i, R')$ on which $\Bigl | \frac{\kappa'(z)}{\kappa'(\xi_i)} -1 \Bigr | < 1/10$.
From the uniform convergence $\frac{f_a^{\circ q}(z) - z}{1-|a|^q} \to \kappa(z)$ on  $B(\xi_i, R')$, it follows that on $B(\xi_i, R'/2)$,  we have
 $$
 (f_a^{\circ q})'(z) \approx 1 + \kappa'(z) (1-|a|^q),
 \qquad
\bigl |  (f_a^{\circ q})''(z) \bigr | \le C_1(1-|a|^q),$$
  where $C_1 = \max_{z \in B(\xi_i, R')} |\kappa''(z)|+\epsilon$.
Therefore, for $a$ sufficiently close to $e(p/q)$, we have $|\xi_i(f_a)-\xi_i| < R'/4$ and
   $|(f_a^{\circ q})''(z)| \le C_2 \bigl |(f_a^{\circ q})'(\xi_1(f_a))-1 \bigr |$ for $z \in B(\xi, R'/2)$,
   which implies that $(f_a^{\circ q})^{-1}$ maps $B(\xi_i(f_a), R)$ into itself, with $R = \min ( R'/4, 1/(2C_2))$.
 
 \medskip
 
 \noindent \textbf{Flower counting hypothesis.} 
From Theorem \ref{petal-separation-2},  we know that  the immediate pre-flower is
 approximately the image of the flower under the M\"obius involution about the critical point, while the pre-flowers are nearly-affine copies of the immediate pre-flower.
Therefore, 
 the pre-flowers of all maps $$\bigl\{ f_a, \ a \in e(p/q) \cdot [1-\epsilon, 1) \bigr\}$$  also have nearly the same shape up to affine scaling.
Let $n(r, f_a)$ denote the number of repeated pre-images of $-a$ that lie in $B(0,r)$. 
By Theorem  \ref{renewal-theory2}, 
 \begin{equation}
 \label{eq:renewal-convergence}
\mathfrak c(f_a) \ := \ \lim_{r \to 1} \frac{n(r, f_a)}{1-r} \ = \  \frac{\log (1/|a|)}{h(f_a, m)} \ \sim \ \sqrt{1-|a|}, \qquad \text{as }|a|\to 1.
 \end{equation}

  The quantity $n(r, f_a)$ roughly counts the number of pre-flowers that intersect $S_r$.
By renewal theory, for $r$ close to 1, the circle 
 $S_r$ intersects pre-flowers at ``hyperbolically random'' locations. Therefore, is reasonable to hypothesize that as $a \to e(p/q)$ radially, 
 $\frac{1}{\mathfrak c(f_a)} \cdot \|\mu_{\lambda} \cdot \chi_{\mathcal G}\|^2_{\WP}$ converges to a constant. 
To justify this, we must show three things:
\begin{enumerate} 
\item The contributions of the pre-flowers are more or less independent.
\item All pre-flowers of the same size contribute roughly the same amount.
\item Most of the integral $\mathcal I_r[\mu] = \frac{1}{2\pi} \int_{|z|=r} |v_{\mu^+}'''/\rho^2|^2 d\theta$ comes from pre-flowers whose size is comparable to $1-r$.
\end{enumerate}

%%%%%%%%%%%%%%%%%%%%%%%%%
%
%                    Decay of Correlations
%
%%%%%%%%%%%%%%%%%%%%%%%%%

\subsection{Decay of correlations}

\label{sec:decay-of-correlations}

In this section, we use ``flower'' to mean either a flower or a pre-flower.
Write the half-optimal coefficient as $\mu_{\half} = \sum_{\mathcal F} \mu_{\mathcal F}$ with $\mu_{\mathcal F}$ supported on $\mathcal F$.
Set
$$
v_{\mathcal F}'''(z) = - \frac{6}{\pi}  \int_{\mathcal F^+} \frac{\mu^+(\zeta)}{(\zeta-z)^4} \cdot |d\zeta|^2.
$$
Then $v'''(z) = \sum_{\mathcal F} v_{\mathcal F}'''(z).$  We wish to show that the integral (\ref{eq:integral-average0}) is proportional to the flower count.
The main difficulty is that (\ref{eq:integral-average0}) features the $L^2$ norm so we have ``correlations''
$ \sum_{\mathcal F_1 \ne \mathcal F_2} \int_{|z|=r} \frac{v_{\mathcal F_1}'''}{\rho^2} \cdot \frac{\overline{v_{\mathcal F_2}'''}}{\rho^2} \, d\theta$. We 
now show that these correlations are insignificant compared to the main term
$ \sum_{\mathcal F} \int_{|z|=r} \bigl | \frac{v_{\mathcal F}'''}{\rho^2} \bigr |^2 d\theta$.

For a point $z \in \mathbb{D}$, let $\mathcal F_z$ be the flower
which is closest to $z$ in the hyperbolic metric  (in case of a tie, we pick $\mathcal F_z$ arbitrarily) and $\mathcal R_z$ be the union of all the other flowers. % 
The integral (\ref{eq:integral-average0}) splits into four parts:$$
 \int_{|z|=r} \ \Biggl \{ \, \biggl |\frac{v'''_{\mathcal F_z}(z)}{\rho(z)^2}\biggr |^2 
\,+\, \frac{v_{\mathcal F_z}'''(z)}{\rho(z)^2} \cdot \frac{\overline{v_{\mathcal R_z}'''(z)}}{\rho(z)^2}
\,+\,\frac{v_{\mathcal R_z}'''(z)}{\rho(z)^2} \cdot \frac{\overline{v_{\mathcal F_z}'''(z)}}{\rho(z)^2}  
\,+\,\biggl |\frac{v'''_{\mathcal R_z}(z)}{\rho(z)^2}\biggr |^2 \, \Biggr \} \ d\theta
$$
By the lower bound established in Section \ref{chap:coarse-geometry}, the first term is bounded below by the flower count which decays roughly like $\asymp \sqrt{1-|a|}$,
while each of the other three terms contribute on the order of $O(1-|a|)$, and so are negligible.
Take for instance the second term. By the triangle inequality, for any $z \in \mathbb{D}$,
$$
\frac{v_{\mathcal F_z}'''(z)}{\rho(z)^2} \cdot \frac{\overline{v_{\mathcal R_z}'''(z)}}{\rho(z)^2} \, \lesssim \,
e^{-d_{\mathbb{D}}(z, \mathcal F_z)} \cdot e^{-d_{\mathbb{D}}(z, \mathcal R_z)} \, \le \, e^{-d_{\mathbb{D}}(\mathcal F_z, \mathcal R_z)}
$$
which is bounded by $e^{-d_{\mathbb{D}}(0,a)} \asymp (1-|a|)$. The estimate for the other two error terms is similar.

%%%%%%%%%%%%%%%%%%%%%%%%%
%
%                    Convergence of Beltrami coefficients
%
%%%%%%%%%%%%%%%%%%%%%%%%%

\subsection{Convergence of Beltrami coefficients}

For a Blaschke product $f_a \in \mathcal B_2$ with $a \approx e(p/q)$, we define an {\em idealized garden} $\mathcal G^{\ideal}(f_a)$ where the pre-flowers
have the model shape. First, we define the idealized flower $\mathcal F^{\ideal}(f_a) := \mathcal F(g^\eta)$ to be the flower  of the limiting vector field. We then define the idealized immediate pre-flower $\mathcal F^{\ideal}_{*}(f_a)$ 
as the image of $\mathcal F(g^\eta)$ under the M\"obius involution about
$c(f_a)$. Finally, for a pre-flower $\mathcal F_z(f_a)$, we define its idealized version $\mathcal F^{\ideal}_z(f_a)$ to be the affine copy of $\mathcal F^{\ideal}_*(f_a)$, which has the same $A$-point $z$.

The {\em idealized half-optimal Beltrami coefficient} $\mu_{\ideal}$ is defined similarly: on $\mathcal F^{\ideal}(f_a)$, we let $\mu_{\ideal} \cdot \chi_{{\mathcal F}^{\ideal}}$ be the half-optimal Beltrami coefficient for the limiting vector field; while on the pre-flowers, we define $\mu_{\ideal} \cdot \chi_{\mathcal F^{\ideal}_z}$ by scaling $\mu_{\ideal} \cdot \chi_{\mathcal F^{\ideal}}$ appropriately.
Our current objective is to show that the integral averages for the model coefficient and the half-optimal coefficient $\mu_{\half}$ are approximately the same:

\begin{lemma}
\label{comparing-id-half}
As $a \to e(p/q)$ radially, 
$$
\mathcal I[{\mu_{\ideal}}] - \mathcal I[\mu_{\half}] \ =  \ o\bigl ( \sqrt{1-|a|} \bigr ).
$$
\end{lemma}

There are two sources of error. First, the pre-flowers don't quite match up with their idealized counterparts. Secondly, since the linearizing maps $\varphi_{a}$ and $\varphi_\kappa$ are slightly different, the Beltrami coefficients $\mu_{\half}$ and $\mu_{\ideal}$ themselves are slightly different.
To prove Lemma \ref{comparing-id-half}, we we split $\mathcal F^\alpha$ into three parts:
  \begin{eqnarray*}
 A^{\alpha, \delta} & = & \mathcal F^\alpha \cap \{ z : |z| < \delta \}, \\
 B^{\alpha, \delta} & = & \mathcal F^\alpha \cap  \{ z : \delta < |z| < 1- \delta \}, \\
 C^{\alpha, \delta} & = & \mathcal F^\alpha \cap  \{ z : 1 - \delta < |z| \}.
 \end{eqnarray*}
Taking pre-images, we obtain ABC decompositions of pre-flowers.
As usual, we take $\alpha = 1/2$ unless specified otherwise. We typically omit the parameter $\delta > 0$ from the notation.

\medskip

\noindent \textbf{Estimating the symmetric difference.} 
For any $\epsilon > 0$, if $a$ is sufficiently close to $e(p/q)$, the symmetric difference of the flower $\mathcal F$ and its idealized version $\mathcal F^{\ideal}$ is contained in the set
\begin{equation}
\Delta (\mathcal F) := A(f_{a}) \cup A(g^\eta) \ \cup \ 
\Bigl ( B^{1/2+\epsilon}(g^\eta) \setminus B^{1/2-\epsilon}(g^\eta) \Bigr )
\ \cup \ C (f_{a}) \cup C (g^\eta).
\end{equation}
Taking pre-images, we obtain sets $\Delta(\mathcal F_z)$ which contain symmetric differences of pre-flowers and their idealized versions. Let $\Delta = \bigcup \Delta(\mathcal F_z)$.
Observe that the proportion
\begin{equation}
\frac{\mathcal A(\Delta)}{\mathcal A(\mathcal G)} = 
 \frac{\limsup_{r \to 1} | \Delta \cap S_r|}{\lim_{r \to 1} | \mathcal G \cap S_r|}
 \end{equation}
can be made arbitrarily small by choosing $\delta, \epsilon > 0$ small.  
By Theorem \ref{wpbounds}, it follows that $\mathcal I[\mu_{\half} \cdot \chi_{\Delta}] = o\bigl ( \sqrt{1-|a|} \bigr )$
and  $\mathcal I[\mu_{\ideal} \cdot \chi_{\Delta}] = o\bigl ( \sqrt{1-|a|} \bigr ).$

\medskip

\noindent \textbf{Estimating the difference between Beltrami Coefficients.} From the convergence of the linearizing maps $\varphi_a \to \varphi_\kappa$,
 when
 $a \approx e(p/q)$, $|\mu_{\half} - \mu_{\ideal}|$ is arbitrarily small on $B^{1/2+\epsilon}(g^\eta)$. By Koebe's distortion theorem,
the same estimate holds on pre-flowers.
Therefore, the difference $|\mu_{\half} \cdot \chi_{\Delta^c} - \mu_{\ideal} \cdot  \chi_{\Delta^c} |$ is small in $L^\infty$ sense. Theorem \ref{wpbounds} implies
$\mathcal I[\mu_{\half} \cdot \chi_{\Delta^c} - \mu_{\ideal} \cdot \chi_{\Delta^c}] = o\bigl ( \sqrt{1-|a|} \bigr )$. 

\medskip

\noindent \textbf{Combining the errors.} 
In  \cite{hedenmalm-atvar}, Hedenmalm observed that  Minkowski's inequality implies that $\mathcal I[\mu]$ behaves like a semi-norm:
\begin{equation}
\label{eq:minkowski}
\Bigl |\sqrt{\mathcal I[\mu]} - \sqrt{\mathcal I[\nu]} \Bigr | \le \sqrt{\mathcal I[\mu-\nu]},
\end{equation}
where in the definition of $\mathcal I[\mu]$, we use $\limsup$ if necessary (if $\mu$ is not invariant).
Returning to the task at hand, since $\mathcal I[{\mu_{\ideal}}]$ and $\mathcal I[\mu_{\half}]$ are
$\gtrsim \sqrt{1-|a|}$ and the errors are $o\bigl ( \sqrt{1-|a|} \bigr )$,
Minkowski's inequality (\ref{eq:minkowski}) completes the proof of Lemma  \ref{comparing-id-half}.
  
\subsection{Flowers: large and small}
Finally, we must show that most of the integral average $ \int_{S_r} |v'''/\rho^2|^2 d\theta$ comes from flowers 
whose size is $\asymp (1-r)$.
 In view of (\ref{eq:renewal-convergence}), given $\epsilon > 0$, there exists $0 < r_{\mix} =  r_{\mix}(f_a) < 1$
such that
$
\frac{n(r, f_a)}{1-r} \approx_\epsilon \mathfrak c(f_a)$ for 
$r \in (r_{\mix},1).$
For $r \in (r_{\mix},1)$, we decompose
\begin{equation}
\mu_{\half} = \mu_{\sma} + \mu_{\med} + \mu_{\lar} + \mu_{\hug}
\end{equation}
 where 
 \begin{equation*}
 \begin{cases} 
\text{ small flowers} & \text{ have size $s \le (1-r)/k$,} \\ 
\text{ medium flowers} & \text{ have size $(1-r)/k \le s \le k(1-r)$,} \\
\text{ large flowers} & \text{ have size $k(1-r) \le s \le 1 - r_{\mix}$,} \\
\text{ huge flowers} & \text{ have size $s \ge 1 - r_{\mix}$.}
   \end{cases}
\end{equation*}
(The {\em size} of a flower $\mathcal F_z$ may be defined as either its diameter or as $1 - |z|$. The two quantities are comparable along radial degenerations.)

From the lower bound, we know that $$\mathcal I[\mu_{\med}] \asymp \mathfrak c(f_a).$$ We claim that
 if the ``tolerance'' $k > 1$ is large, then 
\begin{equation}
\label{eq:medium-dominance}
\bigl |  \mathcal I[\mu_{\half}] - \mathcal I[\mu_{\med}] \bigr | \mathfrak \ \lesssim \  \mathfrak c(f_a)/\sqrt{k}.
\end{equation}
Since there are only finitely many huge flowers and they satisfy the quasi-geodesic property, 
$|\mathcal G_{\hug} \cap S_r| \to 0$ as $r \to 1$. By counting the number of large flowers, we can conclude 
$|\mathcal G_{\lar} \cap S_r| \lesssim \mathfrak c(f_a)/k$ as well.
Therefore by Theorem \ref{wpbounds}, $$\mathcal I[\mu_{\hug}] + \mathcal I[\mu_{\lar}]  \lesssim \mathfrak c(f_a)/k$$
for $r$ close to $1$.
It remains to estimate the contribution of the small flowers.
This can be done by combining the Fubini argument from the proof of Theorem \ref{wpbounds} with part $(b)$ of Theorem \ref{qbounds}. This leads to the estimate
$$\int_{|z|=r} |v_{\sma}'''/\rho^2|^2 d\theta \ \lesssim \ \frac{\|\mu\|_\infty}{k} \cdot \ \limsup_{R \to 1^+}
 \ \frac{1}{2\pi} \bigl |\supp \mu^+ \cap S_R  \bigr | \ \lesssim \ \mathfrak c(f_a)/k.$$
Using Minkowski's inequality (\ref{eq:minkowski}) as before proves (\ref{eq:medium-dominance}). 
This completes the proof of Theorem \ref{fine-geometry}.

%% file: main.bbl
% ----------------------------------------------------------------

%%% Local Variables: 
%%% mode: latex
%%% TeX-master: "main"
%%% End: 

% ----------------------------------------------------------------